\documentclass[a4paper,11pt,reqno]{amsart}
\usepackage{a4wide}
\usepackage[utf8]{inputenc}
\usepackage{amsmath}
\usepackage{amssymb}
\usepackage{amsthm}
\usepackage{hyperref}
\usepackage[arrow, matrix, curve]{xy}
\usepackage{color}
\usepackage{epic}
\usepackage{pict2e}
\usepackage{slashbox}
\usepackage{array}
\usepackage{makecell}
\newcolumntype{x}[1]{>{\centering\arraybackslash}p{#1}}

\usepackage{tikz}
\newcommand\diagonal[4]{%
  \multicolumn{1}{p{#2}|}{\hskip-\tabcolsep
  $\vcenter{\begin{tikzpicture}[baseline=0,anchor=south west,inner sep=#1]
  \path[use as bounding box] (0,0) rectangle (#2+2\tabcolsep,\baselineskip);
  \node[minimum width={#2+2\tabcolsep},minimum height=\baselineskip+\extrarowheight] (box) {};
  \draw (box.north west) -- (box.south east);
  \node[anchor=south west] at (box.south west) {#3};
  \node[anchor=north east] at (box.north east) {#4};
 \end{tikzpicture}}$\hskip-\tabcolsep}}
\makeatletter

\def\thefigure{\thesection.\@arabic\c@figure}
\let\c@figure=\c@equation

\def\diagram{
  \mathsurround\z@ \everymath{}
  \par
  \ifdim\lastskip=\textfloatsep
    \removelastskip
    \addvspace{\floatsep}
  \else
    \addvspace{\textfloatsep}
  \fi
  \def\@captype{figure}
  \vbox\bgroup
    \hsize\columnwidth \@parboxrestore
    \centering}
\def\enddiagram{\par\vskip\z@\egroup\addvspace{
\textfloatsep}}

\newcommand\subcaption[1]{\par\medskip\centerline{#1}
\medskip}

\newlength{\legendindent}
\newenvironment{legend}{
  \mathsurround2\p@
  \par\smallskip
  \leftskip0pt \rightskip\leftskip \parindent\leftskip
  \parfillskip0pt plus1fil\relax
  \def\item[##1]{\noindent\unskip\unpenalty\unkern
    \setbox\z@\lastbox
    \ifdim\wd\z@=\legendindent \else\ifvoid\z@ 
\else\box\z@\fi\par\fi
    \noindent 
    \hangindent\legendindent
    \hbox to\legendindent{\mathsurround\z@\everymath{}
        \unhbox\z@\hfil\ignorespaces##1\kern9.9pt}
   \ignorespaces
  }
  \noindent\hbox to\legendindent{Legend:\hfil}\ignorespaces
}{
  \relax
}

\DeclareMathOperator{\GL}{GL}
\DeclareMathOperator{\SL}{SL}
\DeclareMathOperator{\Sp}{Sp}
\DeclareMathOperator{\Spin}{Spin}
\DeclareMathOperator{\upO}{O}
\DeclareMathOperator{\upU}{U}
\newcommand{\frakg}{\mathfrak{g}}
\newcommand{\frakk}{\mathfrak{k}}
\newcommand{\frakm}{\mathfrak{m}}
\newcommand{\fraka}{\mathfrak{a}}
\newcommand{\frakn}{\mathfrak{n}}
\newcommand{\fraks}{\mathfrak{s}}
\newcommand{\frakq}{\mathfrak{q}}

\renewcommand{\AA}{\mathbb{A}}

\newcommand{\CC}{\mathbb{C}}
\newcommand{\FF}{\mathbb{F}}
\newcommand{\NN}{\mathbb{N}}

\newcommand{\RR}{\mathbb{R}}
\newcommand{\ZZ}{\mathbb{Z}}

\newcommand{\calE}{\mathcal{E}}
\newcommand{\calF}{\mathcal{F}}
\newcommand{\calH}{\mathcal{H}}
\newcommand{\calL}{\mathcal{L}}
\newcommand{\calP}{\mathcal{P}}
\newcommand{\calR}{\mathcal{R}}

\newcommand{\calT}{\mathcal{T}}
\newcommand{\calU}{\mathcal{U}}
\newcommand{\calV}{\mathcal{V}}
\newcommand{\calW}{\mathcal{W}}

\newcommand{\Leven}{L_{\textup{even}}}
\newcommand{\Lodd}{L_{\textup{odd}}}
\newcommand{\0}{\textbf{0}}
\renewcommand{\1}{\textbf{1}}
\DeclareMathOperator{\Ind}{Ind}
\DeclareMathOperator{\tr}{tr}
\DeclareMathOperator{\ad}{ad}
\DeclareMathOperator{\Ad}{Ad}
\DeclareMathOperator{\Cas}{Cas}
\DeclareMathOperator{\End}{End}
\DeclareMathOperator{\Hom}{Hom}
\DeclareMathOperator{\proj}{proj}
\DeclareMathOperator{\rest}{rest}

\DeclareMathOperator{\sgn}{sgn}

\DeclareMathOperator{\diag}{diag}
\DeclareMathOperator{\blank}{\cdot}

\newcommand{\hol}{\textup{hol}}
\newcommand{\ahol}{\textup{ahol}}
\newcommand{\HC}{\textup{\,HC}}

\theoremstyle{plain}
\newtheorem{theorem}{Theorem}[section]
\newtheorem{proposition}[theorem]{Proposition}
\newtheorem{lemma}[theorem]{Lemma}
\newtheorem{corollary}[theorem]{Corollary}

\newtheorem{thmalph}{Theorem}

\theoremstyle{definition}

\newtheorem{example}[theorem]{Example}

\newtheorem{remark}[theorem]{Remark}

\numberwithin{equation}{section}

\title[The compact picture of symmetry breaking operators]{The compact picture of symmetry breaking operators for rank one orthogonal and unitary groups}
\author{Jan Frahm}
\author{Bent {\O}rsted}

\address{Department of Mathematics, Aarhus University, Ny Munkegade 118, 8000 Aarhus C, Denmark}
\email{frahm@math.au.dk}

\address{Department of Mathematics, Aarhus University, Ny Munkegade 118, 8000 Aarhus C, Denmark}
\email{orsted@math.au.dk}

\date{}

\begin{document}

\begin{abstract}
We present a method to calculate intertwining operators between the underlying Harish-Chandra modules of degenerate principal series representations of a reductive Lie group $G$ and a reductive subgroup $G'$, and between their composition factors. Our method describes the restriction of these operators to the $K'$-isotypic components, $K'\subseteq G'$ a maximal compact subgroup, and reduces the representation theoretic problem to an infinite system of scalar equations of a combinatorial nature. For rank one orthogonal and unitary groups and spherical principal series representations we calculate these relations explicitly and use them to classify intertwining operators. We further show that in these cases automatic continuity holds, i.e. every intertwiner between the Harish-Chandra modules extends to an intertwiner between the Casselman--Wallach completions, verifying a conjecture by Kobayashi. Altogether, this establishes the compact picture of the recently studied symmetry breaking operators for orthogonal groups by Kobayashi--Speh, gives new proofs of their main results and extends them to unitary groups.
\end{abstract}



\maketitle

\section{Introduction}

Representation theory of reductive Lie groups consists to a large extent in the study of the structure of standard families of representations, for example principal series representations. Here intertwining operators, such as the classical Knapp--Stein operators, play an important role, and they also provide important examples of integral kernel operators appearing in classical harmonic analysis. Recently similar operators have been introduced by Kobayashi~\cite{Kob15} in connection with branching laws, i.e. the study of how representations behave when restricted to a closed subgroup of the original group (see also \cite{KS13,MOO13}). Again these are integral kernel operators, now intertwining with respect to the subgroup, and they appear to be very natural objects, not only for the problem of restricting representations (see \cite{MO12}), but also for questions in classical harmonic analysis and automorphic forms (see \cite{MO13,MOZ15}).

In this paper we shall give an alternative approach to this new class of symmetry breaking operators, namely one based on the Harish-Chandra module, i.e. the $K$-finite vectors in the representation, in analogy with the idea of spectrum generating operators \cite{BOO96}. This gives new proofs of the main results of \cite{KS13}, and generalizes these results to unitary groups. Moreover, our more algebraic framework provides an alternative proof of the discrete spectrum in certain unitary representations.

The approach is quite general and discussed in the first part of the paper, while in the second part we carry out all details for the real conformal case and the CR case.

\subsection{Symmetry breaking operators}\label{sec:SymBreakingOperators} Let $G$ be a reductive Lie group with compact center and $G'\subseteq G$ a reductive subgroup also with compact center. For irreducible smooth representations $\pi$ of $G$ and $\tau$ of $G'$ the space
$$ \Hom_{G'}(\pi|_{G'},\tau) $$
of continuous $G'$-intertwining operators between $\pi$ and $\tau$ and its dimension $m(\pi,\tau)$ have received considerable attention recently, in particular in connection with multiplicity-one statements asserting that $m(\pi,\tau)\leq1$ for certain pairs $(G,G')$ of classical groups such as $(\GL(n,\RR),\GL(n-1,\RR))$, $(\upO(p,q),\upO(p,q-1))$ or $(\upU(p,q),\upU(p,q-1))$, see \cite{SZ12} and references therein. A more refined problem is to determine whether for given representations $\pi$ and $\tau$ there exist non-trivial $G'$-intertwining operators $\pi|_{G'}\to\tau$, also called \textit{symmetry breaking operators} by Kobayashi~\cite{Kob15}, and to classify them. For the pair $(G,G')=(\upO(1,n),\upO(1,n-1))$ this question was completely answered by Kobayashi--Speh~\cite{KS13} in the case where $\pi$ and $\tau$ are spherical principal series representations, and in joint work with Y. Oshima we generalized in \cite{MOO13} their construction of symmetry breaking operators to a large class of symmetric pairs.

Instead of studying this problem in the smooth category we attempt to apply the ``spectrum generating method'' by Branson--\'{O}lafsson--{\O}rsted~\cite{BOO96} in the study of intertwining operators in the category of $(\frakg',K')$-modules, and verify a conjecture by Kobayashi on the automatic continuity of symmetry breaking operators between Harish-Chandra modules. To given smooth admissible representations $\pi$ of $G$ and $\tau$ of $G'$ one can associate the underlying Harish-Chandra modules $\pi_\HC$ and $\tau_\HC$. These are admissible $(\frakg,K)$-modules resp. $(\frakg',K')$-modules, realized on the spaces of $K$-finite resp. $K'$-finite vectors of $\pi$ resp. $\tau$, where $K\subseteq G$ and $K'\subseteq G'$ are maximal compact subgroups. We consider the space
$$ \Hom_{(\frakg',K')}(\pi_\HC|_{(\frakg',K')},\tau_\HC) $$
of intertwining operators in the category of Harish-Chandra modules. The natural restriction map
\begin{equation}
 \Hom_{G'}(\pi|_{G'},\tau) \to \Hom_{(\frakg',K')}(\pi_\HC|_{(\frakg',K')},\tau_\HC)\label{eq:SymBreakingRestrictionMap}
\end{equation}
is injective but in general not surjective and hence there might be more intertwining operators in the category of Harish-Chandra modules than in the smooth category. According to Kobayashi~\cite[Remark 10.2~(4)]{Kob14} it is plausible that this map is surjective if the space $(G\times G')/\diag(G')$ is real spherical. (Note that for $G'=G$ the map is surjective by the Casselman--Wallach Theorem.)

We remark that for $(G,G')=(\GL(2,\FF)\times\GL(2,\FF),\GL(2,\FF))$, $\FF=\RR,\CC$, and $(G,G')=(\GL(2,\CC),\GL(2,\RR))$ intertwining operators between Harish-Chandra modules were previously studied by Loke~\cite{Lok01} using explicit computations.

\subsection{Symmetry breaking of principal series} In this paper we outline a method to classify symmetry breaking operators between the Harish-Chandra modules of principal series representations induced from maximal parabolic subgroups, and their composition factors. Let $P=MAN\subseteq G$ be a maximal parabolic subgroup of $G$ such that $P'=P\cap G'=M'A'N'$ is maximal parabolic in $G'$ and write $\fraka$ and $\fraka'$ for the Lie algebras of $A$ and $A'$. Fix $\nu\in\fraka^*$ such that the roots of $(P,A)$ are given by $\{\nu,2\nu,\ldots,q\nu\}$ and do similarly for $\nu'\in(\fraka')^*$. Consider the principal series representations (smooth normalized parabolic induction)
$$ \pi_{\xi,r}=\Ind_P^G(\xi\otimes e^{r\nu}\otimes\1), \qquad \tau_{\xi',r}=\Ind_{P'}^{G'}(\xi'\otimes e^{r'\nu'}\otimes\1), $$
where $\xi$ and $\xi'$ are finite-dimensional representations of $M$ and $M'$ and $r,r'\in\CC$. Let $\xi'=\xi|_{M'}$ and assume that for all $K$-types $\alpha$ of $\pi_{\xi,r}$ and all $K'$-types $\alpha'$ of $\tau_{\xi',r'}$ the multiplicity-free properties
$$ \dim\Hom_K(\alpha,\pi_{\xi,r}|_K) \leq 1, \qquad \dim\Hom_{K'}(\alpha',\tau_{\xi',r'}|_{K'}) \leq 1, \qquad \dim\Hom_{K'}(\alpha|_{K'},\alpha') \leq 1. $$
hold, i.e. $\pi_{\xi,r}$ is $K$-multiplicity-free, $\tau_{\xi',r'}$ is $K'$-multiplicity-free, and every $K$-type in $\pi_{\xi,r}$ is $K'$-multiplicity-free.

Let $T:(\pi_{\xi,r})_\HC\to(\tau_{\xi',r'})_\HC$ be a $(\frakg',K')$-intertwining operator, then $T$ is in particular $K'$-intertwining. Consider a pair $(\alpha;\alpha')$ consisting of a $K$-type $\alpha$ in $\pi_{\xi,r}$ and a $K'$-type $\alpha'$ in $\tau_{\xi',r'}$ which also occurs in $\alpha|_{K'}$. By the multiplicity-free assumptions the restriction of $T$ to the $K'$-type $\alpha'$ inside the $K$-type $\alpha$ in $\pi_{\xi,r}$ maps to the $K'$-type $\alpha'$ in $\tau_{\xi',r'}$ and is unique up to a scalar $t_{\alpha,\alpha'}\in\CC$ (see Section~\ref{sec:RelatingKKprimeTypes} for the precise definition). This encodes every $K'$-intertwining operator $T:(\pi_{\xi,r})_\HC\to(\tau_{\xi',r'})_\HC$ into scalars $t_{\alpha,\alpha'}$. Using the method of spectrum generating operators by Branson--\'{O}lafsson--{\O}rsted~\cite{BOO96} we prove:

\begin{thmalph}[see Theorem~\ref{thm:EigenvalueCondition} and Corollary~\ref{cor:CharacterizationIntertwinersScalar}]\label{thm:IntroCharacterizationIntertwiners}
Let $T:(\pi_{\xi,r})_\HC\to(\tau_{\xi',r'})_\HC$ be a $K'$-intertwining operator given by scalars $t_{\alpha,\alpha'}$. Then $T$ is $(\frakg',K')$-intertwining if and only if for all pairs $(\alpha;\alpha')$ and every $K'$-type $\beta'$ the following relation holds:
\begin{equation}
 \sum_{\substack{\beta\\(\alpha;\alpha')\leftrightarrow(\beta;\beta')}}\lambda_{\alpha,\alpha'}^{\beta,\beta'}(\sigma_\beta-\sigma_\alpha+2r)t_{\beta,\beta'} = (\sigma'_{\beta'}-\sigma'_{\alpha'}+2r')t_{\alpha,\alpha'}.\label{eq:IntroRelation}
\end{equation}
\end{thmalph}

Here we write $(\alpha;\alpha')\leftrightarrow(\beta;\beta')$ if the $K'$-type $\beta'$ inside the $K$-type $\beta$ in $\pi_{\xi,r}$ can be reached from $\alpha'$ inside $\alpha$ by a single application of $\pi_{\xi,r}(\frakg')$ for generic $r\in\CC$ (see Section~\ref{sec:IntertwiningOperatorsBetweenHarishChandraModules} for details). Further, $\sigma_\alpha$ and $\sigma'_{\alpha'}$ as well as $\lambda_{\alpha,\alpha'}^{\beta,\beta'}$ are certain constants depending only on the representations $\xi$ and $\xi'$ (see Sections~\ref{sec:SpectrumGeneratingOperator} and \ref{sec:ScalarIdentities} for their definition).

We note that the relations characterizing intertwining operators depend linearly on the induction parameters $r$ and $r'$ and turn the representation theoretic problem of classifying symmetry breaking operators into a combinatorial problem. We also remark that Theorem~\ref{thm:IntroCharacterizationIntertwiners} admits a slight modification characterizing also intertwining operators between any subquotients of $\pi_{\xi,r}$ and $\tau_{\xi',r'}$ (see Remark~\ref{rem:IntertwinersBetweenCompositionFactors}).

\subsection{Examples} For the two pairs $(G,G')=(\upO(1,n),\upO(1,n-1))$, $n\geq3$, and $(\upU(1,n),\upU(1,n-1))$, $n\geq2$, we explicitly write down the linear relations for the scalars $t_{\alpha,\alpha'}$ characterizing intertwining operators in the case where $\xi=\1$ is the trivial representation (see Theorems~\ref{thm:RealCharacterizationIntertwiners} and \ref{thm:ComplexCharacterizationIntertwiners}), and use these relations to compute multiplicities. For the statements we abbreviate $\pi_r=\pi_{\1,r}$ and $\tau_{r'}=\tau_{\1,r'}$. If $\calV$ is a $(\frakg,K)$-module and $\calW$ a $(\frakg',K')$-module we write
$$ m(\calV,\calW) = \dim\Hom_{(\frakg',K')}(\calV|_{(\frakg',K')},\calW). $$
We note that much of the notation used here follows \cite{KS13}.

\begin{thmalph}[see Theorems~\ref{thm:RealMultiplicities}~(1) and \ref{thm:ComplexMultiplicities}~(1)]\label{thm:IntroMultiplicities}
\begin{enumerate}
\item For $(G,G')=(\upO(1,n),\upO(1,n-1))$ we have
$$ m((\pi_r)_\HC,(\tau_{r'})_\HC) = \begin{cases}1 & \mbox{for $(r,r')\in\CC^2\setminus\Leven$,}\\2 & \mbox{for $(r,r')\in\Leven$,}\end{cases} $$
where $\Leven=\{(-\frac{n-1}{2}-i,-\frac{n-2}{2}-j):i,j\in\NN,i-j\in2\NN\}.$
\item For $(G,G')=(\upU(1,n),\upU(1,n-1))$ we have
$$ m((\pi_r)_\HC,(\tau_{r'})_\HC) = \begin{cases}1 & \mbox{for $(r,r')\in\CC^2\setminus L$,}\\2 & \mbox{for $(r,r')\in L$,}\end{cases} $$
where $L=\{(-n-2i,-(n-1)-2j):i,j\in\NN,j\leq i\}.$
\end{enumerate}
\end{thmalph}

Multiplicity two does not contradict the multiplicity one statements for the above pairs $(G,G')$, because for $(r,r')\in\Leven$ resp. $L$ both representations $\pi_r$ and $\tau_{r'}$ are reducible. In the case $(G,G')=(\upO(1,n),\upO(1,n-1))$ the representation $(\pi_r)_\HC$ is reducible iff $r=\pm(\frac{n-1}{2}+i)$, $i\in\NN$, its composition factors consisting of a finite-dimensional subrepresentation $\calF(i)$ and an infinite-dimensional unitarizable quotient $\calT(i)$. Similarly, in the case $(G,G')=(\upU(1,n),\upU(1,n-1))$ the representation $(\pi_r)_\HC$ is reducible iff $r=\pm(n+2i)$, $i\in\NN$, and its composition factors consist of a finite-dimensional subrepresentation $\calF(i)$, two proper subquotients $\calT_\pm(i)$, and a unitarizable quotient $\calT(i)$. Write $\calF'(j)$, $\calT'_\pm(j)$ and $\calT'(j)$ for the corresponding composition factors of $(\tau_{r'})_\HC$ at $r'=-\frac{n-2}{2}-j$ resp. $r'=-(n-1)-2j$.

\begin{thmalph}[see Theorems~\ref{thm:RealMultiplicities}~(2) and \ref{thm:ComplexMultiplicities}~(2)]\label{thm:IntroMultiplicitiesSubquotients}
\begin{enumerate}
\item For $(G,G')=(\upO(1,n),\upO(1,n-1))$, the multiplicities $m(\calV,\calW)$ are given by
\vspace{.2cm}
\begin{center}
 \begin{tabular}{c|cc}
  \diagonal{.1em}{.72cm}{$\calV$}{$\calW$} & $\calF'(j)$ & $\calT'(j)$ \\
  \hline
  $\calF(i)$ & $1$ & $0$\\
  $\calT(i)$ & $0$ & $1$\\
  \multicolumn{3}{c}{for $i-j\in2\NN$,}
 \end{tabular}
 \qquad
 \begin{tabular}{c|cc}
  \diagonal{.1em}{.72cm}{$\calV$}{$\calW$} & $\calF'(j)$ & $\calT'(j)$ \\
  \hline
  $\calF(i)$ & $0$ & $0$\\
  $\calT(i)$ & $1$ & $0$\\
  \multicolumn{3}{c}{otherwise.}
 \end{tabular}
\end{center}
\vspace{.2cm}
\item For $(G,G')=(\upU(1,n),\upU(1,n-1))$, the multiplicities $m(\calV,\calW)$ are given by
\vspace{.2cm}
\begin{center}
 \begin{tabular}{c|cccc}
  \diagonal{.1em}{.85cm}{$\calV$}{$\calW$} & $\calF'(j)$ & $\calT_+'(j)$ & $\calT_-'(j)$ & $\calT'(j)$ \\
  \hline
  $\calF(i)$ & $1$ & $0$ & $0$ & $0$\\
  $\calT_+(i)$ & $0$ & $1$ & $0$ & $0$\\
  $\calT_-(i)$ & $0$ & $0$ & $1$ & $0$\\
  $\calT(i)$ & $0$ & $0$ & $0$ & $1$\\
  \multicolumn{5}{c}{for $j\leq i$,}
 \end{tabular}
 \qquad
 \begin{tabular}{c|cccc}
  \diagonal{.1em}{.85cm}{$\calV$}{$\calW$} & $\calF'(j)$ & $\calT_+'(j)$ & $\calT_-'(j)$ & $\calT'(j)$ \\
  \hline
  $\calF(i)$ & $0$ & $0$ & $0$ & $0$\\
  $\calT_+(i)$ & $0$ & $0$ & $0$ & $0$\\
  $\calT_-(i)$ & $0$ & $0$ & $0$ & $0$\\
  $\calT(i)$ & $1$ & $0$ & $0$ & $0$\\
  \multicolumn{5}{c}{otherwise.}
 \end{tabular}
\end{center}
\vspace{.2cm}
\end{enumerate}
\end{thmalph}

We further construct a basis of $\Hom_{(\frakg',K')}((\pi_r)_\HC|_{(\frakg',K')},(\tau_{r'})_\HC)$ for all $r,r'\in\CC$ by solving the relations \eqref{eq:IntroRelation} explicitly. More precisely, we find a family $(t_{\alpha,\alpha'}(r,r'))_{\alpha,\alpha'}$ consisting of rational functions in $r,r'\in\CC$ that solve the relations \eqref{eq:IntroRelation}. Renormalizing the functions $t_{\alpha,\alpha'}(r,r')$ gives a family $(t^{(1)}_{\alpha,\alpha'}(r,r'))_{\alpha,\alpha'}$ of holomorphic functions in $r,r'\in\CC$ satisfying the relations \eqref{eq:IntroRelation}. By Theorem~\ref{thm:IntroCharacterizationIntertwiners} this constructs intertwining operators $T^{(1)}(r,r')$ depending holomorphically on $r,r'\in\CC$. We show that
$$ T^{(1)}(r,r')=0 \qquad \mbox{if and only if $(r,r')\in\Leven$ resp. $L$.} $$
For each $(r,r')\in\Leven$ resp. $L$ the holomorphic function $T^{(1)}(r,r')$ can be renormalized along two different affine complex lines through $(r,r')$, and one obtains two different non-trivial operators $T^{(2)}(r,r'),T^{(3)}(r,r')$ for every $(r,r')\in\Leven$ resp. $L$ (see Propositions~\ref{prop:RealExplicitEigenvalues} and \ref{prop:ComplexExplicitEigenvalues} for details).

\begin{thmalph}[see Theorems~\ref{thm:RealExplicitBasis} and \ref{thm:ComplexExplicitBasis} and Remarks~\ref{rem:RenormalizationForSubquotientsReal} and \ref{rem:RenormalizationForSubquotientsCplx}]\label{thm:IntroExplicitIntertwiners}
For the pair $(G,G')=(\upO(1,n),\upO(1,n-1))$ resp. $(\upU(1,n),\upU(1,n-1))$ we have
$$ \Hom_{(\frakg',K')}((\pi_r)_\HC|_{(\frakg',K')},(\tau_{r'})_\HC) = \begin{cases}\CC T^{(1)}(r,r') & \mbox{for $(r,r')\in\CC^2\setminus\calL$,}\\\CC T^{(2)}(r,r')\oplus\CC T^{(3)}(r,r') & \mbox{for $(r,r')\in\calL$,}\end{cases} $$
where $\calL=\Leven$ resp. $L$. Moreover, by composing $T^{(1)}(r,r')$ with embeddings and quotient maps for the composition factors of $\pi_r$ and $\tau_{r'}$, and renormalizing along certain affine complex lines, one can obtain every intertwining operator between arbitrary composition factors of $(\pi_r)_\HC$ and $(\tau_{r'})_\HC$.
\end{thmalph}

The previous theorem shows that basically all the information about intertwining operators between spherical principal series of $G$ and $G'$ and their composition factors is contained in the single holomorphic family $T^{(1)}(r,r')$ of intertwiners.

Finally we turn to the question of whether every intertwining operator between the Harish-Chandra modules $(\pi_r)_\HC$ and $(\tau_{r'})_\HC$ lifts to an intertwining operator between the smooth globalizations $\pi_r$ and $\tau_{r'}$, i.e. the question of whether \eqref{eq:SymBreakingRestrictionMap} is an isomorphism.

\begin{thmalph}[see Corollaries~\ref{cor:HCtoInftyReal} and \ref{cor:HCtoInftyCplx}]\label{thm:IntroAutomaticContinuity}
For the pairs $(G,G')=(\upO(1,n),\upO(1,n-1))$ and $(\upU(1,n),\upU(1,n-1))$ every intertwining operator between $(\pi_r)_\HC$ and $(\tau_{r'})_\HC$ (resp. any of their subquotients) extends to a continuous intertwining operator between $\pi_r$ and $\tau_{r'}$ (resp. the Casselman--Wallach completions of the subquotients). In particular, the injective map \eqref{eq:SymBreakingRestrictionMap} is surjective for all spherical principal series representations and their subquotients.
\end{thmalph}

This verifies Kobayashi's conjecture~\cite[Remark 10.2~(4)]{Kob14} in the above cases.

For $(G,G')=(\upO(1,n),\upO(1,n-1))$ the analogues of Theorems~\ref{thm:IntroMultiplicities}, \ref{thm:IntroMultiplicitiesSubquotients} and \ref{thm:IntroExplicitIntertwiners} in the smooth category, i.e. for $\pi_r$ and $\tau_{r'}$ instead of $(\pi_r)_\HC$ and $(\tau_{r'})_\HC$, were established by Kobayashi--Speh~\cite{KS13} using analytic techniques. With Theorem~\ref{thm:IntroAutomaticContinuity} we obtain a new proof of their results as well as the corresponding results for $(G,G')=(\upU(1,n),\upU(1,n-1))$.

\subsection{Application} For $(G,G')=(\upO(1,n),\upO(1,n-1))$ we further present an application of the classification of symmetry breaking operators. In Theorem~\ref{thm:RealDiscreteComponentsInUniReps} we use the explicit formula for the numbers $t_{\alpha,\alpha'}$ to construct discrete components in the restriction of certain unitary representations of $G$ to $G'$. The representations in question are either spherical complementary series representations (i.e. those $\pi_r$ which are unitarizable) or the unitarizable quotients $\calT(i)$. This extends and gives new proofs of previous results by Speh--Venkataramana~\cite{SV11}, Zhang~\cite{Zha15}, Kobayashi--Speh~\cite{KS13} and M\"{o}llers--Oshima~\cite{MO12} (see Remark~\ref{rem:RealDiscreteComponentsInUniReps}). Analogous results hold for $(G,G')=(\upU(1,n),\upU(1,n-1))$.

\subsection{Structure of the paper} In Section~\ref{sec:Preliminaries} we fix the notation for principal series representations and recall the method of spectrum generating operators~\cite{BOO96}. This method is applied in Section~\ref{sec:GeneralTheory} to obtain an equivalent characterization of intertwining operators in the category of $(\frakg',K')$-modules by means of scalar identities. After this quite general approach, we study in Section~\ref{sec:ExampleReal} the special case $(G,G')=(\upO(1,n),\upO(1,n-1))$ in detail and give some applications. Finally, in Section~\ref{sec:ExampleComplex} we repeat the same procedure for $(G,G')=(\upU(1,n),\upU(1,n-1))$ providing a new classification of symmetry breaking operators in this example. Appendix~\ref{app:OrthPoly} contains some basic properties of Gegenbauer and Jacobi polynomials which are used in Appendix~\ref{app:SphericalHarmonics} to describe explicit branching laws for real and complex spherical harmonics.\\

\paragraph{\bf Notation.} $\NN=\{0,1,2,\ldots\}$.

\section{Preliminaries}\label{sec:Preliminaries}

We fix the necessary notation, discuss induced representations and the method of the spectrum generating operator by Branson--\'{O}lafsson--{\O}rsted~\cite{BOO96}.

\subsection{Compatible maximal parabolic subgroups}

Let $G$ be a reductive Lie group with compact center and $G'\subseteq G$ a reductive subgroup also with compact center. Denote by $\frakg$ and $\frakg'$ the Lie algebras of $G$ and $G'$. Choose a maximal parabolic subgroup $P\subseteq G$ with the property that $P'=P\cap G'$ is maximal parabolic in $G'$ and write $P=MAN$ and $P'=M'A'N'$ for the Langlands decompositions of $P$ and $P'$. We fix a Cartan involution $\theta$ of $G$ which leaves $G'$ and the Levi subgroups $MA$ and $M'A'$ invariant. Write $K=G^\theta$ and $K'=(G')^\theta$ for the corresponding fixed point subgroups of $G$ and $G'$ which are maximal compact and denote by $\frakk$ and $\frakk'$ their Lie algebras. Let $\fraks$ and $\fraks'$ be the $(-1)$-eigenspaces of $\theta$ on $\frakg$ and $\frakg'$ so that
$$ \frakg=\frakk\oplus\fraks, \qquad \frakg'=\frakk'\oplus\fraks'. $$

\begin{example}
\begin{enumerate}
\item Let $(G,G')$ be one of the pairs
\begin{align*}
 &(\upO(1,n),\upO(1,n-1)), &&(\upU(1,n),\upU(1,n-1)),\\
 &(\Sp(1,n),\Sp(1,n-1)), &&(F_{4(-20)},\Spin(8,1)).
\end{align*}
Then one can choose the minimal parabolic $P$ such that $P'=P\cap G'$ is minimal parabolic in $G'$. Since $G$ and $G'$ are of rank one, minimal parabolics are maximal and hence satisfy our assumptions.
\item Let
$$ (G,G')=(\SL(n,\RR),\SL(n-1,\RR)) $$
with $G'$ embedded in $G$ as the upper left block. Then all standard maximal parabolics $P=P_{p,q}=(S(\GL(p,\RR)\times\GL(q,\RR)))\ltimes\RR^{p\times q}$ corresponding to the partition $n=p+q$ with $q>1$ satisfy the assumptions. In this case $P'=P\cap G'$ is the standard maximal parabolic of $G'$ corresponding to the partition $n-1=p+(q-1)$.
\end{enumerate}
\end{example}

\subsection{Principal series representations}

For any finite-dimensional representation $(\xi,V_\xi)$ of $M$ and any $\nu\in\fraka_\CC^*$, where $\fraka$ denotes the Lie algebra of $A$, consider the induced representation $\Ind_P^G(\xi\otimes e^\nu\otimes\1)$ (normalized smooth parabolic induction). This representation is realized on the space
$$ \calE(G;\xi,\nu) = \{F\in C^\infty(G,V_\xi):F(gman)=a^{-\nu-\rho}\xi(m)^{-1}F(g)\forall\,g\in G,man\in MAN\}, $$
where $\rho=\frac{1}{2}\tr\ad|_\frakn\in\fraka^*$. The group $G$ acts on $\calE(G;\xi,\nu)$ by the left-regular action. Since $G=KP$, restriction to $K$ is an isomorphism $\calE(G;\xi,\nu)\to\calE(K;\xi|_{M\cap K})$ where
$$ \calE(K;\xi|_{M\cap K}) = \{F\in C^\infty(K,V_\xi):F(km)=\xi(m)^{-1}F(k)\forall\,k\in K,m\in M\cap K\}. $$
Let $\pi_{\xi,\nu}$ denote the action of $G$ on $\calE(K;\xi|_{M\cap K})$ which makes this isomorphism $G$-equivariant. Then $(\pi_{\xi,\nu},\calE(K;\xi|_{M\cap K}))$ is a smooth admissible representation of $G$. The restriction of $\pi_{\xi,\nu}$ to $K$ is simply the left-regular representation of $K$ on $\calE(K;\xi|_{M\cap K})$.

Corresponding to the smooth representation $\pi_{\xi,\nu}$ we consider its underlying $(\frakg,K)$-module $(\pi_{\xi,\nu})_\HC$ realized on the space $\calE=\calE(K;\xi|_{M\cap K})_K$ of $K$-finite vectors. Abusing notation we denote the action of the Lie algebra $\frakg$ on $\calE$ also by $\pi_{\xi,\nu}$. Then the restriction of $(\pi_{\xi,\nu})_\HC$ to $K$ decomposes as
$$ \calE = \bigoplus_{\alpha\in\widehat{K}} \calE(\alpha) $$
with $\calE(\alpha)$ being the $\alpha$-isotypic component in $\calE$. Note that $\calE$ and hence its decomposition into $K$-isotypic components is independent of $\nu\in\fraka_\CC^*$ and only depends on $\xi$.

Similarly we consider $\tau_{\xi',\nu'}=\Ind_{P'}^{G'}(\xi'\otimes e^{\nu'}\otimes\1)$ for a finite-dimensional representation $(\xi',V_{\xi'})$ of $M'$ and an element $\nu'\in(\fraka')_\CC^*$, and its underlying $(\frakg',K')$-module $(\tau_{\xi',\nu'})_\HC$ realized on the space $\calE'=\calE(K';\xi'|_{K'\cap M'})_{K'}$. As above we decompose the restriction of $(\tau_{\xi',\nu'})_\HC$ to $K'$
$$ \calE' = \bigoplus_{\alpha'\in\widehat{K'}} \calE'(\alpha') $$
with $\calE'(\alpha')$ being the $\alpha'$-isotypic component.

\subsection{The spectrum generating operator}\label{sec:SpectrumGeneratingOperator}

Since $P$ is a maximal parabolic subgroup we have $\dim\fraka=1$ and we can choose $H\in\fraka$ such that the eigenvalues of $\ad(H)$ on the Lie algebra $\frakn$ of $N$ are $1,\ldots,q$. Define $\nu\in\fraka^*$ by $\nu(H)=1$, then $\Sigma(\frakg,\fraka)=\{\pm\nu,\ldots,\pm q\nu\}$. We abbreviate $\pi_{\xi,r}=\pi_{\xi,r\nu}$ for $r\in\CC$.

Let $B$ be an invariant non-degenerate symmetric bilinear form on $\frakg$ normalized by $B(H,H)=1$. For $1\leq j\leq q$ let
$$ \frakk_j = \frakk\cap(\frakg_{j\nu}+\frakg_{-j\nu}). $$
Choose a basis $(X_{j,k})_k$ of $\frakk_j$, denote by $(X_{j,k}')_k$ the corresponding dual basis with respect to $B$ and put
$$ \Cas_j = \sum_k X_{j,k}X_{j,k}'. $$
Then $\Cas_j$ is an element of $\calU(\frakk)$, the universal enveloping algebra of $\frakk$. Clearly the elements $\Cas_j\in\calU(\frakk)$ do not depend on the choice of the corresponding bases. Following \cite{BOO96} we define the \textit{spectrum generating operator} as the second order element in $\calU(\frakk)$ given by
$$ \calP = \sum_{j=1}^q j^{-1}\Cas_j. $$
We remark that even though the spaces $\frakk_j$ do not form subalgebras the operator $\calP$ can be written as a rational linear combination of Casimir elements of subalgebras of $\frakk$ (see \cite[Remark 2.4]{BOO96}). Since the left-regular representation of $K$ on $\calE$ commutes with the right-action $\calR_\calP$ of $\calP$ the restriction of $\calR_\calP$ to each isotypic component $\calE(\alpha)$ is a linear transformation
$$ \sigma_\alpha=\sigma_{\alpha,\xi|_{M\cap K}}\in\End\calE(\alpha) $$
which only depends on $\xi$ but not on $\nu$.

Similarly we define $H'\in\fraka'$, $\nu'\in(\fraka')^*$ and choose an invariant non-degenerate symmetric bilinear form $B'$ on $\frakg'$ with $B'(H',H')=1$. Let $\calP'$ denote the spectrum generating operator for $G'$ and write $\sigma'_{\alpha'}\in\End\calE'(\alpha')$ for the restriction of $\calR_{\calP'}$ to the isotypic component $\calE'(\alpha')$.

\subsection{Reduction to the cocycle}

For each $X\in\frakg_\CC$ we define a scalar-valued function $\omega(X)$ on $K$ by
$$ \omega(X)(k) = B(\Ad(k^{-1})X,H), \qquad k\in K, $$
where we extend $B$ to a symmetric $\CC$-bilinear form on $\frakg_\CC$. This defines a $K$-equivariant map
$$ \omega: \frakg_\CC\to\calE(K;\1)\cong C^\infty(K/(M\cap K)), $$
where $\1$ is the trivial $M\cap K$-representation. The map $\omega$ is called a \textit{cocycle}. Note that $\omega$ vanishes on $\frakk_\CC$. Let $m(\omega(X))$ denote the multiplication operator
$$ \calE\to\calE, \quad \varphi\mapsto\omega(X)\varphi. $$
For $\alpha,\beta\in\widehat{K}$ with $\calE(\alpha),\calE(\beta)\neq0$ we let
$$ \omega_\alpha^\beta(X) = \proj_{\calE(\beta)}\circ\, m(\omega(X))|_{\calE(\alpha)}, \qquad X\in\frakg_\CC, $$
where $\proj_{\calE(\beta)}$ denotes the projection from $\calE$ onto $\calE(\beta)$. We can now express the differential representation $\pi_{\xi,r}$ of $\frakg$ on $\calE$ in terms of the cocycle $\omega$ and the maps $\sigma_\alpha$:

\begin{theorem}[{\cite[Corollary 2.6]{BOO96}}]\label{thm:ReductionToCocycle}
For $X\in\fraks_\CC$ and any $\alpha,\beta\in\widehat{K}$ with $\calE(\alpha),\calE(\beta)\neq0$ we have
\begin{equation}
 \proj_{\calE(\beta)}\circ\,\pi_{\xi,r}(X)|_{\calE(\alpha)} = \tfrac{1}{2}(\sigma_\beta\omega_\alpha^\beta(X)-\omega_\alpha^\beta(X)\sigma_\alpha+2r\omega_\alpha^\beta(X)).\label{eq:ReductionToCocycle}
\end{equation}
\end{theorem}

Similarly we denote by $\omega'(X)$ the corresponding cocycle for $G'$ and by $\omega_{\alpha'}^{\beta'}(X)$ the corresponding map from $\calE'(\alpha')$ to $\calE'(\beta')$. Then we obtain for $X\in\fraks'_\CC$ and any $\alpha',\beta'\in\widehat{K'}$ with $\calE'(\alpha'),\calE'(\beta')\neq0$ the analogous identity:
\begin{equation}
 \proj_{\calE'(\beta')}\circ\,\tau_{\xi',r'}(X)|_{\calE'(\alpha')} = \tfrac{1}{2}(\sigma'_{\beta'}\omega_{\alpha'}^{\beta'}(X)-\omega_{\alpha'}^{\beta'}(X)\sigma'_{\alpha'}+2r'\omega_{\alpha'}^{\beta'}(X)).\label{eq:ReductionToCocycle'}
\end{equation}

\section{The compact picture of symmetry breaking operators}\label{sec:GeneralTheory}

Consider the admissible $(\frakg,K)$-module $(\pi_{\xi,r})_\HC$. Then its restriction $(\pi_{\xi,r})_\HC|_{(\frakg',K')}$ is a $(\frakg',K')$-module which is in general not admissible anymore. However, we can still study the space
$$ \Hom_{(\frakg',K')}((\pi_{\xi,r})_\HC|_{(\frakg',K')},(\tau_{\xi',r'})_\HC) $$
of intertwining operators between the $(\frakg',K')$-modules. In this section we use Theorem~\ref{thm:ReductionToCocycle} to characterize these intertwining operators in terms of their action on the $K'$-isotypic components in the $K$-types $\calE(\alpha)$.

\subsection{Relating $K$-types and $K'$-types}\label{sec:RelatingKKprimeTypes}

From now on we assume that both $\calE$ and $\calE'$ are multiplicity-free, i.e.
\begin{equation}
 \dim\Hom_K(\alpha,\calE),\dim\Hom_{K'}(\alpha',\calE')\leq 1 \qquad \forall\,\alpha\in\widehat{K},\alpha'\in\widehat{K'}.\tag{MF1}\label{eq:MF1}
\end{equation}
This implies by Schur's Lemma that the maps $\sigma_\alpha$ and $\sigma'_{\alpha'}$ are scalars which we denote by the same symbols. We further assume that each $K'$-type $\calE'(\alpha')\neq0$ occurs at most once in each $K$-type $\calE(\alpha)\neq0$, i.e.
\begin{equation}
 \dim\Hom_{K'}(\calE(\alpha),\calE'(\alpha')) \leq 1 \qquad \forall\,\calE(\alpha),\calE'(\alpha')\neq0.\tag{MF2}\label{eq:MF2}
\end{equation}

Each $K$-isotypic component $\calE(\alpha)$ decomposes under the action of $K'\subseteq K$ into
$$ \calE(\alpha) = \bigoplus_{\alpha'\in\widehat{K'}} \calE(\alpha;\alpha'), $$
where $\calE(\alpha;\alpha')$ is the $\alpha'$-isotypic component in $\calE(\alpha)$. Then our assumptions imply that whenever $\calE(\alpha;\alpha'),\calE'(\alpha')\neq0$ then $\calE(\alpha;\alpha')\cong\calE'(\alpha')$. In all such cases we fix an isomorphism
$$ R_{\alpha,\alpha'}:\calE(\alpha;\alpha')\stackrel{\sim}{\to}\calE'(\alpha'). $$
To simplify notation, let $R_{\alpha,\alpha'}=0$ whenever $\calE'(\alpha')=0$, so that we have surjective $K'$-equivariant maps $R_{\alpha,\alpha'}:\calE(\alpha;\alpha')\to\calE'(\alpha')$ for all $\calE(\alpha;\alpha')\neq0$.

In applications, it is often useful to choose a natural isomorphism $\calE(\alpha;\alpha')\cong\calE'(\alpha')$ relating $K$-types and $K'$-types. For this we study the restriction of functions from $K$ to $K'$. Assume for simplicity that $\xi'=\xi|_{M'}$. In this case we can consider the restriction operator
$$ \rest:\calE\to\calE', \quad \varphi\mapsto\varphi|_{K'}. $$
This operator is $K'$-equivariant and hence, if $\rest$ is non-zero on some $K'$-type $\calE(\alpha;\alpha')$ in $\calE$ then $\rest|_{\calE(\alpha;\alpha')}$ is an isomorphism onto $\calE'(\alpha')$ by Schur's Lemma. However, $\rest$ might also vanish on some $\calE(\alpha;\alpha')$ and therefore we need to combine the restriction with differentiation in the normal direction.

For this we write
$$ \frakk = (\frakm\cap\frakk) \oplus \frakq $$
where $\frakq$ is the orthogonal complement of $(\frakm\cap\frakk)$ in $\frakk$ with respect to the invariant form $B$. Note that $M\cap K$ acts on $\frakq$. Similarly
$$ \frakk' = (\frakm'\cap\frakk') \oplus \frakq'. $$
Let $\frakq''$ denote the orthogonal complement of $\frakq'$ in $\frakq$, then
$$ \frakq = \frakq' \oplus \frakq''. $$
We note that since $M\cap K$ acts on $\frakq$ and $M'\cap K'$ acts on $\frakq'$, the group $M'\cap K'$ also acts on $\frakq''$. We then have
$$ \frakk/(\frakm\cap\frakk) \cong \frakk'/(\frakm'\cap\frakk') \oplus \frakq'', $$
i.e. $\frakq''$ identifies with the normal space of $K'/(M'\cap K')$ in $K/(M\cap K)$ at the base point. Denote by $S(\frakq'')$ the symmetric algebra over $\frakq''$ and by $S(\frakq'')^{M'\cap K'}$ its $(M'\cap K')$-invariants. Note that $S(\frakq'')^{M'\cap K'}$ acts naturally from the right by differential operators on functions defined on a small neighborhood of $K'/(M'\cap K')$ in $K/(M\cap K)$.

\begin{lemma}\label{lem:ResNormalDiff}
Let $(\alpha,\alpha')\in\widehat{K}\times\widehat{K'}$ with $\calE(\alpha;\alpha')\neq0$ and $D\in S(\frakq'')^{M'\cap K'}$. Then the map $\rest\circ\, D:\calE\to\calE'$ is $K'$-equivariant. In particular,
$$ (\rest\circ\, D)|_{\calE(\alpha;\alpha')}:\calE(\alpha;\alpha')\to\calE'(\alpha') $$
is an isomorphism whenever it is non-zero.
\end{lemma}

\begin{remark}
Of course one could as well consider other irreducible $M'\cap K'$-subrepresentations of $S(\frakq'')$ than the trivial one. In fact, using an idea of {\O}rsted--Vargas~\cite{OV04} one can construct an injective $K'$-equivariant map
$$ \calE=C^\infty(K\times_{M\cap K}\xi)_K \to \bigoplus_{m=0}^\infty C^\infty(K'\times_{M'\cap K'}(\xi\otimes S^m(\frakq'')))_{K'} $$
and use it to relate $K$-types and $K'$-types of the induced representations $\pi_{\xi,r}$ and $\tau_{\xi',r'}$ for $\xi'|_{M'\cap K'}$ any subrepresentation of $\xi|_{M'\cap K'}\otimes S(\frakq'')$. Lemma~\ref{lem:ResNormalDiff} can then be viewed as the special case $\xi'=\xi\otimes\CC D$ where $D\in S(\frakq'')^{M'\cap K'}$ and hence $\CC D$ is the trivial $M'\cap K'$-representation.
\end{remark}

\subsection{Intertwining operators between Harish-Chandra modules}\label{sec:IntertwiningOperatorsBetweenHarishChandraModules}

Let $\calV\subseteq\calE$ be a $(\frakg',K')$-submodule of $(\pi_{\xi,r})_\HC$, i.e. $\calV$ is stable under $\pi_{\xi,r}(\frakg')$ and stable under $\pi_{\xi,r}(K')$. A linear map $T:\calV\to\calE'$ is called an \textit{intertwining operator for $\pi_{\xi,r}$ and $\tau_{\xi',r'}$} if for every $v\in\calV$ we have
\begin{align}
 (T\circ\, \pi_{\xi,r}(X))v &= (\tau_{\xi',r'}(X)\circ\, T)v && \forall\,X\in\frakg',\label{eq:IntertwiningPropertyDiff}\\
 (T\circ\, \pi_{\xi,r}(k))v &= (\tau_{\xi',r'}(k)\circ\, T)v && \forall\,k\in K'.\label{eq:IntertwiningPropertyK}
\end{align}
In particular an intertwining operator commutes by \eqref{eq:IntertwiningPropertyK} with the action of $K'$ and hence restricts to a map $T_{\alpha,\alpha'}=T|_{\calE(\alpha;\alpha')}:\calE(\alpha;\alpha')\to\calE'(\alpha')$ for all $\calE(\alpha;\alpha')\subseteq\calV$. If $\calE'(\alpha')=0$ then clearly $T_{\alpha,\alpha'}=0$. Recall that we fixed in Section~\ref{sec:RelatingKKprimeTypes} $K'$-equivariant maps $R_{\alpha,\alpha'}:\calE(\alpha;\alpha')\to\calE'(\alpha')$, then by Schur's Lemma $T_{\alpha,\alpha'}$ is a scalar multiple of $R_{\alpha,\alpha'}$. We write
\begin{equation}
 T_{\alpha,\alpha'} = t_{\alpha,\alpha'}\cdot R_{\alpha,\alpha'} \qquad \forall\,0\neq\calE(\alpha;\alpha')\subseteq\calV\label{eq:IntertwinerScalarMultiple}
\end{equation}
with $t_{\alpha,\alpha'}\in\CC$.

Restricting \eqref{eq:IntertwiningPropertyDiff} to $\calE(\alpha;\alpha')$ and composing with the projection $\proj_{\calE'(\beta')}$ we obtain
\begin{equation}
 \proj_{\calE'(\beta')}\circ\,T\circ\,\pi_{\xi,r}(X)|_{\calE(\alpha;\alpha')} = \proj_{\calE'(\beta')}\circ\,\tau_{\xi',r'}(X)\circ\,T|_{\calE(\alpha;\alpha')}.\label{eq:IntertwiningPropertyDiffRest}
\end{equation}
To simplify both sides we let
$$ \omega_{\alpha,\alpha'}^{\beta,\beta'}:\fraks'_\CC\otimes\calE(\alpha;\alpha')\to\calE(\beta;\beta'), \quad \omega_{\alpha,\alpha'}^{\beta,\beta'}(X) = \proj_{\calE(\beta;\beta')}\circ\, m(\omega(X))|_{\calE(\alpha;\alpha')}, $$
where we view $\omega_{\alpha,\alpha'}^{\beta,\beta'}(X)$, $X\in\fraks'$, as a linear map $\calE(\alpha;\alpha')\to\calE(\beta;\beta')$. Write $(\alpha;\alpha')\to(\beta;\beta')$ if $\omega_{\alpha,\alpha'}^{\beta,\beta'}\neq0$. The following lemma is proved along the same lines as \cite[Lemma 4.4~(c)]{BOO96} and justifies the use of the notation $(\alpha;\alpha')\leftrightarrow(\beta;\beta')$ instead of $(\alpha;\alpha')\to(\beta;\beta')$:

\begin{lemma}
For an orthonormal basis $(X_k)_k\subseteq\fraks'$ put
$$ s_{\alpha,\alpha'}^{\beta,\beta'} = \sum_k \omega_{\beta,\beta'}^{\alpha,\alpha'}(X_k)\circ\,\omega_{\alpha,\alpha'}^{\beta,\beta}(X_k). $$
Then $s_{\alpha,\alpha'}^{\beta,\beta'}$ is independent of the choice of $(X_k)_k$ and
$$ (\alpha;\alpha')\to(\beta;\beta') \quad \Leftrightarrow \quad s_{\alpha,\alpha'}^{\beta,\beta'}\neq0 \quad \Leftrightarrow \quad (\beta;\beta')\to(\alpha;\alpha'). $$
\end{lemma}

Now, on the left hand side of the identity~\eqref{eq:IntertwiningPropertyDiffRest} we can express $\pi_{\xi,r}(X)|_{\calE(\alpha;\alpha')}$ in terms of the cocycle using \eqref{eq:ReductionToCocycle}:
\begin{align*}
 \proj_{\calE'(\beta')}\circ\,T\circ\,\pi_{\xi,r}(X)|_{\calE(\alpha;\alpha')} &= \sum_{\substack{\beta\\(\alpha;\alpha')\leftrightarrow(\beta;\beta')}}\hspace{-.5cm}T\circ\,\proj_{\calE'(\beta;\beta')}\circ\,\pi_{\xi,r}(X)|_{\calE(\alpha;\alpha')}\\
 &= \hspace{.2cm}\tfrac{1}{2}\hspace{-.5cm}\sum_{\substack{\beta\\(\alpha;\alpha')\leftrightarrow(\beta;\beta')}}\hspace{-.5cm}(\sigma_\beta-\sigma_\alpha+2r)\cdot\left(T\circ\,\omega_{\alpha,\alpha'}^{\beta,\beta'}(X)\right)\\
 &= \hspace{.2cm}\tfrac{1}{2}\hspace{-.5cm}\sum_{\substack{\beta\\(\alpha;\alpha')\leftrightarrow(\beta;\beta')}}\hspace{-.5cm}(\sigma_\beta-\sigma_\alpha+2r)t_{\beta,\beta'}\cdot\left(R_{\beta,\beta'}\circ\,\omega_{\alpha,\alpha'}^{\beta,\beta'}(X)\right).
\end{align*}
Similarly we use \eqref{eq:ReductionToCocycle'} to obtain for the right hand side:
$$ \proj_{\calE'(\beta')}\circ\,\tau_{\xi',r'}(X)\circ\,T|_{\calE(\alpha;\alpha')} = \tfrac{1}{2}(\sigma'_{\beta'}-\sigma'_{\alpha'}+2r')t_{\alpha,\alpha'}\cdot\left(\omega_{\alpha'}^{\beta'}(X)\circ\, R_{\alpha,\alpha'}\right). $$
Inserting both expressions into the initial equation \eqref{eq:IntertwiningPropertyDiffRest} we obtain:

\begin{theorem}\label{thm:EigenvalueCondition}
Assume \eqref{eq:MF1} and \eqref{eq:MF2} and fix $R_{\alpha,\alpha'}:\calE(\alpha;\alpha')\to\calE'(\alpha')$ as in Section~\ref{sec:RelatingKKprimeTypes}. Let $\calV\subseteq\calE$ be a $(\frakg',K')$-submodule of $(\pi_{\xi,r})_\HC$. A linear map $T:\calV\to\calE'$ is an intertwining operator for $\pi_{\xi,r}$ and $\tau_{\xi',r'}$ if and only if
$$ T|_{\calE(\alpha;\alpha')} = t_{\alpha,\alpha'}\cdot R_{\alpha,\alpha'} \qquad \forall\,0\neq\calE(\alpha;\alpha')\subseteq\calV, $$
and for all $0\neq\calE(\alpha;\alpha')\subseteq \calV$ and $\calE'(\beta')\neq0$ we have
\begin{equation}
 \sum_{\substack{\beta\\(\alpha;\alpha')\leftrightarrow(\beta;\beta')}}\hspace{-.5cm}(\sigma_\beta-\sigma_\alpha+2r)t_{\beta,\beta'}\cdot\left(R_{\beta,\beta'}\circ\,\omega_{\alpha,\alpha'}^{\beta,\beta'}\right) = (\sigma'_{\beta'}-\sigma'_{\alpha'}+2r')t_{\alpha,\alpha'}\cdot\left(\omega_{\alpha'}^{\beta'}\circ\, R_{\alpha,\alpha'}\right).\label{eq:CharacterizationIntertwiners}
\end{equation}
\end{theorem}

\begin{remark}\label{rem:IntertwinersBetweenCompositionFactors}
Through the formulation of Theorem~\ref{thm:EigenvalueCondition} for any submodule $\calV$ of $(\pi_{\xi,r})_\HC$ one can also use \eqref{eq:CharacterizationIntertwiners} to describe intertwining operators from subquotients of $(\pi_{\xi,r})_\HC$ to $(\tau_{\xi',r'})_\HC$. In fact, if $\calV'\subseteq\calV\subseteq\calE$ are $(\frakg,K)$-submodules of $(\pi_{\xi,r})_\HC$ then any intertwining operator $\calV/\calV'\to\calE'$ for the actions $\pi_{\xi,r}$ and $\tau_{\xi',r'}$ is given by an intertwining operator $\calV\to\calE'$ which vanishes on $\calV'$.\\
A little more complicated is the description of intertwining operators into subquotients of $(\tau_{\xi',r'})_\HC$. Let $\calW'\subseteq\calW\subseteq\calE'$ be $(\frakg',K')$-submodules of $(\tau_{\xi',r'})_\HC$ and decompose $\calW=\calW'\oplus\calW''$ as $K'$-modules. Then a close examination of the arguments above shows that any operator $\calV\to\calW/\calW'$ which intertwines the actions of $\pi_{\xi,r}$ and $\tau_{\xi',r'}$ is given by a $K'$-intertwining linear map $T:\calV\to\calW''$ with $T|_{\calE(\alpha;\alpha')}=t_{\alpha,\alpha'}\cdot R_{\alpha,\alpha'}$ such that the relations \eqref{eq:CharacterizationIntertwiners} hold for any $0\neq\calE(\alpha;\alpha')\subseteq\calV$ and $0\neq\calE'(\beta')\subseteq\calW''$. Note that $t_{\alpha,\alpha'}=0$ whenever $\calE'(\alpha')\nsubseteq\calW''$.
\end{remark}

\subsection{Scalar identities}\label{sec:ScalarIdentities}

To extract from equation~\eqref{eq:CharacterizationIntertwiners} information on the constants $t_{\alpha,\alpha'}$ we have to transform it into a scalar identity. For this we assume additionally that
\begin{equation}
 \dim\Hom_{K'}(\fraks'_\CC\otimes\alpha',\beta') \leq 1 \qquad \forall\,0\neq\calE(\alpha;\alpha')\subseteq \calV,\calE(\beta')\neq0.\tag{MF3}\label{eq:MF3}
\end{equation}
This implies that the maps
$$ \eta_{\alpha,\alpha'}^{\beta,\beta'} = R_{\beta,\beta'}\circ\,\omega_{\alpha,\alpha'}^{\beta,\beta'}:\fraks'_\CC\otimes\calE'(\alpha,\alpha')\to\calE'(\beta') $$
are proportional to each other. If further the map
$$ \eta_{\alpha,\alpha'}^{\beta'} = \omega_{\alpha'}^{\beta'}\circ\, R_{\alpha,\alpha'}:\fraks'_\CC\otimes\calE'(\alpha;\alpha')\to\calE'(\beta'), $$
is non-zero then there exist constants $\lambda_{\alpha,\alpha'}^{\beta,\beta'}\neq0$ such that
$$ \eta_{\alpha,\alpha'}^{\beta,\beta'} = \lambda_{\alpha,\alpha'}^{\beta,\beta'}\eta_{\alpha,\alpha'}^{\beta'}. $$
We call $\lambda_{\alpha,\alpha'}^{\beta,\beta'}$ the \textit{proportionality constants}. In this case equation~\eqref{eq:CharacterizationIntertwiners} simplifies:

\begin{corollary}\label{cor:CharacterizationIntertwinersScalar}
Under the multiplicity-freeness assumption \eqref{eq:MF3} the identity \eqref{eq:CharacterizationIntertwiners} is equivalent to
\begin{equation}
 \sum_{\substack{\beta\\(\alpha;\alpha')\leftrightarrow(\beta;\beta')}} \lambda_{\alpha,\alpha'}^{\beta,\beta'}(\sigma_\beta-\sigma_\alpha+2r)t_{\beta,\beta'} = (\sigma'_{\beta'}-\sigma'_{\alpha'}+2r')t_{\alpha,\alpha'}.\label{eq:CharacterizationIntertwinersScalar}
\end{equation}
\end{corollary}

Whereas the constants $\sigma_\alpha$ and $\sigma_{\alpha'}'$ are easy to calculate using the highest weights of $\alpha$ and $\alpha'$ (see \cite{BOO96}), we do not have a general method to find the constants $\lambda_{\alpha,\alpha'}^{\beta,\beta'}$. Of course one can always try to compute the action of the cocycle on explicit $K$-finite vectors and decompose the result, but this turns out to be quite involved already in low rank cases. However, in some special cases the following information is enough to determine $\lambda_{\alpha,\alpha'}^{\beta,\beta'}$:

\begin{lemma}\label{lem:CalculationOfLambdas}
Assume that the elements $H\in\fraka$ and $H'\in\fraka'$ coincide. Let $\calE(\alpha;\alpha')\neq0$ and $\calE'(\beta')\neq0$ and assume that $R_{\alpha,\alpha'}=R_{\beta,\beta'}=\rest$ for all $\beta$ with $(\alpha;\alpha')\leftrightarrow(\beta;\beta')$. Then
\begin{align*}
 \sum_{\substack{\beta\\(\alpha;\alpha')\leftrightarrow(\beta;\beta')}}\lambda_{\alpha,\alpha'}^{\beta,\beta'} &= 1,\\
 \sum_{\substack{\beta\\(\alpha;\alpha')\leftrightarrow(\beta;\beta')}}(\sigma_\beta-\sigma_\alpha)\lambda_{\alpha,\alpha'}^{\beta,\beta'} &= \sigma_{\beta'}'-\sigma_{\alpha'}'+2(\rho-\rho').
\end{align*}
Here $\rho$ and $\rho'$ are identified with the numbers $\rho(H)$ and $\rho'(H')$.
\end{lemma}

\begin{proof}
For the first identity we note that $H=H'$ implies $\omega(X)|_{K'}=\omega'(X)$ for all $X\in\fraks'$. Hence
$$ R_{\beta,\beta'}\circ\,\omega(X) = \omega'(X)\circ\, R_{\alpha,\alpha'} \qquad \forall\,X\in\frakg' $$
which implies 
$$ \eta_{\alpha,\alpha'}^{\beta'} = \sum_{\substack{\beta\\(\alpha;\alpha')\leftrightarrow(\beta;\beta')}}\eta_{\alpha,\alpha'}^{\beta,\beta'} $$
and the claimed identity follows. For the second identity note that for $r+\rho=r'+\rho'$ the restriction operator $\rest:\calE\to\calE'$ is intertwining for $\pi_{\xi,r}$ and $\tau_{\xi',r'}$. Hence the identity \eqref{eq:CharacterizationIntertwinersScalar} is satisfied with $t_{\alpha,\alpha'}=1$ for all $\calE(\alpha;\alpha')\neq0$. Eliminating $r$ and $r'$ gives the desired formula.
\end{proof}

\begin{remark}
The knowledge of any intertwining operator $T:(\pi_{\xi,r})_\HC\to(\tau_{\xi',r'})_\HC$ and the corresponding numbers $t_{\alpha,\alpha'}$ provides an additional identity for the constants $\lambda_{\alpha,\alpha'}^{\beta,\beta'}$ just as in the proof of Lemma~\ref{lem:CalculationOfLambdas} for the restriction operator $T=\rest$ with $r+\rho=r'+\rho'$ and $t_{\alpha,\alpha'}=1$.
\end{remark}

\subsection{Automatic continuity}

In this section we study the question of whether $(\frakg',K')$-inter\-twining operators $(\pi_{\xi,r})_\HC\to(\tau_{\xi',r'})_\HC$ between Harish-Chandra modules extend to $G'$-intertwining operators $\pi_{\xi,r}\to\tau_{\xi',r'}$ between the smooth representations, i.e. whether the natural injective map
$$ \Hom_{G'}(\pi_{\xi,r}|_{G'},\tau_{\xi',r'}) \to \Hom_{(\frakg',K')}((\pi_{\xi,r})_\HC|_{(\frakg',K')},(\tau_{\xi',r'})_\HC) $$
is an isomorphism. It is expected (see Kobayashi~\cite[Remark 10.2~(4)]{Kob14}) that this is true if the space $(G\times G')/\diag(G')$ is real spherical. Statements of this type are also known as ``automatic continuity theorems'' since they imply continuity with respect to the smooth topologies of every intertwining operator between the algebraic Harish-Chandra modules. We provide a criterion to show automatic continuity in the context of this paper.

Fix a Haar measure $dk'$ on $K'$. Then the non-degenerate bilinear pairing
$$ \calE(K';\xi'|_{M'\cap K'})\times\calE(K';\xi'^\vee|_{M'\cap K'})\to\CC,\,(f_1,f_2)\mapsto\int_{K'}\langle f_1(k'),f_2(k')\rangle\,dk' $$
is invariant under $\tau_{\xi',r'}\otimes\tau_{\xi'^\vee,-r'}$ for any $r'\in\CC$, where $\xi'^\vee$ denotes the contragredient representation of $\xi'$ on the dual space $V_{\xi'}^\vee$. Using this pairing we identify $\tau_{\xi',r'}$ with a subrepresentation of the contragredient representation $\tau_{\xi'^\vee,-r'}^\vee$ of $\tau_{\xi'^\vee,-r'}$, which is realized on the topological dual space $\calE(K';\xi'^\vee|_{M'\cap K'})^\vee$ carrying the weak-$\star$ topology.

\begin{lemma}\label{lem:AutomaticSmoothness}
Every continuous $G'$-intertwining operator $T:\pi_{\xi,r}\to\tau_{\xi'^\vee,-r'}^\vee$ maps into $\tau_{\xi',r'}$ and defines a continuous $G'$-intertwining operator $T:\pi_{\xi,r}\to\tau_{\xi',r'}$.
\end{lemma}

\begin{proof}
Let $T:\calE(K;\xi|_{M\cap K})\to\calE(K';\xi'^\vee|_{M'\cap K'})^\vee$ be a continuous linear operator which is $G'$-intertwining for $\pi_{\xi,r}$ and $\tau_{\xi'^\vee,-r'}^\vee$. Then $T$ induces a continuous linear functional
$$ \overline{T}:\calE(K;\xi|_{M\cap K})\,\widehat{\otimes}\,\calE(K';\xi'^\vee|_{M'\cap K'})\to\CC, $$
which is invariant under the diagonal action of $K'$. The left hand side is naturally isomorphic to $\calE(K\times K';(\xi\otimes\xi'^\vee)|_{(M\cap K)\times(M'\cap K')})$. Composing with the surjective continuous linear operator
\begin{multline*}
 \flat:C^\infty(K\times K';V_\xi\otimes V_{\xi'}^\vee) \to \calE(K\times K';(\xi\otimes\xi'^\vee)|_{(M\cap K)\times(M'\cap K')})\\
 \flat F(k,k') = \int_{M\cap K}\int_{M'\cap K'} (\xi(m)\otimes\xi'(m')^\vee)F(km,k'm')\,dm'\,dm
\end{multline*}
we obtain a functional
$$ K_T:=\overline{T}\circ\flat:C^\infty(K\times K';V_\xi\otimes V_{\xi'}^\vee)\to\CC, $$
i.e. a distribution on $K\times K'$ with values in $V_\xi\otimes V_{\xi'}^\vee$. (This is basically the Schwartz kernel of the operator $T$, avoiding distribution sections of vector bundles.) The distribution $K_T$ is invariant under the diagonal action of $K'$ from the left and equivariant under the action of $(M\cap K)\times(M'\cap K')$ from the right. We define a distribution $\tilde{K}_T$ on $K$ with values in $V_\xi\otimes V_{\xi'}^\vee$, i.e. a continuous linear functional on $C^\infty(K;V_\xi\otimes V_{\xi'}^\vee)$, by
$$ \langle\tilde{K}_T,\phi\rangle := \langle K_T(x,x'),\phi(x'^{-1}x)\rangle. $$
Then for $\phi\in\calE(K;\xi|_{M\cap K})$ and $\psi\in\calE(K';\xi'^\vee|_{M'\cap K'})$ we have
\begin{align*}
 \langle T\phi,\psi\rangle &= \langle K_T,\phi\otimes\psi\rangle = \int_{K'}\langle K_T(x,x'),\phi(k'x)\otimes\psi(k'x')\rangle\,dk'\\
 &= \langle K_T(x,x'),\int_{K'}\phi(k'x)\otimes\psi(k'x')\,dk'\rangle = \langle K_T(x,x'),\int_{K'}\phi(k'x'^{-1}x)\otimes\psi(k')\,dk'\rangle\\
 &= \int_{K'}\langle K_T(x,x'),\phi(k'x'^{-1}x)\otimes\psi(k')\rangle\,dk' = \int_{K'}\langle\tilde{K}_T,\phi(k'\blank)\otimes\psi(k')\rangle\,dk'.
\end{align*}
This implies that for any $\lambda\in V_{\xi'}^\vee$:
$$ \langle\lambda,T\phi(k')\rangle = \langle\tilde{K}_T,\phi(k'\blank)\otimes\lambda\rangle $$
which shows that $T\phi\in C^\infty(K';V_{\xi'})$. That $T\phi\in\calE(K';\xi'|_{M'\cap K'})$ easily follows from the equivariance property of $\tilde{K}_T$ with respect to $M\cap K$ and $M'\cap K'$. Finally, continuity of the thus defined operator $T:\pi_{\xi,r}\to\tau_{\xi',r'}$ follows from the continuity of the functional $\tilde{K}_T$ on $C^\infty(K;V_\xi\otimes V_{\xi'}^\vee)$ and the proof is complete.
\end{proof}

Fix invariant inner products on the representation $\xi|_{M\cap K}$ resp. $\xi'|_{M'\cap K'}$ and let $\|\blank\|$ resp. $\|\blank\|'$ denote the corresponding $L^2$-norm on $L^2(K\times_{M\cap K}\xi)$ resp. $L^2(K'\times_{M'\cap K'}\xi')$. These norms induce norms on each $K'$-type $\calE(\alpha;\alpha')$ resp. $\calE'(\alpha')$. Write $\|R_{\alpha,\alpha'}\|_{L^2\to L^2}$ for the operator norm of $R_{\alpha,\alpha'}:\calE(\alpha;\alpha')\to\calE'(\alpha')$ with respect to the $L^2$-norms.

For any $F\in L^2(K\times_{M\cap K}\xi)$ write
$$ F=\sum_{\alpha\in\widehat{K}}F_\alpha $$
with $F_\alpha\in\calE(\alpha)$. Then the sequence $\{\|F_\alpha\|\}_\alpha$ belongs to $\ell^2(\widehat{K})$, the space of square-summable sequences. We identify the set $\widehat{K}$ resp. $\widehat{K'}$ with the corresponding weight lattice so that it becomes a subset of a finite-dimensional vector space. Denote by $|\blank|$ resp. $|\blank|'$ a norm on this finite-dimensional vector space. It is known that $F\in\calE(K;\xi_{M\cap K})$ if and only if the sequence $\{\|F_\alpha\|\}_\alpha$ belongs to $s(\widehat{K})$, the space of rapidly decreasing sequences, i.e. those that are still bounded if multiplied with any power $|\alpha|^N$. Moreover, $\calE(K;\xi^\vee|_{M\cap K})^\vee$ is identified with all formal expansions $F=\sum_\alpha F_\alpha$ where $\{\|F_\alpha\|\}_\alpha$ belongs to $s'(\widehat{K})$, the space of tempered sequences, i.e. those that grow at most at the rate of $|\alpha|^N$ for some $N\in\NN$.

\begin{proposition}\label{prop:HCtoInfty}
A $(\frakg',K')$-intertwining operator $T:(\pi_{\xi,r})_\HC\to(\tau_{\xi',r'})_\HC$ with $T|_{\calE(\alpha;\alpha')}=t_{\alpha,\alpha'}\cdot R_{\alpha,\alpha'}$ extends to a continuous $G'$-intertwining operator $\pi_{\xi,r}\to\tau_{\xi',r'}$ if both $t_{\alpha,\alpha'}$ and $\|R_{\alpha,\alpha'}\|_{L^2\to L^2}$ are of at most polynomial growth in $\alpha$ and $\alpha'$.
\end{proposition}

\begin{proof}
By Lemma~\ref{lem:AutomaticSmoothness} it suffices to show that $T$ extends to a continuous $G'$-intertwining operator $\pi_{\xi,r}\to(\tau_{\xi'^\vee,-r'})'$. Let $F\in\pi_{\xi,r}$, then $F=\sum_\alpha F_\alpha$ with $\{\|F_\alpha\|\}_\alpha$ a sequence in $s(\widehat{K})$. We have $TF=\sum_{\alpha'}(TF)_{\alpha'}$ with
$$ (TF)_{\alpha'} = \sum_{\alpha}t_{\alpha,\alpha'}\cdot R_{\alpha,\alpha'}F_\alpha. $$
By the assumptions $|t_{\alpha,\alpha'}|\leq C_1(1+|\alpha|+|\alpha'|)^{N_1}$ and $\|R_{\alpha,\alpha'}\|_{L^2\to L^2}\leq C_2(1+|\alpha|+|\alpha'|)^{N_2}$ for some $C_1,C_2>0$ and $N_1,N_2\in\NN$. Further, since $\|F_\alpha\|\in s(\widehat{K})$, for every $N\in\NN$ there exists $C>0$ such that $\|F_\alpha\|\leq C(1+|\alpha|)^{-N}$. Hence, we have for any $\alpha'$:
$$ \|(TF)_{\alpha'}\|' \leq CC_1C_2\sum_\alpha(1+|\alpha|+|\alpha'|)^{N_1+N_2}(1+|\alpha|)^{-N}. $$
Choosing $N$ large enough this is uniformly bounded by a constant times $(1+|\alpha'|)^{N_1+N_2}$ and hence $\|(TF)_{\alpha'}\|'\in s'(\widehat{K'})$ so that $TF\in\calE(K';\xi'^\vee|_{M'\cap K'})^\vee$. This shows that $T$ extends to a $G'$-intertwining operator $\pi_{\xi,r}\to\tau_{\xi'^\vee,-r'}^\vee$. Continuity of this operator also follows by the above estimates.
\end{proof}

\section{Rank one orthogonal groups}\label{sec:ExampleReal}

In this section we apply our method to classify symmetry breaking operators for rank one orthogonal groups. Let $n\geq3$ and consider the indefinite orthogonal group $G=\upO(1,n)$ of $(n+1)\times(n+1)$ real matrices leaving the standard bilinear form on $\RR^{n+1}$ of signature $(1,n)$ invariant. The subgroup $G'\subseteq G$ of matrices fixing the last standard basis vector $e_{n+1}$ is isomorphic to $\upO(1,n-1)$.

\subsection{$K$-types}\label{sec:RealKTypes}

We fix $K=\upO(1)\times \upO(n)$ and choose
$$ H = \left(\begin{array}{ccc}0&1&\\1&0&\\&&\0_{n-1}\end{array}\right) $$
so that $P=MAN$ with $M=\Delta \upO(1)\times \upO(n-1)$ where $\Delta \upO(1)=\{\diag(x,x):x\in \upO(1)\}$. Note that $\rho=\frac{n-1}{2}$. Then $K$ acts transitively on $S^{n-1}$ via $\diag(\varepsilon,k)\cdot x=\varepsilon kx$, $\varepsilon\in \upO(1)$, $k\in \upO(n)$, $x\in S^{n-1}$, and $M$ is the stabilizer subgroup of the first standard basis vector $e_1\in S^{n-1}$ whence $K/M\cong S^{n-1}$. The subgroup $G'=\upO(1,n-1)$ is embedded into $G$ such that $K'=\upO(1)\times \upO(n-1)$ and $P'=G'\cap P=M'A'N'$ with $A'=A$ and $M'=\Delta \upO(1)\times \upO(n-2)$. Then $K'/M'=S^{n-2}$, viewed as the equator in $K/M=S^{n-1}\subseteq\RR^n$ given by $x_n=0$. Further we have $\nu=\nu'$ and $\rho'=\frac{n-2}{2}$.

Let $\xi=\1$, $\xi'=\1$ be the trivial representations of $M$ and $M'$ and abbreviate $\pi_r=\pi_{\xi,r}$ and $\tau_{r'}=\tau_{\xi',r'}$. As $K$-modules resp. $K'$-modules we have
$$ \calE = \bigoplus_{\alpha=0}^\infty \underbrace{\sgn^{\alpha}\boxtimes\,\calH^\alpha(\RR^n)}_{\calE(\alpha)=}, \qquad \calE' = \bigoplus_{\alpha'=0}^\infty \underbrace{\sgn^{\alpha'}\boxtimes\,\calH^{\alpha'}(\RR^{n-1})}_{\calE'(\alpha')=}, $$
so that \eqref{eq:MF1} is satisfied. Further, each $K$-type decomposes by \eqref{eq:BranchingRealSphericalHarmonics} into $K'$-types as follows:
$$ \left.\left(\sgn^\alpha\boxtimes\,\calH^\alpha(\RR^n)\right)\right|_{K'} \simeq \bigoplus_{0\leq\alpha'\leq\alpha} \left(\sgn^\alpha\boxtimes\,\calH^{\alpha'}(\RR^{n-1})\right), $$
and hence \eqref{eq:MF2} holds. Comparing the sign representations of the $\upO(1)$-factor of $K'$ we find that $\Hom_{K'}(\calE(\alpha)|_{K'},\calE'(\alpha'))\neq0$ if and only if $\alpha-\alpha'\in2\ZZ$. In this case formulas \eqref{eq:ExplicitBranchingRealSphericalHarmonics} and \eqref{eq:GegenbauerValue0} show that the restriction operator
$$ R_{\alpha,\alpha'}=\rest|_{\calE(\alpha;\alpha')}:\calE(\alpha;\alpha')\to\calE'(\alpha') $$
is an isomorphism. Hence the restriction $T_{\alpha,\alpha'}=T|_{\calE(\alpha;\alpha')}$ of a $K'$-intertwining operator $T:\calE\to\calE'$ is given by $T_{\alpha,\alpha'}=t_{\alpha,\alpha'}R_{\alpha,\alpha'}$ for $\alpha-\alpha'\in2\NN$ and $T_{\alpha,\alpha'}=0$ else. The $K$- and $K'$-types are illustrated in Diagram~\ref{Ktypes}.

\begin{diagram}
\setlength{\unitlength}{4pt}
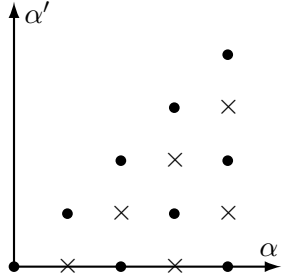
\begin{picture}(26,26)
\thicklines
\put(0,0){\vector(1,0){25}}
\put(0,0){\vector(0,1){25}}

\multiput(0,0)(10,0){3}{\circle*{1}}
\multiput(5,5)(10,0){2}{\circle*{1}}
\multiput(10,10)(10,0){2}{\circle*{1}}
\multiput(15,15)(10,0){1}{\circle*{1}}
\multiput(20,20)(10,0){1}{\circle*{1}}

\multiput(4,-0.6)(10,0){2}{$\times$}
\multiput(9,4.4)(10,0){2}{$\times$}
\multiput(14,9.4)(10,0){1}{$\times$}
\multiput(19,14.4)(10,0){1}{$\times$}
\put(23,1){$\alpha$}
\put(1,23){$\alpha'$}
\end{picture}
\vspace{10pt}
\subcaption{\vbox{\hsize22pc
\begin{legend}
\item[\raise3pt\hbox{\circle*{1}}] $K'$-types $\calE(\alpha;\alpha')$ with $\alpha-\alpha'\in2\ZZ$,
\item[\hspace{1.858cm}\hbox{$\times$}] $K'$-types $\calE(\alpha;\alpha')$ with $\alpha-\alpha'\in2\ZZ+1$.
\endgraf
\end{legend}}}
\caption[]{}\label{Ktypes}
\end{diagram}

\subsection{Proportionality constants}

The eigenvalues of the spectrum generating operator on the $K$-types are simply the eigenvalues of the Laplacian on $S^{n-1}$ and given by (see \cite[Section 3.a]{BOO96})
$$ \sigma_\alpha = \alpha(\alpha+n-2), \qquad \sigma'_{\alpha'} = \alpha'(\alpha'+n-3). $$
We identify $\fraks\cong\RR^n$ via
$$ \RR^n \to \fraks, \qquad y \mapsto X_y=\left(\begin{array}{cc}0&y^t\\y&\0_n\end{array}\right). $$
Then $\fraks'\cong\RR^{n-1}$, embedded in $\RR^n$ as the first $n-1$ coordinates. Since $\fraks'_\CC\simeq\CC^{n-1}$ is a weight multiplicity-free $K'$-module, \eqref{eq:MF3} holds and we can use Corollary~\ref{cor:CharacterizationIntertwinersScalar}. To compute the proportionality constants $\lambda_{\alpha,\alpha'}^{\beta,\beta'}$ we use Lemma~\ref{lem:CalculationOfLambdas} which applies to this situation, because $H=H'$ and $R_{\alpha,\alpha'}=\rest$.  The cocycle $\omega$ is given by
$$ \omega(X_y)(x) = y^tx, \qquad x\in S^{n-1},y\in\RR^n. $$
Using \eqref{eq:DecompositionOfMultRealSphericalHarmonics} it is easy to see that for fixed $0\leq\alpha'\leq\alpha$:
$$ (\alpha;\alpha')\leftrightarrow(\beta;\beta') \quad \Leftrightarrow \quad |\alpha-\beta|=|\alpha'-\beta'|=1. $$
By Lemma~\ref{lem:CalculationOfLambdas} we have the following equations for $\lambda_{\alpha,\alpha'}^{\beta,\beta'}$: For $\beta'=\alpha'+1$ we obtain
$$
\arraycolsep=1.4pt\def\arraystretch{1.5}
\begin{array}{rcrcl}
 \lambda_{\alpha,\alpha'}^{\alpha+1,\alpha'+1}&+&\lambda_{\alpha,\alpha'}^{\alpha-1,\alpha'+1} &=& 1,\\
 (2\alpha+n-1)\lambda_{\alpha,\alpha'}^{\alpha+1,\alpha'+1}&-&(2\alpha+n-3)\lambda_{\alpha,\alpha'}^{\alpha-1,\alpha'+1} &=& 2\alpha'+n-1,
\end{array}
$$
which gives
\begin{equation*}
 \lambda_{\alpha,\alpha'}^{\alpha+1,\alpha'+1} = \frac{\alpha+\alpha'+n-2}{2\alpha+n-2}, \qquad \lambda_{\alpha,\alpha'}^{\alpha-1,\alpha'+1} = \frac{\alpha-\alpha'}{2\alpha+n-2},
\end{equation*}
and for $\beta'=\alpha'-1$ we get
$$
\arraycolsep=1.4pt\def\arraystretch{1.5}
\begin{array}{rcrcl}
 \lambda_{\alpha,\alpha'}^{\alpha+1,\alpha'-1}&+&\lambda_{\alpha,\alpha'}^{\alpha-1,\alpha'-1} &=& 1,\\
 (2\alpha+n-1)\lambda_{\alpha,\alpha'}^{\alpha+1,\alpha'-1}&-&(2\alpha+n-3)\lambda_{\alpha,\alpha'}^{\alpha-1,\alpha'-1} &=& -2\alpha'-n+5,
\end{array}
$$
implying
\begin{equation*}
 \lambda_{\alpha,\alpha'}^{\alpha+1,\alpha'-1} = \frac{\alpha-\alpha'+1}{2\alpha+n-2}, \qquad \lambda_{\alpha,\alpha'}^{\alpha-1,\alpha'-1} = \frac{\alpha+\alpha'+n-3}{2\alpha+n-2}.
\end{equation*}
We remark that the constants $\lambda_{\alpha,\alpha'}^{\beta,\beta'}$ can in this case also be obtained by computing the action of $\omega(X)$ on explicit $K$-finite vectors using \eqref{eq:ExplicitBranchingRealSphericalHarmonics} and recurrence relations for the Gegenbauer polynomials. With the explicit form of the constants $\lambda_{\alpha,\alpha'}^{\beta,\beta'}$ Corollary~\ref{cor:CharacterizationIntertwinersScalar} now provides the following characterization of symmetry breaking operators:

\begin{theorem}\label{thm:RealCharacterizationIntertwiners}
An operator $T:\calE\to\calE'$ is intertwining for $\pi_r$ and $\tau_{r'}$ if and only if
$$ T|_{\calE(\alpha;\alpha')} = \begin{cases}t_{\alpha,\alpha'}\cdot\rest|_{\calE(\alpha;\alpha')} & \mbox{for $\alpha-\alpha'\in2\ZZ$,}\\0 & \mbox{for $\alpha-\alpha'\in2\ZZ+1$,}\end{cases} $$
with numbers $t_{\alpha,\alpha'}$ satisfying
\begin{multline}
 (2\alpha+n-2)(2r'+2\alpha'+n-2)t_{\alpha,\alpha'} = (\alpha+\alpha'+n-2)(2r+2\alpha+n-1)t_{\alpha+1,\alpha'+1}\\
 + (\alpha-\alpha')(2r-2\alpha-n+3)t_{\alpha-1,\alpha'+1}\label{eq:RealRel1}
\end{multline}
and
\begin{multline}
 (2\alpha+n-2)(2r'-2\alpha'-n+4)t_{\alpha,\alpha'} = (\alpha-\alpha'+1)(2r+2\alpha+n-1)t_{\alpha+1,\alpha'-1}\\
 + (\alpha+\alpha'+n-3)(2r-2\alpha-n+3)t_{\alpha-1,\alpha'-1}.\label{eq:RealRel2}
\end{multline}
\end{theorem}

We view these two relations as triangles connecting three vertices in the $K$-type picture (see Diagram~\ref{fig:RealHooks}).

\begin{diagram}
\setlength{\unitlength}{4pt}
\begin{picture}(48,15)
\thicklines
\put(1,7){\line(1,-1){6}}
\put(15,7){\line(-1,-1){6}}
\put(1,8){\line(1,0){14}}
\put(33,1){\line(1,1){6}}
\put(47,1){\line(-1,1){6}}
\put(33,0){\line(1,0){14}}
\put(0,8){\circle*{1}}
\put(8,0){\circle*{1}}
\put(16,8){\circle*{1}}
\put(32,0){\circle*{1}}
\put(40,8){\circle*{1}}
\put(48,0){\circle*{1}}
\put(4.6,-3){$(\alpha,\alpha')$}
\put(-10,10){$(\alpha-1,\alpha'+1)$}
\put(10.5,10){$(\alpha+1,\alpha'+1)$}
\put(36.6,10){$(\alpha,\alpha')$}
\put(22,-3){$(\alpha-1,\alpha'-1)$}
\put(42.5,-3){$(\alpha+1,\alpha'-1)$}
\end{picture}
\vspace{.3cm}
\caption[]{The relations \eqref{eq:RealRel1} and \eqref{eq:RealRel2}}\label{fig:RealHooks}
\end{diagram}

Note that if $r\notin-\rho-\NN$ then $2r+2\alpha+n-1\neq0$ for all $\alpha$ and hence one can define $t_{\alpha+1,\alpha'+1}$ in terms of $t_{\alpha,\alpha'}$ and $t_{\alpha-1,\alpha'+1}$ using \eqref{eq:RealRel1} and do similarly for $t_{\alpha+1,\alpha'-1}$ using \eqref{eq:RealRel2}. If $r=-\rho-i\in-\rho-\NN$ and $\alpha=i$ the coefficient $(2r+2\alpha+n-1)$ vanishes and \eqref{eq:RealRel1} and \eqref{eq:RealRel2} reduce to identities involving only two terms. We indicate this by drawing a vertical line between $i$ and $i+1$ indicating that one cannot `step' from the left hand side to the right hand side (see Diagram~\ref{fig:RealBarriers}). Similarly we have that if $r'\notin-\rho'-\NN$ then $2r'+2\alpha'+n-2\neq0$ for all $\alpha'$ and we can define $t_{\alpha,\alpha'}$ in terms of $t_{\alpha\pm1,\alpha'+1}$ using \eqref{eq:RealRel1}. If $r'=-\rho'-j\in-\rho'-\NN$ and $\alpha'=j$ we obtain a horizontal line between $j$ and $j+1$ as barrier, indicating that we cannot `step' from the part above this line to the part below. Note that if there is a vertical resp. horizontal barrier like this the coefficient $(2r-2\alpha-n+3)$ resp. $(2r'-2\alpha'-n+4)$ never vanishes and one can `step' in the other direction, namely from right to left resp. from the part below the line to the part above.

\begin{diagram}
\setlength{\unitlength}{4pt}
\begin{picture}(15,15)
\thicklines
\put(1,7){\line(1,-1){6}}
\put(15,7){\line(-1,-1){6}}
\put(1,8){\line(1,0){14}}
{\color{red}
\put(12,-4){\line(0,1){16}}}
\put(0,8){\circle*{1}}
\put(8,0){\circle*{1}}
\put(16,8){\circle*{1}}
\put(3,-3){$(\alpha,\alpha')$}
\put(-10,10){$(\alpha-1,\alpha'+1)$}
\put(14.5,10){$(\alpha+1,\alpha'+1)$}
\end{picture}
\hspace{4cm}
\begin{picture}(15,15)
\thicklines
\put(1,7){\line(1,-1){6}}
\put(15,7){\line(-1,-1){6}}
\put(1,8){\line(1,0){14}}
{\color{red}
\put(-10,4){\line(1,0){36}}}
\put(0,8){\circle*{1}}
\put(8,0){\circle*{1}}
\put(16,8){\circle*{1}}
\put(4.6,-3){$(\alpha,\alpha')$}
\put(-10,10){$(\alpha-1,\alpha'+1)$}
\put(10.5,10){$(\alpha+1,\alpha'+1)$}
\end{picture}
\vspace{.3cm}
\caption[]{Barriers for $r=-\rho-i$ and $r'=-\rho'-j$}\label{fig:RealBarriers}
\end{diagram}

\subsection{Multiplicities}\label{sec:MultiplicitiesReal}

The $(\frakg,K)$-module $(\pi_r)_\HC$ is reducible if and only if $r\in\pm(\rho+\NN)$. More precisely, for $r=-\rho-i$ the module $(\pi_r)_\HC$ contains a unique non-trivial finite-dimensional $(\frakg,K)$-submodule $\calF(i)\subseteq\calE$ with $K$-types $\calE(\alpha)$, $0\leq\alpha\leq i$. Its quotient $\calT(i)=\calE/\calF(i)$ is irreducible and can be identified with the unique non-trivial $(\frakg,K)$-submodule of $(\pi_{-r})_\HC$. Similarly we denote for $r'=-\rho'-j$ by $\calF'(j)$ the unique finite-dimensional $(\frakg',K')$-submodule of $(\tau_{r'})_\HC$ and by $\calT'(j)$ its irreducible quotient. Let
\begin{align*}
 \Leven &= \{(r,r'):r=-\rho-i,r'=-\rho'-j,i-j\in2\NN\},\\
 \Lodd &= \{(r,r'):r=-\rho-i,r'=-\rho'-j,i-j\in2\NN+1\}.
\end{align*}
This notation agrees with the notation used by Kobayashi--Speh~\cite{KS13}.

\begin{theorem}\label{thm:RealMultiplicities}
\begin{enumerate}
\item The multiplicities between spherical principal series of $G$ and $G'$ are given by
$$ m((\pi_r)_\HC,(\tau_{r'})_\HC) = \begin{cases}1&\mbox{for $(r,r')\in\CC^2\setminus\Leven$,}\\2&\mbox{for $(r,r')\in\Leven$.}\end{cases} $$
\item For $i,j\in\NN$ the multiplicities $m(\calV,\calW)$ between subquotients are given by
\vspace{.2cm}
\begin{center}
 \begin{tabular}{c|cc}
  \diagonal{.1em}{.72cm}{$\calV$}{$\calW$} & $\calF'(j)$ & $\calT'(j)$ \\
  \hline
  $\calF(i)$ & $1$ & $0$\\
  $\calT(i)$ & $0$ & $1$\\
  \multicolumn{3}{c}{for $i-j\in2\NN$,}
 \end{tabular}
 \qquad
 \begin{tabular}{c|cc}
  \diagonal{.1em}{.72cm}{$\calV$}{$\calW$} & $\calF'(j)$ & $\calT'(j)$ \\
  \hline
  $\calF(i)$ & $0$ & $0$\\
  $\calT(i)$ & $1$ & $0$\\
  \multicolumn{3}{c}{otherwise.}
 \end{tabular}
\end{center}
\vspace{.2cm}
\end{enumerate}
\end{theorem}

To prove Theorem~\ref{thm:RealMultiplicities} we study how the relations \eqref{eq:RealRel1} and \eqref{eq:RealRel2} determine the numbers $t_{\alpha,\alpha'}$. We first consider the diagonal $\alpha=\alpha'$. Relation~\eqref{eq:RealRel1} then simplifies to
\begin{equation}
 (2r'+2\alpha+n-2)t_{\alpha,\alpha} = (2r+2\alpha+n-1)t_{\alpha+1,\alpha+1}.\label{eq:RealRelDiagonal}
\end{equation}
This immediately yields:

\begin{lemma}\label{lem:RealDiagonalSequences}
\begin{enumerate}
\item For $(r,r')\in\CC^2\setminus(\Leven\cup\Lodd)$ the space of diagonal sequences $(t_{\alpha,\alpha})_\alpha$ satisfying \eqref{eq:RealRelDiagonal} has dimension $1$. Any generator $(t_{\alpha,\alpha})_\alpha$ satisfies:
\begin{enumerate}
\item for $r\notin-\rho-\NN$, $r'\notin-\rho'-\NN$:
$$ t_{\alpha,\alpha} \neq 0 \quad \forall\,\alpha\in\NN, $$
\item for $r=-\rho-i\in-\rho-\NN$, $r'\notin-\rho'-\NN$:
$$ t_{\alpha,\alpha} = 0 \quad \forall\,\alpha\leq i \qquad \mbox{and} \qquad t_{\alpha,\alpha} \neq 0 \quad \forall\,\alpha> i, $$
\item for $r\notin-\rho-\NN$, $r'=-\rho'-j\in-\rho'-\NN$:
$$ t_{\alpha,\alpha} \neq 0 \quad \forall\,\alpha\leq j \qquad \mbox{and} \qquad t_{\alpha,\alpha} = 0 \quad \forall\,\alpha> j, $$
\item for $r=-\rho-i\in-\rho-\NN$, $r'=-\rho'-j\in-\rho'-\NN$ with $i<j$:
$$ t_{\alpha,\alpha} \neq 0 \quad \forall\,i<\alpha\leq j \qquad \mbox{and} \qquad t_{\alpha,\alpha} = 0 \quad \mbox{else}. $$
\end{enumerate}
\item For $(r,r')=(-\rho-i,-\rho'-j)\in(\Leven\cup\Lodd)$, the space of diagonal sequences $(t_{\alpha,\alpha})_\alpha$ satisfying \eqref{eq:RealRelDiagonal} has dimension $2$. It has a basis $(t'_{\alpha,\alpha})_\alpha$, $(t''_{\alpha,\alpha})_\alpha$ with the properties
\begin{align*}
 & t'_{\alpha,\alpha} \neq 0 \quad \forall\,\alpha\leq j, && t'_{\alpha,\alpha} = 0 \quad \forall\,\alpha>j,\\
 & t''_{\alpha,\alpha} = 0 \quad \forall\,\alpha\leq i, && t''_{\alpha,\alpha} \neq 0 \quad \forall\,\alpha>i.
\end{align*}
\end{enumerate}
\end{lemma}

Next we investigate how a diagonal sequence $(t_{\alpha,\alpha})_\alpha$ satisfying \eqref{eq:RealRelDiagonal} can be extended to a sequence $(t_{\alpha,\alpha'})_{(\alpha,\alpha')}$ satisfying \eqref{eq:RealRel1} and \eqref{eq:RealRel2}.

\begin{lemma}\label{lem:RealDiagonalExtensionGeneric}
Let $(r,r')\in\CC^2\setminus(\Leven\cup\Lodd)$. Then every diagonal sequence $(t_{\alpha,\alpha})_\alpha$ satisfying \eqref{eq:RealRelDiagonal} has a unique extension to a sequence $(t_{\alpha,\alpha'})_{(\alpha,\alpha')}$ satisfying \eqref{eq:RealRel1} and \eqref{eq:RealRel2}.
\end{lemma}

\begin{proof}
\textbf{Step 1.} We first treat the case $r\notin-\rho-\NN$. In this case the coefficients $(2r+2\alpha+n-1)$ in \eqref{eq:RealRel1} and \eqref{eq:RealRel2} never vanish. We now extend the diagonal sequence $(t_{\alpha,\alpha})_\alpha$ inductively to a sequence $(t_{\alpha,\alpha'})_{\alpha-\alpha'\leq2k}$ with $k\in\NN$ which satisfies \eqref{eq:RealRel1} for $(\alpha,\alpha')$ with $\alpha-\alpha'\leq2k$ and \eqref{eq:RealRel2} for $(\alpha,\alpha')$ with $\alpha-\alpha'\leq2k-2$ (i.e. the two relations hold whenever the corresponding triangles in Diagram~\ref{fig:RealHooks} are contained in the region $\alpha-\alpha'\leq2k$).
\begin{diagram}
\setlength{\unitlength}{4pt}
\begin{picture}(41,30)
\thicklines
\put(0,0){\vector(1,0){40}}
\put(0,0){\vector(0,1){30}}
\put(37,20){\vector(1,-1){5}}
\dashline{1}(13,-2)(40,25)
{\color{blue}
\put(20,10){\line(1,1){5}}
\put(30,10){\line(-1,1){5}}
\put(20,10){\line(1,0){10}}}
\multiput(0,0)(10,0){2}{\circle*{1}}
\multiput(5,5)(10,0){2}{\circle*{1}}
\multiput(10,10)(10,0){2}{\circle*{1}}
\multiput(15,15)(10,0){2}{\circle*{1}}
\multiput(20,20)(10,0){2}{\circle*{1}}
\multiput(25,25)(10,0){2}{\circle*{1}}
\multiput(20,0)(10,0){2}{\circle{.8}}
\multiput(25,5)(10,0){2}{\circle{.8}}
\multiput(30,10)(10,0){1}{\circle{.8}}
\multiput(35,15)(10,0){1}{\circle{.8}}
\put(28,28){$\alpha-\alpha'\leq2k$}
\put(38,1){$\alpha$}
\put(1,28){$\alpha'$}
\end{picture}
\vspace{10pt}
\subcaption{\vbox{\hsize30pc
\begin{legend}
\item[\raise3pt\hbox{\circle*{1}}] $K'$-types $\calE(\alpha;\alpha')$ with $\alpha-\alpha'\leq2k$ ($t_{\alpha,\alpha'}$ already defined)
\item[\raise3pt\hbox{\circle{.8}}] $K'$-types $\calE(\alpha;\alpha')$ with $\alpha-\alpha'>2k$ ($t_{\alpha,\alpha'}$ yet to define)
\endgraf
\end{legend}}}
\caption[]{}
\end{diagram}
For $k=0$ the diagonal sequence we start with satisfies these assumptions. For the induction step $k\to k+1$ let $\alpha-\alpha'=2k$ and define $t_{\alpha+1,\alpha'-1}$ and $t_{\alpha+2,\alpha'}$ using \eqref{eq:RealRel2} (the blue triangles in Diagram~\ref{fig:RealRelConsistency}) in terms of $t_{\alpha-1,\alpha'-1}$, $t_{\alpha,\alpha'}$ and $t_{\alpha+1,\alpha'+1}$. This is possible, because $2r+2\alpha+n-1\neq0$ for all $\alpha$ and hence the corresponding coefficients in \eqref{eq:RealRel2} are non-zero. Now we have to prove that \eqref{eq:RealRel1} holds for $(\alpha+1,\alpha'-1)$ (the red triangle). This can be done by an elementary calculation using the blue triangles that are by definition valid as well as the green triangles that are valid by the induction assumption.
\begin{diagram}
\setlength{\unitlength}{4pt}
\begin{picture}(34,19)
\thicklines
\put(1,9){\line(1,1){6}}
\put(1,7){\line(1,-1){6}}
\put(15,9){\line(-1,1){6}}
\put(15,7){\line(-1,-1){6}}
\put(17,9){\line(1,1){6}}
\put(17,7){\line(1,-1){6}}
\put(31,9){\line(-1,1){6}}
\put(31,7){\line(-1,-1){6}}
\dashline{1}(17,0)(33,16)
{\color{green}
\polygon*(4.5,11.5)(11.5,11.5)(8,15)
\polygon*(4.5,4.5)(11.5,4.5)(8,1)
\polygon*(16,9)(12.5,12.5)(19.5,12.5)
\color{blue}
\polygon*(20.5,11.5)(27.5,11.5)(24,15)
\polygon*(16,7)(12.5,3.5)(19.5,3.5)
\color{red}
\polygon*(20.5,4.5)(27.5,4.5)(24,1)
}
\put(0,8){\circle*{1}}
\put(8,0){\circle*{1}}
\put(8,16){\circle*{1}}
\put(16,8){\circle*{1}}
\put(24,0){\circle{.8}}
\put(24,16){\circle*{1}}
\put(32,8){\circle{.8}}
\put(-13,7.5){$(\alpha-2,\alpha')$}
\put(-4,18){$(\alpha-1,\alpha'+1)$}
\put(-4,-3){$(\alpha-1,\alpha'-1)$}
\put(17.5,7.5){$(\alpha,\alpha')$}
\put(20,18){$(\alpha+1,\alpha'+1)$}
\put(20,-3){$(\alpha+1,\alpha'-1)$}
\put(34,7.5){$(\alpha+2,\alpha')$}
\end{picture}
\vspace{.3cm}
\caption[]{}\label{fig:RealRelConsistency}
\vspace{-.3cm}
\end{diagram}
Hence this extends the diagonal sequence $(t_{\alpha,\alpha})_\alpha$ to a sequence $(t_{\alpha,\alpha'})_{0\leq\alpha'\leq\alpha}$ satisfying \eqref{eq:RealRel1} and \eqref{eq:RealRel2}. Since the relations were used to extend the diagonal sequence this extension is unique.\par

\textbf{Step 2.} Next assume $r=-\rho-i\in-\rho-\NN$ and $r'\notin-\rho'-\NN$. Then the coefficient $(2r+2\alpha+n-1)$ vanishes if and only if $\alpha=i$. We can therefore use the technique in Step 1 to extend the upper part $(t_{\alpha,\alpha})_{\alpha>i}$ of the diagonal sequence to a sequence $(t_{\alpha,\alpha'})_{i<\alpha'\leq\alpha}$ in the region $\alpha'>i$. Next we extend the sequence $(t_{\alpha,\alpha})_{\alpha'>i}$ inductively to a sequence $(t_{\alpha,\alpha'})_{\alpha'> i-k}$ with $k=0,\ldots,i+1$ which satisfies \eqref{eq:RealRel1} for $(\alpha,\alpha')$ with $\alpha'>i-k$ and \eqref{eq:RealRel2} for $(\alpha,\alpha')$ with $\alpha'>i-k+1$ (i.e. the two relations hold whenever the corresponding triangles in Diagram~\ref{fig:RealHooks} are contained in the region $\alpha'>i-k$).
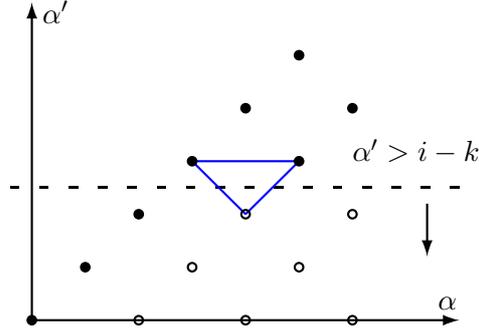
\begin{diagram}
\setlength{\unitlength}{4pt}
\begin{picture}(41,30)
\thicklines
\put(0,0){\vector(1,0){40}}
\put(0,0){\vector(0,1){30}}
\put(37,11){\vector(0,-1){5}}
\dashline{1}(-2,12.5)(40,12.5)
{\color{blue}
\put(20,10){\line(1,1){5}}
\put(20,10){\line(-1,1){5}}
\put(15,15){\line(1,0){10}}}
\multiput(0,0)(5,5){3}{\circle*{1}}
\multiput(10,0)(10,0){3}{\circle{.8}}
\multiput(15,5)(10,0){2}{\circle{.8}}
\multiput(20,10)(10,0){2}{\circle{.8}}
\multiput(15,15)(10,0){2}{\circle*{1}}
\multiput(20,20)(10,0){2}{\circle*{1}}
\multiput(25,25)(10,0){1}{\circle*{1}}
\put(30,15){$\alpha'>i-k$}
\put(38,1){$\alpha$}
\put(1,28){$\alpha'$}
\end{picture}
\vspace{10pt}
\subcaption{\vbox{\hsize30pc
\begin{legend}
\item[\raise3pt\hbox{\circle*{1}}] $K'$-types $\calE(\alpha;\alpha')$ with $\alpha'>i-k$ ($t_{\alpha,\alpha'}$ already defined)
\item[\raise3pt\hbox{\circle{.8}}] $K'$-types $\calE(\alpha;\alpha')$ with $\alpha'\leq i-k$ ($t_{\alpha,\alpha'}$ yet to define)
\endgraf
\end{legend}}}
\caption[]{}
\end{diagram}
For $k=0$ the sequence we obtained using Step 1 satisfies these assumptions by Step 1. For the induction step $k\to k+1$ let $\alpha'=i-k+1$ and define $t_{\alpha-1,\alpha'-1}$ and $t_{\alpha+1,\alpha'-1}$ using \eqref{eq:RealRel1} (the blue triangles in Diagram~\ref{fig:RealRelConsistency2}) in terms of $t_{\alpha-2,\alpha'}$, $t_{\alpha,\alpha'}$ and $t_{\alpha+2,\alpha'}$. This is possible, because $r'\notin-\rho'-\NN$ and hence the corresponding coefficient $(2r'+2\alpha'+n-2)$ in \eqref{eq:RealRel1} never vanishes. Now we have to prove that \eqref{eq:RealRel2} holds for $(\alpha,\alpha')$ (the red triangle) which is done in a similar fashion as in Step 1 using the green triangles. This finishes Step 2.
\begin{diagram}
\setlength{\unitlength}{4pt}
\begin{picture}(34,19)
\thicklines
\put(1,9){\line(1,1){6}}
\put(1,7){\line(1,-1){6}}
\put(15,9){\line(-1,1){6}}
\put(15,7){\line(-1,-1){6}}
\put(17,9){\line(1,1){6}}
\put(17,7){\line(1,-1){6}}
\put(31,9){\line(-1,1){6}}
\put(31,7){\line(-1,-1){6}}
\dashline{1}(-1,5.5)(33,5.5)
{\color{green}
\polygon*(4.5,11.5)(11.5,11.5)(8,15)
\polygon*(16,9)(12.5,12.5)(19.5,12.5)
\polygon*(20.5,11.5)(27.5,11.5)(24,15)
\color{blue}
\polygon*(4.5,4.5)(11.5,4.5)(8,1)
\polygon*(20.5,4.5)(27.5,4.5)(24,1)
\color{red}
\polygon*(16,7)(12.5,3.5)(19.5,3.5)
}
\put(0,8){\circle*{1}}
\put(8,0){\circle{.8}}
\put(8,16){\circle*{1}}
\put(16,8){\circle*{1}}
\put(24,0){\circle{.8}}
\put(24,16){\circle*{1}}
\put(32,8){\circle*{1}}
\put(-13,7.5){$(\alpha-2,\alpha')$}
\put(-4,18){$(\alpha-1,\alpha'+1)$}
\put(-4,-3){$(\alpha-1,\alpha'-1)$}
\put(18.5,7.5){$(\alpha,\alpha')$}
\put(20,18){$(\alpha+1,\alpha'+1)$}
\put(20,-3){$(\alpha+1,\alpha'-1)$}
\put(34,7.5){$(\alpha+2,\alpha')$}
\end{picture}
\vspace{.3cm}
\caption[]{}\label{fig:RealRelConsistency2}
\vspace{-.3cm}
\end{diagram}

\textbf{Step 3.} Now let $r=-\rho-i\in-\rho-\NN$ and $r'=-\rho'-j\in-\rho'-\NN$ with $i,j\in\NN$, $j>i$. Note that to carry out Step 2 we only need that $(2r'+2\alpha'+n-2)\neq0$ for $\alpha'\leq i$. This is satisfied since
$$ 2r'+2\alpha'+n-2=2(\alpha'-j)<2(\alpha'-i)\leq0 $$
by assumption. Hence the technique in Step 2 carries over to this case.
\end{proof}

\begin{lemma}\label{lem:RealDiagonalExtensionSingular}
Let $(r,r')=(-\rho-i,-\rho'-j)$, $i,j\in\NN$.
\begin{enumerate}
\item For $(r,r')\in\Leven$ every diagonal sequence $(t_{\alpha,\alpha})_\alpha$ satisfying \eqref{eq:RealRelDiagonal} has a unique extension to a sequence $(t_{\alpha,\alpha'})_{(\alpha,\alpha')}$ satisfying \eqref{eq:RealRel1} and \eqref{eq:RealRel2}.
\item For $(r,r')\in\Lodd$ any sequence $(t_{\alpha,\alpha'})_{\alpha,\alpha'}$ satisfying \eqref{eq:RealRel1} and \eqref{eq:RealRel2} has the property $t_{\alpha,\alpha'}=0$ for $\alpha\leq i$ or $\alpha'>j$. Conversely, for any choice of $t_{i+1,j}\in\CC$ there exists a unique extension to a sequence $(t_{\alpha,\alpha'})_{(\alpha,\alpha')}$ satisfying \eqref{eq:RealRel1} and \eqref{eq:RealRel2}.
\end{enumerate}
\end{lemma}

\begin{proof}
\begin{enumerate}
\item First Step 1 and Step 2 in the proof of Lemma~\ref{lem:RealDiagonalExtensionGeneric} extend a diagonal sequence $(t_{\alpha,\alpha})_\alpha$ uniquely to the range $\{(\alpha,\alpha'):\alpha\leq i\mbox{ or }\alpha'>j\}$. This extension satisfies $t_{\alpha,\alpha'}=0$ whenever $j<\alpha'\leq\alpha\leq i$. Next one can use \eqref{eq:RealRel2} for $(\alpha,\alpha')=(i+1,j+1)$ to define $t_{i+2,j}$ in terms of $t_{i,j}$ and $t_{i+1,j+1}$ (the blue triangle in Diagram~\ref{KtypesPropagationLeven}). Inductively, using \eqref{eq:RealRel2} for $(\alpha,\alpha')=(i+2k+1,j+1)$, $k=0,1,2,\ldots$, the values of $t_{i+2k+2,j}$ are determined for all $k$. In the next step the technique from Step 2 in the proof of Lemma~\ref{lem:RealDiagonalExtensionGeneric} is used to inductively define $t_{\alpha,\alpha'}$ for $\alpha>i$ and $\alpha'=j-k$, $k=0,\ldots,j$ (the red triangle).
\begin{diagram}
\setlength{\unitlength}{4pt}
\begin{picture}(44,40)
\thicklines
\put(0,0){\vector(1,0){40}}
\put(0,0){\vector(0,1){40}}
\put(22.5,-2){\line(0,1){42.5}}
\put(-2,12.5){\line(1,0){42.5}}
{\color{blue}
\put(20,10){\line(1,1){5}}
\put(20,10){\line(1,0){10}}
\put(30,10){\line(-1,1){5}}
\put(32,10){\vector(1,0){6}}}
\multiput(0,0)(10,0){3}{\circle*{1}}
\put(30,0){\circle{.8}}
\multiput(5,5)(10,0){2}{\circle*{1}}
\put(25,5){\circle{.8}}
\put(35,5){\circle{.8}}
\multiput(10,10)(10,0){2}{\circle*{1}}
\put(30,10){\circle{.8}}
\multiput(15,15)(10,0){3}{\circle*{1}}
\multiput(20,20)(10,0){2}{\circle*{1}}
\multiput(25,25)(10,0){2}{\circle*{1}}
\multiput(30,30)(10,0){1}{\circle*{1}}
\multiput(35,35)(10,0){1}{\circle*{1}}
\put(38,1){$\alpha$}
\put(1,38){$\alpha'$}
\put(19.5,-3.5){$i$}
\put(23.5,-3.5){$i+1$}
\put(-3,9.5){$j$}
\put(-6.5,14.5){$j+1$}
\put(-1.5,1){$*$}
\put(3.5,6){$*$}
\put(8.5,11){$*$}
\put(8.5,1){$*$}
\put(13.5,6){$*$}
\put(18.5,11){$*$}
\put(18.5,1){$*$}
\put(13.5,16){$0$}
\put(18.5,21){$0$}
\put(23.5,26){$*$}
\put(28.5,31){$*$}
\put(33.5,36){$*$}
\put(23.5,16){$*$}
\put(28.5,21){$*$}
\put(33.5,26){$*$}
\put(33.5,16){$*$}
\end{picture}
\begin{picture}(40,40)
\thicklines
\put(0,0){\vector(1,0){40}}
\put(0,0){\vector(0,1){40}}
\put(22.5,-2){\line(0,1){42.5}}
\put(-2,12.5){\line(1,0){42.5}}
{\color{red}
\put(30,10){\line(-1,0){10}}
\put(20,10){\line(1,-1){5}}
\put(30,10){\line(-1,-1){5}}
\put(30,8){\vector(0,-1){6}}}
\multiput(0,0)(10,0){3}{\circle*{1}}
\put(30,0){\circle{.8}}
\multiput(5,5)(10,0){2}{\circle*{1}}
\put(25,5){\circle{.8}}
\put(35,5){\circle{.8}}
\multiput(10,10)(10,0){3}{\circle*{1}}
\multiput(15,15)(10,0){3}{\circle*{1}}
\multiput(20,20)(10,0){2}{\circle*{1}}
\multiput(25,25)(10,0){2}{\circle*{1}}
\multiput(30,30)(10,0){1}{\circle*{1}}
\multiput(35,35)(10,0){1}{\circle*{1}}
\put(38,1){$\alpha$}
\put(1,38){$\alpha'$}
\put(19.5,-3.5){$i$}
\put(23.5,-3.5){$i+1$}
\put(-1.5,1){$*$}
\put(3.5,6){$*$}
\put(8.5,11){$*$}
\put(8.5,1){$*$}
\put(13.5,6){$*$}
\put(18.5,11){$*$}
\put(18.5,1){$*$}
\put(13.5,16){$0$}
\put(18.5,21){$0$}
\put(23.5,26){$*$}
\put(28.5,31){$*$}
\put(28.5,11){$*$}
\put(33.5,36){$*$}
\put(23.5,16){$*$}
\put(28.5,21){$*$}
\put(33.5,26){$*$}
\put(33.5,16){$*$}
\end{picture}
\vspace{20pt}
\subcaption{\vbox{\hsize25pc
\begin{legend}
\item[\raise3pt\hbox{\circle*{1}}] $K'$-types $\calE(\alpha;\alpha')$ with $t_{\alpha,\alpha'}$ already defined
\item[\raise3pt\hbox{\circle{.8}}] $K'$-types $\calE(\alpha;\alpha')$ with $t_{\alpha,\alpha'}$ yet to define
\endgraf
\end{legend}}}
\caption[]{}\label{KtypesPropagationLeven}
\end{diagram}
That all relations \eqref{eq:RealRel1} and \eqref{eq:RealRel2} are satisfied within the four quadrants in Diagram~\ref{KtypesPropagationLeven} is clear from the arguments in Step 1 and Step 2 in the proof of Lemma~\ref{lem:RealDiagonalExtensionGeneric}. That these relations are also satisfied at the edges between the quadrants holds either by definition or since all terms vanish.

\item Let $(t_{\alpha,\alpha'})_{\alpha,\alpha'}$ be a sequence satisfying \eqref{eq:RealRel1} and \eqref{eq:RealRel2}. Note that Lemma~\ref{lem:RealDiagonalSequences}~(2) already implies $t_{\alpha,\alpha}=0$ for $j<\alpha\leq i$. Then by Step 1 in the proof of Lemma~\ref{lem:RealDiagonalExtensionGeneric} we have $t_{\alpha,\alpha'}=0$ whenever $j<\alpha'\leq\alpha\leq i$ (the black zeroes in Diagram~\ref{KtypesPropagationLodd}). We first show inductively that $t_{i-2k-1,j}=0$ for $k=0,\ldots,\frac{i-j-1}{2}$ (the red zeroes). To show the statement for $k=0$ consider the relation \eqref{eq:RealRel2} for $(\alpha,\alpha')=(i,j+1)$. By the previous considerations $t_{\alpha,\alpha'}=0$ and further the coefficient $(2r+2\alpha+n-1)$ of $t_{\alpha+1,\alpha'-1}$ vanishes. Hence $t_{\alpha-1,\alpha'-1}=t_{i-1,j}=0$. For the induction step assume $t_{i-2k-1,j}=0$ and consider the relation \eqref{eq:RealRel2} for $(\alpha,\alpha')=(i-2k-2,j+1)$. Then $t_{\alpha,\alpha'}=t_{\alpha+1,\alpha'-1}=0$ and therefore $t_{\alpha-1,\alpha'-1}=t_{i-2(k+1)-1,j}=0$. Thus we have showed $t_{j,j}=0$. But in view of \eqref{eq:RealRelDiagonal} this yields $t_{\alpha,\alpha}=0$ for $\alpha\leq j$.
\begin{diagram}
\setlength{\unitlength}{4pt}
\begin{picture}(44,40)
\thicklines
\put(0,0){\vector(1,0){40}}
\put(0,0){\vector(0,1){40}}
\put(27.5,-2){\line(0,1){42.5}}
\put(-2,12.5){\line(1,0){42.5}}
{\color{red}
\put(20,10){\line(1,1){5}}
\put(20,10){\line(1,0){10}}
\put(29.9,9.9){\line(-1,1){5}}
\put(18,10){\vector(-1,0){6}}
\put(18.5,7.5){$0$}
\put(8.5,7.5){$0$}
\put(3.5,2.5){$0$}
\put(-1.5,-2.5){$0$}
\color{blue}
\put(30.1,10.1){\line(-1,1){5}}
\put(30,10){\line(1,1){5}}
\put(25,15){\line(1,0){10}}
\put(35,17){\vector(0,1){6}}
\put(35.5,16){$0$}
\put(30.5,21){$0$}
\put(35.5,26){$0$}
\put(30.5,31){$0$}
\put(35.5,36){$0$}}
\multiput(0,0)(10,0){4}{\circle*{1}}
\multiput(5,5)(10,0){4}{\circle*{1}}
\multiput(10,10)(10,0){3}{\circle*{1}}
\multiput(15,15)(10,0){3}{\circle*{1}}
\multiput(20,20)(10,0){2}{\circle*{1}}
\multiput(25,25)(10,0){2}{\circle*{1}}
\multiput(30,30)(10,0){1}{\circle*{1}}
\multiput(35,35)(10,0){1}{\circle*{1}}
\put(38,1){$\alpha$}
\put(1,38){$\alpha'$}
\put(24.5,-3.5){$i$}
\put(28.5,-3.5){$i+1$}
\put(-3,9.5){$j$}
\put(-6.5,14.5){$j+1$}
\put(23.5,16){$0$}
\put(18.5,21){$0$}
\put(23.5,26){$0$}
\put(13.5,16){$0$}
\end{picture}
\begin{picture}(40,40)
\thicklines
\put(0,0){\vector(1,0){40}}
\put(0,0){\vector(0,1){40}}
\put(27.5,-2){\line(0,1){42.5}}
\put(-2,12.5){\line(1,0){42.5}}
\multiput(0,0)(10,0){3}{\circle*{1}}
\put(30,0){\circle{.8}}
\multiput(5,5)(10,0){3}{\circle*{1}}
\put(35,5){\circle{.8}}
\multiput(10,10)(10,0){3}{\circle*{1}}
\multiput(15,15)(10,0){3}{\circle*{1}}
\multiput(20,20)(10,0){2}{\circle*{1}}
\multiput(25,25)(10,0){2}{\circle*{1}}
\multiput(30,30)(10,0){1}{\circle*{1}}
\multiput(35,35)(10,0){1}{\circle*{1}}
\put(38,1){$\alpha$}
\put(1,38){$\alpha'$}
\put(24.5,-3.5){$i$}
\put(28.5,-3.5){$i+1$}
\put(23.5,16){$0$}
\put(-1.5,1){$0$}
\put(3.5,6){$0$}
\put(8.5,11){$0$}
\put(8.5,1){$0$}
\put(13.5,6){$0$}
\put(18.5,11){$0$}
\put(18.5,1){$0$}
\put(23.5,6){$0$}
\put(13.5,16){$0$}
\put(18.5,21){$0$}
\put(23.5,26){$0$}
\put(28.5,31){$0$}
\put(33.5,36){$0$}
\put(28.5,21){$0$}
\put(33.5,26){$0$}
\put(33.5,16){$0$}
\put(28.5,11){$*$}
\put(32,10){\vector(1,0){6}}
\end{picture}
\vspace{20pt}
\subcaption{\vbox{\hsize20pc
}}
\caption[]{}\label{KtypesPropagationLodd}
\end{diagram}
In a similar way one uses \eqref{eq:RealRel1} and \eqref{eq:RealRel2} for $(\alpha,\alpha')=(i+1,j+2k)$, $k=0,\ldots,\frac{i-j+1}{2}$ to show that $t_{i+1,i+1}=0$ and hence $t_{\alpha,\alpha}=0$ for all $\alpha>i$. From the vanishing of the diagonal the techniques in Step 1 and Step 2 in the proof of Lemma~\ref{lem:RealDiagonalExtensionGeneric} yield $t_{\alpha,\alpha'}=0$ whenever $\alpha\leq i$ or $\alpha'>j$.\par
Now let $t_{i+1,j}\in\CC$ be given and put $t_{\alpha,\alpha'}=0$ whenever $\alpha\leq i$ or $\alpha'>j$. Then \eqref{eq:RealRel1} and \eqref{eq:RealRel2} are trivially satisfied whenever all three terms are defined. Further, using Step 1 and Step 2 it is again easy to see that this sequence has a unique extension $(t_{\alpha,\alpha'})_{\alpha,\alpha'}$ with the required properties.\qedhere
\end{enumerate}
\end{proof}

\begin{proof}[{Proof of Theorem~\ref{thm:RealMultiplicities}}]
\begin{enumerate}
\item Let first $(r,r')\in\CC^2\setminus(\Leven\cup\Lodd)$. Then by Lemma~\ref{lem:RealDiagonalSequences} the space of diagonal sequences satisfying \eqref{eq:RealRelDiagonal} is one-dimensional and each such sequence gives by Lemma~\ref{lem:RealDiagonalExtensionGeneric} rise to a unique extension $(t_{\alpha,\alpha'})_{(\alpha,\alpha')}$ satisfying \eqref{eq:RealRel1} and \eqref{eq:RealRel2}. Hence, by Theorem~\ref{thm:RealCharacterizationIntertwiners} the multiplicity is $1$. Similarly we obtain multiplicity $2$ for $(r,r')\in\Leven$ using Lemma~\ref{lem:RealDiagonalExtensionSingular}~(1). For $(r,r')\in\Lodd$ the multiplicity statement is contained in Lemma~\ref{lem:RealDiagonalExtensionSingular}~(2).
\item We first consider the case $\calV=\calF(i)$ and $\calW=\calF(j)$. Then any intertwining operator in $\Hom_{(\frakg',K')}(\calV|_{(\frakg',K')},\calW)$ corresponds to an intertwining operator $T:(\pi_r)_\HC\to(\tau_{r'})_\HC$ for $r=\rho+i$ and $r'=-\rho'-j$ such that $T|_{\calE(\alpha)}=0$ for all $\alpha>i$ and $T(\calE)\subseteq\calF'(j)$. This implies that $T$ is given by a sequence $(t_{\alpha,\alpha'})_{\alpha,\alpha'}$ with $t_{\alpha,\alpha'}=0$ if either $\alpha>i$ or $\alpha'>j$. By part (1) the space of intertwining operators $T:(\pi_r)_\HC\to(\tau_{r'})_\HC$ is one-dimensional, and using Lemma~\ref{lem:RealDiagonalSequences}~(1)~(c) and Step~1 in the proof of Lemma~\ref{lem:RealDiagonalExtensionGeneric} it is easy to see that this operator satisfies the conditions on $t_{\alpha,\alpha'}$ if and only if $i-j\in2\NN$. Hence $m(\calF(i),\calF'(j))=1$ for $i-j\in2\NN$ and $=0$ else. Similar considerations for $r=-\rho-i$ and $r'=\rho'+j$ show that $m(\calT(i),\calT'(j))=1$ for $i-j\in2\NN$ and $=0$ else.\par
Now let $\calV=\calT(i)$ and $\calW=\calF'(j)$. Then $m(\calV,\calW)\neq0$ if and only if there exists a non-trivial sequence $(t_{\alpha,\alpha'})_{\alpha,\alpha'}$ satisfying \eqref{eq:RealRel1} and \eqref{eq:RealRel2} for $r=-\rho-i$ and $r'=-\rho'-j$ such that $t_{\alpha,\alpha'}=0$ whenever $\alpha\leq i$ or $\alpha'>j$. First assume $j>i$, then by part (1) there exists a unique sequence $(t_{\alpha,\alpha'})_{\alpha,\alpha'}$, and by Lemma~\ref{lem:RealDiagonalSequences}~(1)~(d) and Step~3 in the proof of Lemma~\ref{lem:RealDiagonalExtensionGeneric} it is easy to see that for this sequence $t_{\alpha,\alpha'}=0$ if either $\alpha\leq i$ or $\alpha'>j$. Hence $m(\calT(i),\calF'(j))=1$ in this case. Next assume $j\leq i$, then by Lemma~\ref{lem:RealDiagonalSequences}~(2) and Lemma~\ref{lem:RealDiagonalExtensionSingular} there can only exist a sequence $(t_{\alpha,\alpha'})_{\alpha,\alpha'}$ with the above properties if $i-j\in2\NN+1$. This shows the claimed formulas for $m(\calT(i),\calF'(j))$. That $m(\calF(i),\calT'(j))=0$ for any $i,j$ follows easily by similar considerations.\qedhere
\end{enumerate}
\end{proof}

\subsection{Explicit formula for the spectral function}\label{sec:RealSpectralFunction}

From the relations \eqref{eq:RealRel1} and \eqref{eq:RealRel2} one can deduce an explicit \textit{spectral function} $(t_{\alpha,\alpha'}(r,r'))_{0\leq\alpha'\leq\alpha}$, i.e. a set of solutions to the relations for all $r,r'\in\CC$ depending meromorphically on $r$ and $r'$:

\begin{proposition}\label{prop:RealExplicitEigenvalues}
For $(\alpha,\alpha')\in\NN$ with $\alpha-\alpha'\in2\ZZ$ the numbers
\begin{multline}
 t_{\alpha,\alpha'}(r,r') = \sum_{k=0}^\infty\frac{2^{4k}\Gamma(\frac{\alpha+\alpha'+n-2}{2}+k)\Gamma(\frac{\alpha-\alpha'+2}{2})}{(2k)!\Gamma(\frac{\alpha+\alpha'+n-2}{2})\Gamma(\frac{\alpha-\alpha'+2}{2}-k)}\\
 \times \frac{\Gamma(r+\rho)\Gamma(r'+\rho'+\alpha')\Gamma(\frac{2r'+2r+1}{4}+k)\Gamma(\frac{2r'-2r+3}{4})}{\Gamma(r+\rho+\alpha'+2k)\Gamma(r'+\rho')\Gamma(\frac{2r'+2r+1}{4})\Gamma(\frac{2r'-2r+3}{4}-k)}\label{eq:ExplicitFormulaReal}
\end{multline}
are rational functions in $r$ and $r'$ satisfying \eqref{eq:RealRel1} and \eqref{eq:RealRel2}. They are normalized to $t_{0,0}\equiv1$.
\end{proposition}

\begin{proof}
First note that since $\alpha-\alpha'\in2\ZZ$ the number $\frac{\alpha-\alpha'+2}{2}-k$ is a negative integer for $k\gg0$ and hence the sum is actually finite for each fixed pair $(\alpha,\alpha')$. It is also easy to see that each summand is a rational function in $r$ and $r'$. A short calculation shows that for each $k\in\NN$ the term
$$ \frac{\Gamma(\frac{\alpha+\alpha'+n-2}{2}+k)\Gamma(\frac{\alpha-\alpha'+2}{2})\Gamma(r'+\rho'+\alpha')}{\Gamma(\frac{\alpha+\alpha'+n-2}{2})\Gamma(\frac{\alpha-\alpha'+2}{2}-k)\Gamma(r+\rho+\alpha'+2k)} $$
solves \eqref{eq:RealRel1}. If we further make the ansatz
$$ t_{\alpha,\alpha'} = \sum_{k=0}^\infty b_k\frac{\Gamma(\frac{\alpha+\alpha'+n-2}{2}+k)\Gamma(\frac{\alpha-\alpha'+2}{2})\Gamma(r'+\rho'+\alpha')}{\Gamma(\frac{\alpha+\alpha'+n-2}{2})\Gamma(\frac{\alpha-\alpha'+2}{2}-k)\Gamma(r+\rho+\alpha'+2k)} $$
with $b_k=b_k(r,r')$ not depending on $\alpha$ and $\alpha'$ then we find that \eqref{eq:RealRel2} holds if and only if
\begin{multline*}
 \sum_{k=0}^\infty b_k \frac{\Gamma(\frac{\alpha+\alpha'+n-2}{2}+k-1)\Gamma(\frac{\alpha-\alpha'+2}{2})\Gamma(r'+\rho'+\alpha'-1)}{\Gamma(\frac{\alpha+\alpha'+n-2}{2})\Gamma(\frac{\alpha-\alpha'+2}{2}-k+1)\Gamma(r+\rho+\alpha'+2k)}\\
 \times \Big[(2r'+2r+4k+1)(2r'-2r-4k-1)\left(\tfrac{\alpha-\alpha'+2}{2}-k\right)\left(\tfrac{\alpha+\alpha'+n-2}{2}+k-1\right)\\
 -2k(2k-1)\left(r+\rho+\alpha'+2k-1\right)\left(r+\rho+\alpha'+2k-2\right)\Big] = 0.
\end{multline*}
Substituting $k-1$ for $k$ in the first summand in the brackets gives the condition
\begin{multline*}
 \sum_{k=1}^\infty \frac{\Gamma(\frac{\alpha+\alpha'+n-2}{2}+k-1)\Gamma(\frac{\alpha-\alpha'+2}{2})\Gamma(r'+\rho'+\alpha'-1)}{\Gamma(\frac{\alpha+\alpha'+n-2}{2})\Gamma(\frac{\alpha-\alpha'+2}{2}-k+1)\Gamma(r+\rho+\alpha'+2k-2)}\\
 \times \Big[(2r'+2r+4k-3)(2r'-2r-4k+3)b_{k-1}-2k(2k-1)b_k\Big] = 0
\end{multline*}
which holds if
$$ 2k(2k-1)b_k = (2r'+2r+4k-3)(2r'-2r-4k+3)b_{k-1}. $$
This recurrence relation has the solution
$$ b_k = c\cdot\frac{2^{4k}\Gamma(\frac{2r'+2r+1}{4}+k)}{(2k)!\Gamma(\frac{2r'-2r+3}{4}-k)} $$
with $c=c(r,r')$ not depending on $k$. Finally $t_{0,0}\equiv1$ implies
\begin{equation*}
 c = \frac{\Gamma(r+\rho)\Gamma(\frac{2r'-2r+3}{4})}{\Gamma(r'+\rho')\Gamma(\frac{2r'+2r+1}{4})}.\qedhere
\end{equation*}
\end{proof}

\begin{corollary}\label{cor:RealRenormalizations}
\begin{enumerate}
\item The renormalized numbers
$$ t^{(1)}_{\alpha,\alpha'}(r,r') = \frac{1}{\Gamma(r+\rho)}t_{\alpha,\alpha'}(r,r') $$
are holomorphic in $(r,r')\in\CC^2$ for every $\alpha,\alpha'\in\NN$, $\alpha-\alpha'\in2\NN$. Further, $t^{(1)}_{\alpha,\alpha'}(r,r')=0$ for all $\alpha,\alpha'$ if and only if $(r,r')\in\Leven$.
\item Fix $r'=-\rho'-j$, $j\in\NN$, then the renormalized numbers
$$ t^{(2)}_{\alpha,\alpha'}(r,r') = \frac{\Gamma(\frac{(r+\rho)-(r'+\rho')}{2})}{\Gamma(r+\rho)}t_{\alpha,\alpha'}(r,r') $$
are holomorphic in $r\in\CC$ for every $\alpha,\alpha'\in\NN$, $\alpha-\alpha'\in2\NN$. We have $t^{(2)}_{\alpha,\alpha'}(r,r')\equiv0$ for $\alpha'>j$. Further, for every $r\in\CC$ there exists a pair $(\alpha,\alpha')$ with $t^{(2)}_{\alpha,\alpha'}(r,r')\neq0$.
\item Fix $N\in\NN$ and let $r'+\rho'=r+\rho+2N$, then the renormalized numbers
$$ t^{(3)}_{\alpha,\alpha'}(r,r') = \frac{\Gamma(r'+\rho')}{\Gamma(r+\rho)}t_{\alpha,\alpha'}(r,r') $$
are holomorphic in $r\in\CC$ for every $\alpha,\alpha'\in\NN$, $\alpha-\alpha'\in2\NN$. Further, for every $r\in\CC$ there exists $\alpha_0\in\NN$ such that $t^{(3)}_{\alpha,\alpha}(r,r')\neq0$ for $\alpha\geq\alpha_0$.
\end{enumerate}
\end{corollary}

\begin{proof}
\begin{enumerate}
\item We can write
$$ \hspace{.5cm}t^{(1)}_{\alpha,\alpha'}(r,r') = (r'+\rho')_{\alpha'}\sum_{k=0}^{\frac{\alpha-\alpha'}{2}} \frac{2^{4k}(\frac{\alpha+\alpha'+n-2}{2})_k(-\frac{\alpha-\alpha'}{2})_k(\frac{2r'+2r+1}{4})_k(\frac{2r-2r'+1}{4})_k}{(2k)!\Gamma(r+\rho+\alpha'+2k)}, $$
where $(\lambda)_n=\lambda(\lambda+1)\cdots(\lambda+n-1)$ denotes the Pochhammer symbol. This expression is obviously holomorphic in $(r,r')\in\CC^2$. Now assume $t^{(1)}_{\alpha,\alpha'}(r,r')=0$ for all $\alpha,\alpha'$. For $\alpha=\alpha'$ we have $(-\frac{\alpha-\alpha'}{2})_k=0$ for $k>0$ and hence
$$ t^{(1)}_{\alpha,\alpha}(r,r') = \frac{(r'+\rho')_\alpha}{\Gamma(r+\rho+\alpha)} $$
which vanishes for all $\alpha\in\NN$ if and only if $r+\rho=-i$ and $r'+\rho'=-j$ with $j\leq i$. We claim that $i-j\in2\NN$. In fact, if $i-j\in2\NN+1$ then for $(\alpha,\alpha')=(i+1,j)$ only the summand for $k=\frac{i-j+1}{2}$ is non-zero and hence $t^{(1)}_{\alpha,\alpha'}(r,r')\neq0$, a contradiction. Therefore $i-j\in2\NN$ which means $(r,r')\in\Leven$.\par Conversely assume $r+\rho=-i$, $r'+\rho'=-j$ with $i-j\in2\NN$. Then in each summand at least one of the three factors
$$ \hspace{.5cm}(\tfrac{2r-2r'+1}{4})_k=(-\tfrac{i-j}{2})_k, \quad (r'+\rho')_{\alpha'}=(-j)_{\alpha'}, \quad \tfrac{1}{\Gamma(r+\rho+\alpha'+2k)}=\tfrac{1}{\Gamma(-i+\alpha'+2k)} $$
vanishes and hence $t^{(1)}_{\alpha,\alpha'}(r,r')=0$ for all $\alpha,\alpha'$.
\item We can write
$$ \hspace{1.5cm}t^{(2)}_{\alpha,\alpha'}(r,r') = (-j)_{\alpha'}\sum_{k=0}^{\frac{\alpha-\alpha'}{2}} \frac{2^{4k}(\frac{\alpha+\alpha'+n-2}{2})_k(-\frac{\alpha-\alpha'}{2})_k(\frac{2r-2j-n+3}{4})_k\Gamma(\frac{2r+2j+n-1}{4}+k)}{(2k)!\Gamma(r+\rho+\alpha'+2k)}, $$
as a meromorphic function of $r$. Then $t^{(2)}_{\alpha,\alpha'}(r,r')\equiv0$ for $\alpha'>j$. Further, for $\alpha'\leq j$ each pole $r$ of the factor $\Gamma(\frac{2r+2j+n-1}{4}+k)$ is simple and also a pole of the denominator $\Gamma(r+\rho+\alpha'+2k)$ whence $t^{(2)}_{\alpha,\alpha'}(r,r')$ is holomorphic in $r\in\CC$. Now assume $t^{(2)}_{\alpha,\alpha'}(r,r')=0$ for all $\alpha,\alpha'$. Then
$$ 0 = t^{(2)}_{j,j}(r,r') = (-j)_j \frac{\Gamma(\frac{2r+2j+n-1}{4})}{\Gamma(r+\rho+j)} $$
and hence $r$ has to be a pole of the denominator while it is a regular point for the numerator. This means $r+\rho=-i\in-\NN$ with $i\geq j$ and $\frac{2r+2j+n-1}{4}=\frac{j-i}{2}\notin-\NN$, i.e. $i-j\in2\NN+1$. But for $(\alpha,\alpha')=(i+1,j)$ only the summand for $k=\frac{i-j+1}{2}$ is non-zero and hence $t^{(2)}_{\alpha,\alpha'}(r,r')\neq0$, a contradiction.
\item Note that $(\frac{2r-2r'+1}{4})_k=(-N)_k=0$ for $k>N$ and hence we can write
\begin{multline}
 \hspace{2cm}t^{(3)}_{\alpha,\alpha'}(r,r') = \sum_{k=0}^N \frac{2^{4k}(-N)_k(\frac{\alpha+\alpha'+n-2}{2})_k(-\frac{\alpha-\alpha'}{2})_k}{(2k)!}\\
 \times(r+N+\tfrac{1}{2})_k(r+\rho+\alpha'+2k)_{2N-2k}\label{eq:EigenvaluesJuhlOperators}
\end{multline}
which is clearly holomorphic in $r\in\CC$. Further, $t^{(3)}_{\alpha,\alpha}(r,r')=(r+\rho+\alpha)_{2N}$ which is non-zero for $\alpha>-(r+\rho)$.\qedhere
\end{enumerate}
\end{proof}

\begin{remark}
After a few modifications we find that
$$ t_{\alpha,\alpha'}(r,r') = \frac{(r'+\rho')_{\alpha'}}{(r+\rho)_{\alpha'}}{_4F_3}\left(-\tfrac{\alpha-\alpha'}{2},\tfrac{\alpha+\alpha'+n-2}{2},\tfrac{2r+2r'+1}{4},\tfrac{2r-2r'+1}{4};\tfrac{1}{2},\tfrac{r+\rho+\alpha'}{2},\tfrac{r+\rho+\alpha'+1}{2};1\right). $$
Note that the generalized hypergeometric function $_4F_3(a_1,a_2,a_3,a_4;b_1,b_2,b_3;z)$ occurring here is balanced, i.e. $a_1+a_2+a_3+a_4+1=b_1+b_2+b_3$. However, there does not exist an explicit formula for its special value at $z=1$ in the literature. Also, we could not find estimates for special values of such hypergeometric functions for large/small parameters, and therefore were not able to show that $t_{\alpha,\alpha'}(r,r')$ grows at most polynomially in $\alpha,\alpha'\geq0$ for fixed $(r,r')\in\CC^2$. This is what is needed to apply Proposition~\ref{prop:HCtoInfty} in order to show automatic continuity of intertwining operators. We will therefore first describe all intertwining operators in terms of the holomorphic family $T^{(1)}(r,r')$ (see Theorem~\ref{thm:RealExplicitBasis}) and then show automatic continuity using the corresponding holomorphic family in the smooth category obtained in joint work with Y. Oshima~\cite{MOO13}. This is done in Corollary~\ref{cor:HCtoInftyReal}.
\end{remark}

\begin{theorem}\label{thm:RealExplicitBasis}
For $i=1,2,3$ we let $T^{(i)}(r,r')$ be the intertwining operators $(\pi_r)_\HC\to(\tau_{r'})_\HC$ corresponding to the numbers $t^{(i)}_{\alpha,\alpha'}(r,r')$ in Corollary~\ref{cor:RealRenormalizations}. Then the operator $T^{(1)}(r,r')$ is defined for $(r,r')\in\CC^2$, the operator $T^{(2)}(r,r')$ is defined for $r'\in-\rho'-\NN$ and the operator $T^{(3)}(r,r')$ is defined for $(r+\rho)-(r'+\rho')\in-2\NN$. We have
$$ \Hom_{(\frakg',K')}((\pi_r)_\HC|_{(\frakg',K')},(\tau_{r'})_\HC) = \begin{cases}\CC T^{(1)}(r,r') & \mbox{for $(r,r')\in\CC^2\setminus\Leven$,}\\\CC T^{(2)}(r,r')\oplus\CC T^{(3)}(r,r') & \mbox{for $(r,r')\in\Leven$.}\end{cases} $$
\end{theorem}

\begin{remark}\label{rem:RenormalizationForSubquotientsReal}
By the proof of Theorem~\ref{thm:RealMultiplicities}~(2) every intertwining operator between the subquotients $\calF(i),\calT(i)$ and $\calF'(j),\calT'(j)$ can be constructed by composing an intertwining operator $(\pi_r)_\HC\to(\tau_{r'})_\HC$ for particular $r,r'$ with embeddings and/or quotient maps for the subquotients. Hence, also every intertwining operator between subquotients is given by an operator in one of the three families $T^{(i)}(r,r')$. Therefore, all information about intertwining operators between $(\pi_r)_\HC$ and $(\tau_{r'})_\HC$ and any of their subquotients is contained in the holomorphic family $T^{(1)}(r,r')$.
\end{remark}

\begin{remark}\label{rem:JuhlOperators}
The family of operators $T^{(3)}$ is (up to a constant) equal to Juhl's family of conformally invariant differential restriction operators $D_{2N}(r):C^\infty(S^{n-1})\to C^\infty(S^{n-2})$ (see \cite{Juh09} and also \cite{KS13}). The constants $t_{\alpha,\alpha'}^{(3)}$ then give the ``spectrum'' of Juhl's operators in the sense that they describe how the operators are acting on explicit $K$-finite vectors. Note that by \eqref{eq:EigenvaluesJuhlOperators} the number of summands for $t^{(3)}_{\alpha,\alpha'}(r,r')$ is at most $N+1$.
\end{remark}

\begin{corollary}\label{cor:HCtoInftyReal}
For $(G,G')=(\upO(1,n),\upO(1,n-1))$ the natural injective map
\begin{equation}
 \Hom_{G'}(\pi|_{G'},\tau)\to\Hom_{(\frakg',K')}(\pi_\HC|_{(\frakg',K')},\tau_\HC)\label{eq:NaturalInjMapInftyToHC}
\end{equation}
is an isomorphism for all spherical principal series $\pi$ of $G$ and $\tau$ of $G'$ and their subquotients.
\end{corollary}

\begin{proof}
By Remark~\ref{rem:RenormalizationForSubquotientsReal} all intertwining operators between subquotients arise by composing with quotient maps and embeddings. It therefore suffices to show that \eqref{eq:NaturalInjMapInftyToHC} is an isomorphism for $\pi=\pi_r$ and $\tau=\tau_{r'}$ for all $(r,r')\in\CC^2$. In \cite{MOO13} a holomorphic family $A(r,r')\in\Hom_{G'}(\pi_r|_{G'},\tau_{r'})$ was constructed in the smooth category using singular integral operators (see Section~\ref{sec:SingularIntegralOperators} for details). Denote by $\overline{A}(r,r')\in\Hom_{(\frakg',K')}((\pi_r)_\HC|_{(\frakg',K')},(\tau_{r'})_\HC)$ its image under the map \eqref{eq:NaturalInjMapInftyToHC}. By Theorem~\ref{thm:RealExplicitBasis} this space is generically spanned by $T^{(1)}(r,r')$, and since both $\overline{A}(r,r')$ and $T^{(1)}(r,r')$ depend holomorphically on $(r,r')\in\CC^2$ there exists a meromorphic function $\phi(r,r')$ such that
$$ \overline{A}(r,r') = \phi(r,r')\cdot T^{(1)}(r,r'). $$
Replacing $A(r,r')$ and $\overline{A}(r,r')$ by $\phi(r,r')^{-1}A(r,r')$ and $\phi(r,r')^{-1}\overline{A}(r,r')$ we may assume that
$$ \overline{A}(r,r') = T^{(1)}(r,r'). $$
This already implies that for $(r,r')\in\CC^2\setminus\Leven$ every intertwining operator in the space $\Hom_{(\frakg',K')}((\pi_r)_\HC|_{(\frakg',K')},(\tau_{r'})_\HC)$ extends to the smooth globalization. Further, for $(r,r')\in\Leven$ we may restrict $(r,r')\mapsto T^{(1)}(r,r')$ to an affine complex line and renormalize to obtain all intertwining operators in $\Hom_{(\frakg',K')}((\pi_r)_\HC|_{(\frakg',K')},(\tau_{r'})_\HC)$ by Theorem~\ref{thm:RealExplicitBasis}. The same restriction and renormalization can be applied to $(r,r')\mapsto A(r,r')$, and in this way one obtains extensions of all operators in $\Hom_{(\frakg',K')}((\pi_r)_\HC|_{(\frakg',K')},(\tau_{r'})_\HC)$ to the smooth globalization. Note that renormalization of $A(r,r')$ preserves continuity of the operators. This shows that the map \eqref{eq:NaturalInjMapInftyToHC} is surjective, hence an isomorphism for all $(r,r')\in\CC^2$.
\end{proof}

\begin{remark}
The operators $T^{(i)}(r,r')$ are related to the operators $\tilde{\AA}_{\lambda,\nu}$, $\tilde{\tilde{\AA}}_{\lambda,\nu}$ and $\tilde{\CC}_{\lambda,\nu}$ studied by Kobayashi--Speh~\cite{KS13} for $\lambda=r+\rho$, $\nu=r'+\rho'$. In fact, using their notation we have
\begin{align*}
 T^{(1)}(r,r') &= \pi^{-\frac{n-2}{2}}\tilde{\AA}_{\lambda,\nu},\\
 T^{(2)}(r,r') &= \pi^{-\frac{n-2}{2}}\tilde{\tilde{\AA}}_{\lambda,\nu},\\
 T^{(3)}(r,r') &= \frac{(-1)^NN!}{2^{2N}}\tilde{\CC}_{\lambda,\nu},
\end{align*}
where for $i=3$ we write $r'+\rho'=r+\rho+2N$ with $N\in\NN$.
\end{remark}

\subsection{Discrete components in the restriction of unitary representations}

We apply our results to branching problems for unitary representations. The $(\frakg,K)$-modules $(\pi_r)_\HC$ are unitarizable if and only if $r\in i\RR\cup(-\rho,\rho)$ and we denote by $\widehat{\pi}_r$ their unitary completions. For $r\in i\RR$ these representations form the unitary principal series and for $r\in(-\rho,\rho)$ they belong to the complementary series. Further, all irreducible quotients $\calT(i)$ are unitarizable and their unitary completions will be denoted by $\widehat{\pi}_{-\rho-i}$. We note that for $r\in-(\rho+\ZZ)$, $r<0$, each representation $\widehat{\pi}_r$ is isomorphic to some Zuckerman derived functor module $A_\frakq(\lambda)$ and occurs discretely in the decomposition of the regular representation on $L^2(G/G')$.

Similarly we denote by $\widehat{\tau}_{r'}$, $r'\in i\RR\cup(-\rho',\rho')$ the unitary completions of $\tau_{r'}$ and by $\widehat{\tau}_{-\rho'-j}$, $j\in\NN$, the unitary completions of $\calT'(j)$.

For $r\in\RR$ we define the finite set
$$ D(r) = (r+\tfrac{1}{2}+2\NN)\cap(-\infty,0) $$
and note that for $r\in(-\rho,0)\cup(-\rho-\NN)$ and $r'\in D(r)$ we have $r'\in(-\rho',0)\cup(-\rho'-\NN)$, i.e. $\widehat{\tau}_{r'}$ is a unitary representation.

\begin{theorem}\label{thm:RealDiscreteComponentsInUniReps}
Let $r\in(-\rho,0)\cup(-\rho-\NN)$. Then for every $r'\in D(r)$ the representation $\widehat{\tau}_{r'}$ occurs discretely with multiplicity one in the restriction of $\widehat{\pi}_r$ to $G'$.
\end{theorem}

We note that for a complementary series representation $\widehat{\pi}_r$, $r\in(-\rho,0)$, all representations $\widehat{\tau}_{r'}$, $r'\in D(r)$, are complementary series representations. If $\widehat{\pi}_r$ is an $A_\frakq(\lambda)$-module, $r\in-\rho+\ZZ$, $r<0$, then so are the representations $\widehat{\tau}_{r'}$, $r'\in D(r)$. The restriction of the $A_\frakq(\lambda)$-modules $\widehat{\pi}_r$ to $G'$ decomposes with both discrete and continuous spectrum and is therefore hard to study by purely algebraic methods.

\begin{remark}\label{rem:RealDiscreteComponentsInUniReps}
For the special case $r'=r+\frac{1}{2}$, i.e. $N=0$, the occurrence of $\widehat{\tau}_{r'}$ in $\widehat{\pi}_r|_{G'}$ was first proved by Speh--Venkataramana~\cite{SV11} for $r\in[-\rho,-\frac{1}{2})$ and generalized by Zhang~\cite{Zha15} to the case $r\in(-\rho,-\frac{1}{2})\cup(-\rho-\NN)$. Later Kobayashi--Speh~\cite[Theorem 1.4]{KS13} proved Theorem~\ref{thm:RealDiscreteComponentsInUniReps} for the case $r\in(-\rho,0)$. The full decomposition of $\widehat{\pi}_r|_{G'}$ for $r\in(-\rho,0)\cup(-\rho-\NN)$ including the continuous spectrum was given by M\"{o}llers--Oshima~\cite{MO12}.
\end{remark}

We first describe the invariant norms on the unitarizable constituents for $r\in\RR$. For this we fix the $L^2$-norm $\|\blank\|_{L^2(S^{n-1})}$ on $L^2(K/M)=L^2(S^{n-1})$ corresponding to the standard Euclidean measure on $S^{n-1}$. For $r\in(-\rho,\rho)$ the norm $\|\blank\|_r$ on $\calE$ given by
$$ \|v\|_r^2 = \sum_{\alpha=0}^\infty b_\alpha(r)\|v_\alpha\|_{L^2(S^{n-1})}^2, \qquad \mbox{for }v=\sum_{\alpha=0}^\infty v_\alpha\in\bigoplus_{\alpha=0}^\infty\calE(\alpha), $$
with
$$ b_\alpha = \frac{\Gamma(\rho-r+\alpha)}{\Gamma(\rho+r+\alpha)} \sim (1+\alpha)^{-2r} $$
turns $(\pi_r)_\HC$ into a unitary $(\frakg,K)$-module. Further, for $r=-\rho-i$ the seminorm $\|\blank\|_r$ on $\calE$ has kernel $\calF(i)$ and turns the quotient $\calT(i)=\calE/\calF(i)$ into a unitary $(\frakg,K)$-module.

Similarly we denote by $\|\blank\|_{r'}'$ the $\tau_{r'}$-invariant norm on $\calE'$ respectively $\calT'(j)$ given by
$$ \|w\|_{r'}'^2 = \sum_{\alpha=0}^\infty b'_{\alpha'}(r')\|w_{\alpha'}\|_{L^2(S^{n-2})}^2, \qquad \mbox{for }w=\sum_{\alpha'=0}^\infty w_{\alpha'}\in\bigoplus_{\alpha'=0}^\infty\calE'(\alpha'), $$
with
$$ b'_{\alpha'} = \frac{\Gamma(\rho'-r'+\alpha')}{\Gamma(\rho'+r'+\alpha')} \sim (1+\alpha')^{-2r'}. $$

We need the following two basic results (see e.g. \cite[Lemma 3.2 and 3.5]{Zha15}):

\begin{lemma}\label{lem:GenCriterionBoundedness}
Let $\calV\subseteq\calE$ be a $K$-invariant subspace and $\calW\subseteq\calE'$ a $K'$-invariant subspace and assume that $\calV$ and $\calW$ are endowed with pre-Hilbert space structures with respect to which the groups $K$ and $K'$ act unitarily. A linear map $T:\calV\to\calW$ is bounded if and only if there exists a constant $C>0$ such that
$$ \sum_{\substack{\alpha\\\calE(\alpha;\alpha')\subseteq\calV}} \|T|_{\calE(\alpha;\alpha')}\|_{\calV\to\calW}^2 \leq C \qquad \forall\,\alpha', $$
where $\|\blank\|_{\calV\to\calW}$ denotes the operator norm with respect to the given pre-Hilbert space structures.
\end{lemma}

\begin{lemma}\label{lem:ElementarySumEstimate}
Suppose that $\alpha>-1$, $\beta\geq0$ and $\beta-\alpha>1$. Then there exists a constant $C>0$ such that
$$ \sum_{p=0}^\infty \frac{(1+p)^\alpha}{(1+p+q)^\beta} \leq \frac{C}{(1+q)^{\beta-\alpha-1}} \qquad \forall\,q\geq0. $$
\end{lemma}

\begin{proof}[{Proof of Theorem~\ref{thm:RealDiscreteComponentsInUniReps}}]
For $r\in(-\rho,0)$ let $\calV=\calE$ and for $r=-\rho-i\in-\rho-\NN$ let $\calV=\bigoplus_{\alpha=i+1}^\infty\calE(\alpha)$. Let $r'\in D(r)$ then similarly we put $\calW=\calE'$ for $r'\in(-\rho',0)$ and $\calW=\bigoplus_{\alpha'=j+1}^\infty\calE'(\alpha')$ for $r'=-\rho'-j\in-\rho'-\NN$. By Theorem~\ref{thm:RealMultiplicities} there exists (up to scalar) a unique non-zero intertwining operator $T:(\pi_r)_\HC\to(\tau_{r'})_\HC$ with $T(\calV)\subseteq\calW$ and if $r=-\rho-i$ additionally $T|_{\calF(i)}=0$. In our notation
$$ T|_{\calE(\alpha;\alpha')} = t_{\alpha,\alpha'}\cdot\rest|_{\calE(\alpha;\alpha')} $$
with $t_{\alpha,\alpha'}=t^{(3)}_{\alpha,\alpha'}$ for $\alpha'>j$ and $t_{\alpha,\alpha'}=0$ else (see Corollary~\ref{cor:RealRenormalizations} for the definition of $t_{\alpha,\alpha'}^{(3)}$). We show that $T$ is bounded if we endow $\calV$ with the norm $\|\blank\|_r$ and $\calW$ with the norm $\|\blank\|_{r'}$. To apply Lemma~\ref{lem:GenCriterionBoundedness} we calculate
$$ \|T|_{\calE(\alpha;\alpha')}\|_{\calV\to\calW}^2 = t_{\alpha,\alpha'}^2\|\rest|_{\calE(\alpha;\alpha')}\|_{\calE(\alpha;\alpha')\to\calE'(\alpha')}^2\frac{b'_{\alpha'}(r')}{b_\alpha(r)}, $$
where $\|\blank\|_{\calE(\alpha;\alpha')\to\calE'(\alpha')}$ denotes the operator norm with respect to the $L^2$-inner products on $\calE(\alpha;\alpha')\subseteq L^2(S^{n-1})$ and $\calE'(\alpha')\subseteq L^2(S^{n-2})$. Using \eqref{eq:GegenbauerValue0}, \eqref{eq:ExplicitBranchingRealSphericalHarmonics} and \eqref{eq:PlancherelEmbeddingReal} it is easy to see that for $\alpha=\alpha'+2\ell$ we have
\begin{align*}
 \|\rest|_{\calE(\alpha;\alpha')}\|_{\calE(\alpha;\alpha')\to\calE'(\alpha')}^2 &= \frac{2^{2\alpha'+n-3}(\alpha'+2\ell+\frac{n-2}{2})(2\ell)!\Gamma(\alpha'+\ell+\frac{n-2}{2})^2}{\pi(\ell!)^2\Gamma(2\alpha'+2\ell+n-2)}\\
 &= \frac{(\alpha'+2\ell+\frac{n-2}{2})\Gamma(\ell+\frac{1}{2})\Gamma(\alpha'+\ell+\frac{n-2}{2})}{\pi\Gamma(\ell+1)\Gamma(\alpha'+\ell+\frac{n-1}{2})} \sim \frac{(1+\alpha'+\ell)^{\frac{1}{2}}}{(1+\ell)^{\frac{1}{2}}}.
\end{align*}
Then Lemma~\ref{lem:GenCriterionBoundedness} translates into
$$ \sum_{\ell=0}^\infty t_{\alpha'+2\ell,\alpha'}^2 \frac{(1+\alpha'+\ell)^{\frac{1}{2}+2r}}{(1+\ell)^{\frac{1}{2}}} \leq C(1+\alpha')^{2r'}. $$
It is enough to check this for each of the $N+1$ summands of $t_{\alpha'+2\ell,\alpha'}$ in \eqref{eq:EigenvaluesJuhlOperators} where $r'+\rho'=r+\rho+2N$. The $k$'th summand grows of order
$$ \sim(1+\alpha')^{(r'+\rho')-(r+\rho+2k)}(1+\ell)^k(1+\alpha'+\ell)^k $$
and hence the claim follows by Lemma~\ref{lem:ElementarySumEstimate}. Altogether this shows that $T$ induces a bounded $G'$-intertwining operator $\tilde{T}:\widehat{\pi}_r|_{G'}\to\widehat{\tau}_{r'}$ whose adjoint $\tilde{T}^*:\widehat{\tau}_{r'}\to\widehat{\pi}_r|_{G'}$ embeds $\widehat{\tau}_{r'}$ isometrically as a subrepresentation of $\widehat{\pi}_r$ by Schur's Lemma. Multiplicity one follows from the fact that any $G'$-equivariant embedding $S:\widehat{\tau}_{r'}\to\widehat{\pi}_r|_{G'}$ induces an intertwiner $S^*:(\pi_r)_\HC\to(\tau_{r'})_\HC$ between the Harish-Chandra modules by taking the adjoint operator and then passing to $K$-finite vectors. Such an operator is unique (up to scalars) by Theorem~\ref{thm:RealMultiplicities} and since $K$-finite vectors are dense in $\widehat{\pi}_r$ the embedding $S$ is unique (up to scalars). This finishes the proof.
\end{proof}

\subsection{Comparison with singular integral operators}\label{sec:SingularIntegralOperators}

In \cite{KS13,MOO13} a meromorphic family of intertwining operators $A(r,r'):u_{\1,r\nu}|_{G'}\to u'_{\1,r'\nu}$ in the smooth category is constructed as family of singular integral operators. In the compact picture this family is (up to scalars) given by
\begin{multline*}
 A(r,r'): C^\infty(S^{n-1})\to C^\infty(S^{n-2}),\\
 A(r,r')f(y) = \int_{S^{n-1}} (|x'-y|^2+x_n^2)^{-(r'+\rho')}|x_n|^{(r-\rho)+(r'+\rho')} f(x) dx,
\end{multline*}
where $dx$ denotes the Euclidean measure on $S^{n-1}$.

\begin{theorem}
Let $T(r,r'):C^\infty(S^{n-1})\to C^\infty(S^{n-2})$ denote the intertwining operator with spectrum given by the numbers $t_{\alpha,\alpha'}(r,r')$ in \eqref{eq:ExplicitFormulaReal}. Then
$$ A(r,r') = \frac{2^{r-r'+\frac{1}{2}}\pi^{\frac{n-2}{2}}\Gamma(\frac{2r+2r'+1}{4})\Gamma(\frac{2r-2r'+1}{4})}{\Gamma(r+\frac{n-1}{2})}\cdot T(r,r'). $$
\end{theorem}

\begin{proof}
Since by Theorem~\ref{thm:RealMultiplicities}~(1) and Corollary~\ref{cor:HCtoInftyReal} we generically have $\dim\Hom_{G'}(\pi_r|_{G'},\tau_{r'})=1$ and both $A(r,r')$ and $T(r,r')$ are meromorphic in $r,r'\in\CC$ there exists a scalar meromorphic function $c(r,r')$ with $A(r,r')=c(r,r')T(r,r')$. To determine $c(r,r')$ we put $f\equiv 1$:
$$ c(r,r') = \int_{S^{n-1}} (|x'-y|^2+x_n^2)^{-(r'+\frac{n-2}{2})}|x_n|^{r+r'-\frac{1}{2}} dx. $$
Using the stereographic projection $x=(\frac{1-|z|^2}{1+|z|^2},\frac{2z}{1+|z|^2})$, $z\in\RR^{n-1}$, the measure transforms by $dx=2^{n-1}(1+|z|^2)^{-(n-1)}dz$ where $dz$ is the standard Lebesgue measure on $\RR^{n-1}$. Writing $y=(\frac{1-|w|^2}{1+|w|^2},\frac{2w}{1+|w|^2})$ with $w\in\RR^{n-2}$ we find
\begin{multline*}
 c(r,r') = 2^{r-r'+\frac{1}{2}}(1+|w|^2)^{r'+\frac{n-2}{2}} \int_{\RR^{n-1}} (|z'-w|^2+z_{n-1}^2)^{-(r'+\frac{n-2}{2})}\\
 |z_{n-1}|^{r+r'-\frac{1}{2}} (1+|z|^2)^{-(r+\frac{n-1}{2})} dz,
\end{multline*}
where we have written $z=(z',z_{n-1})$. This integral is evaluated in \cite[Proposition 7.4]{KS13} and we obtain
$$ c(r,r') = \frac{2^{r-r'+\frac{1}{2}}\pi^{\frac{n-2}{2}}\Gamma(\frac{2r+2r'+1}{4})\Gamma(\frac{2r-2r'+1}{4})}{\Gamma(r+\frac{n-1}{2})} $$
which shows the claim.
\end{proof}

\begin{remark}
The special value of the intertwiners $A(r,r')$ at the spherical vector $f\equiv1$ was also calculated in \cite{MO13} by a different method.
\end{remark}

\begin{remark}
The action of $A(r,r')$ on $K$-finite vectors was also computed by Kobayashi--Speh~\cite[Lemma 7.7]{KS13}. However, their parametrization of $K$-finite vectors differs from our parametrization by $(\alpha,\alpha')$, and therefore it is non-trivial to see the equivalence of their identity and our identity \eqref{eq:ExplicitFormulaReal}.
\end{remark}

%
%

\section{Rank one unitary groups}\label{sec:ExampleComplex}

We indicate in this section how the calculations in Section~\ref{sec:ExampleReal} can be generalized to rank one unitary groups and state the corresponding results. Let $n\geq2$ and consider the indefinite unitary group $G=\upU(1,n)$ of $(n+1)\times(n+1)$ complex matrices leaving the standard Hermitian form on $\CC^{n+1}$ of signature $(1,n)$ invariant. The subgroup $G'\subseteq G$ of matrices fixing the last standard basis vector $e_{n+1}$ is isomorphic to $\upU(1,n-1)$.

\subsection{$K$-types}\label{sec:ComplexKTypes}

We fix $K=\upU(1)\times \upU(n)$ and choose
$$ H = \left(\begin{array}{ccc}0&1&\\1&0&\\&&\0_{n-1}\end{array}\right) $$
so that $P=MAN$ with $M=\Delta \upU(1)\times \upU(n-1)$ where $\Delta \upU(1)=\{\diag(x,x):x\in \upU(1)\}$. Note that $\rho=n$. Then $K$ acts transitively on the unit sphere $S^{2n-1}\subseteq\CC^n$ via $\diag(\lambda,k)\cdot z=\lambda^{-1}kz$, $\lambda\in \upU(1)$, $k\in \upU(n)$, $z\in S^{2n-1}$, and $M$ is the stabilizer subgroup of the first standard basis vector $e_1$ whence $K/M\cong S^{2n-1}$. The subgroup $G'=\upU(1,n-1)$ is embedded into $G$ such that $K'=\upU(1)\times \upU(n-1)$ and $P'=G'\cap P=M'A'N'$ with $A'=A$ and $M'=\Delta \upU(1)\times \upU(n-2)$. Then $K'/M'=S^{2n-3}\subseteq\CC^{n-1}$, viewed as the codimension two submanifold in $K/M=S^{2n-1}\subseteq\CC^n$ given by $z_n=0$. Further we have $\rho'=n-1$.

Let $\xi=\1$, $\xi'=\1$ be the trivial representations of $M$ and $M'$ and abbreviate $\pi_r=\pi_{\xi,r}$ and $\tau_{r'}=\tau_{\xi',r'}$. Then as $K$-modules resp. $K'$-modules we have
$$ \calE =\! \bigoplus_{\alpha_1,\alpha_2=0}^\infty\! \underbrace{e^{i(\alpha_1-\alpha_2)\theta}\boxtimes\,\calH^{\alpha_1,\alpha_2}(\CC^n)}_{\calE(\alpha)=}, \quad \calE' =\! \bigoplus_{\alpha_1',\alpha_2'=0}^\infty\! \underbrace{e^{i(\alpha_1'-\alpha_2')\theta}\boxtimes\,\calH^{\alpha_1',\alpha_2'}(\CC^{n-1})}_{\calE'(\alpha')=}, $$
where we abbreviate $\alpha=(\alpha_1,\alpha_2)$ and $\alpha'=(\alpha_1',\alpha_2')$. Hence, \eqref{eq:MF1} is satisfied. Further, each $K$-type decomposes by \eqref{eq:BranchingComplexSphericalHarmonics} into $K'$-types as follows:
$$ \left.\left(e^{i(\alpha_1-\alpha_2)\theta}\boxtimes\,\calH^{\alpha_1,\alpha_2}(\CC^n)\right)\right|_{K'} = \bigoplus_{\substack{0\leq\alpha_1'\leq\alpha_1\\0\leq\alpha_2'\leq\alpha_2}} \left(e^{i(\alpha_1-\alpha_2)\theta}\boxtimes\,\calH^{\alpha_1',\alpha_2'}(\CC^{n-1})\right), $$
so that \eqref{eq:MF2} holds. Comparing the characters of the $\upU(1)$-factor of $K'$ we find that $\Hom_{K'}(\calE(\alpha)|_{K'},\calE'(\alpha'))\neq0$ if and only if $\alpha_1-\alpha_2=\alpha_1'-\alpha_2'$. In this case formulas \eqref{eq:ExplicitBranchingComplexSphericalHarmonics} and \eqref{eq:JacobiValue1} show that the restriction operator
$$ R_{\alpha,\alpha'}=\rest|_{\calE(\alpha;\alpha')}:\calE(\alpha;\alpha')\to\calE'(\alpha') $$
is an isomorphism. Hence the restriction $T_{\alpha,\alpha'}=T|_{\calE(\alpha;\alpha')}$ of a $K'$-intertwining operator $T:\calE\to\calE'$ is given by $T_{\alpha,\alpha'}=t_{\alpha,\alpha'}R_{\alpha,\alpha'}$ for $\alpha_1-\alpha_2=\alpha_1'-\alpha_2'$ and $T_{\alpha,\alpha'}=0$ else.

\subsection{Proportionality constants}

The eigenvalues of the spectrum generating operator on the $K$-types are given by (see \cite[Section 3.b]{BOO96})
\begin{align*}
 \sigma_{(\alpha_1,\alpha_2)} &= 2\alpha_1(\alpha_1+n-1)+2\alpha_2(\alpha_2+n-1),\\
 \sigma_{(\alpha_1',\alpha_2')}' &= 2\alpha_1'(\alpha_1'+n-2)+2\alpha_2'(\alpha_2'+n-2).
\end{align*}
We write $\fraks_\CC=\fraks+J\fraks=\fraks_++\fraks_-$ and identify $\fraks_\pm\cong\CC^n$ via
$$ \CC^n \to \fraks_\pm, \quad w \mapsto X_{w,\pm}=\left(\begin{array}{cc}0&w^*\mp Jiw^*\\w\pm Jiw&\0_n\end{array}\right). $$
Then $\fraks_\pm'\simeq\CC^{n-1}$, embedded in $\CC^n$ as the first $n-1$ coordinates. Since both $\fraks_\pm'$ are multiplicity-free $K'$-modules, \eqref{eq:MF3} holds (with $\fraks'_\CC$ replaced by $\fraks'_\pm$) and we can use Corollary~\ref{cor:CharacterizationIntertwinersScalar}. The cocycle $\omega$ is given by
$$ \omega(X_{w,+})(z) = w^*z, \quad w\in\fraks_+, \qquad \omega(X_{w,-})(z) = z^*w, \quad w\in\fraks_-, $$
where $z\in S^{2n-1}\subseteq\CC^n$.

We note by \eqref{eq:DecompositionOfMultComplexSphericalHarmonics} that if $X\in\fraks_+$ then the multiplication map $m(\omega(X))$ maps the $K$-type $\calE(\alpha_1,\alpha_2)$ into the $K$-types $\calE(\alpha_1+1,\alpha_2)$ and $\calE(\alpha_1,\alpha_2-1)$ and if $X\in\fraks_-$ into the $K$-types $\calE(\alpha_1,\alpha_2+1)$ and $\calE(\alpha_1-1,\alpha_2)$. Because of similar considerations for $\fraks_+'$ and $\fraks_-'$ the equivalence relation $(\alpha,\alpha')\leftrightarrow(\beta,\beta')$ is given by
\begin{align*}
 ((\alpha_1,\alpha_2);(\alpha_1',\alpha_2'))\leftrightarrow(\beta;(\alpha_1'+1,\alpha_2')) &\Leftrightarrow \beta\in\{(\alpha_1+1,\alpha_2),(\alpha_1,\alpha_2-1)\},\\
 ((\alpha_1,\alpha_2);(\alpha_1',\alpha_2'))\leftrightarrow(\beta;(\alpha_1'-1,\alpha_2')) &\Leftrightarrow \beta\in\{(\alpha_1-1,\alpha_2),(\alpha_1,\alpha_2+1)\},\\
 ((\alpha_1,\alpha_2);(\alpha_1',\alpha_2'))\leftrightarrow(\beta;(\alpha_1',\alpha_2'+1)) &\Leftrightarrow \beta\in\{(\alpha_1-1,\alpha_2),(\alpha_1,\alpha_2+1)\},\\
 ((\alpha_1,\alpha_2);(\alpha_1',\alpha_2'))\leftrightarrow(\beta;(\alpha_1',\alpha_2'-1)) &\Leftrightarrow \beta\in\{(\alpha_1+1,\alpha_2),(\alpha_1,\alpha_2-1)\}.
\end{align*}
Now, Lemma~\ref{lem:CalculationOfLambdas} yields the following equations for $\lambda_{\alpha,\alpha'}^{\beta,\beta'}$: For $\beta'=(\alpha_1'+1,\alpha_2')$ we obtain
$$
\arraycolsep=1.4pt\def\arraystretch{1.5}
\begin{array}{rcrcl}
 \lambda_{(\alpha_1,\alpha_2),(\alpha_1',\alpha_2')}^{(\alpha_1+1,\alpha_2),(\alpha_1'+1,\alpha_2')}&+&\lambda_{(\alpha_1,\alpha_2),(\alpha_1',\alpha_2')}^{(\alpha_1,\alpha_2-1),(\alpha_1'+1,\alpha_2')} &=& 1,\\
 (2\alpha_1+n)\lambda_{(\alpha_1,\alpha_2),(\alpha_1',\alpha_2')}^{(\alpha_1+1,\alpha_2),(\alpha_1'+1,\alpha_2')}&-&(2\alpha_2+n-2)\lambda_{(\alpha_1,\alpha_2),(\alpha_1',\alpha_2')}^{(\alpha_1,\alpha_2-1),(\alpha_1'+1,\alpha_2')} &=& 2\alpha_1'+n,
\end{array}
$$
which gives
\begin{equation*}
 \lambda_{(\alpha_1,\alpha_2),(\alpha_1',\alpha_2')}^{(\alpha_1+1,\alpha_2),(\alpha_1'+1,\alpha_2')} = \frac{\alpha_1'+\alpha_2+n-1}{\alpha_1+\alpha_2+n-1}, \quad \lambda_{(\alpha_1,\alpha_2),(\alpha_1',\alpha_2')}^{(\alpha_1,\alpha_2-1),(\alpha_1'+1,\alpha_2')} = \frac{\alpha_1-\alpha_1'}{\alpha_1+\alpha_2+n-1},
\end{equation*}
for $\beta'=(\alpha_1'-1,\alpha_2')$ we get
$$
\arraycolsep=1.4pt\def\arraystretch{1.5}
\begin{array}{rcrcl}
 \lambda_{(\alpha_1,\alpha_2),(\alpha_1',\alpha_2')}^{(\alpha_1-1,\alpha_2),(\alpha_1'-1,\alpha_2')}&+&\lambda_{(\alpha_1,\alpha_2),(\alpha_1',\alpha_2')}^{(\alpha_1,\alpha_2+1),(\alpha_1'-1,\alpha_2')} &=& 1,\\
 (2\alpha_1+n-2)\lambda_{(\alpha_1,\alpha_2),(\alpha_1',\alpha_2')}^{(\alpha_1-1,\alpha_2),(\alpha_1'-1,\alpha_2')}&-&(2\alpha_2+n)\lambda_{(\alpha_1,\alpha_2),(\alpha_1',\alpha_2')}^{(\alpha_1,\alpha_2+1),(\alpha_1'-1,\alpha_2')} &=& 2\alpha_1'+n-4,
\end{array}
$$
implying
\begin{equation*}
 \lambda_{(\alpha_1,\alpha_2),(\alpha_1',\alpha_2')}^{(\alpha_1-1,\alpha_2),(\alpha_1'-1,\alpha_2')} = \frac{\alpha_1'+\alpha_2+n-2}{\alpha_1+\alpha_2+n-1}, \quad \lambda_{(\alpha_1,\alpha_2),(\alpha_1',\alpha_2')}^{(\alpha_1,\alpha_2+1),(\alpha_1'-1,\alpha_2')} = \frac{\alpha_1-\alpha_1'+1}{\alpha_1+\alpha_2+n-1},
\end{equation*}
and similarly we find
\begin{align*}
 \lambda_{(\alpha_1,\alpha_2),(\alpha_1',\alpha_2')}^{(\alpha_1,\alpha_2+1),(\alpha_1',\alpha_2'+1)} &= \frac{\alpha_1+\alpha_2'+n-1}{\alpha_1+\alpha_2+n-1}, & \lambda_{(\alpha_1,\alpha_2),(\alpha_1',\alpha_2')}^{(\alpha_1-1,\alpha_2),(\alpha_1',\alpha_2'+1)} &= \frac{\alpha_2-\alpha_2'}{\alpha_1+\alpha_2+n-1},\\
 \lambda_{(\alpha_1,\alpha_2),(\alpha_1',\alpha_2')}^{(\alpha_1,\alpha_2-1),(\alpha_1',\alpha_2'-1)} &= \frac{\alpha_1+\alpha_2'+n-2}{\alpha_1+\alpha_2+n-1}, & \lambda_{(\alpha_1,\alpha_2),(\alpha_1',\alpha_2')}^{(\alpha_1+1,\alpha_2),(\alpha_1',\alpha_2'-1)} &= \frac{\alpha_2-\alpha_2'+1}{\alpha_1+\alpha_2+n-1}.
\end{align*}
We remark that the constants $\lambda_{\alpha,\alpha'}^{\beta,\beta'}$ can in this case also be obtained by computing the action of $\omega(X)$ on explicit $K$-finite vectors using \eqref{eq:ExplicitBranchingComplexSphericalHarmonics} and recurrence relations for the Jacobi polynomials. With the explicit form of the constants $\lambda_{\alpha,\alpha'}^{\beta,\beta'}$ Corollary~\ref{cor:CharacterizationIntertwinersScalar} now provides the following characterization of symmetry breaking operators:

\begin{theorem}\label{thm:ComplexCharacterizationIntertwiners}
An operator $T:\calE\to\calE'$ is intertwining for $\pi_r$ and $\tau_{r'}$ if and only if
$$ T|_{\calE(\alpha;\alpha')} = \begin{cases}t_{\alpha,\alpha'}\cdot\rest|_{\calE(\alpha;\alpha')} & \mbox{for $\alpha_1-\alpha_2=\alpha_1'-\alpha_2'$,}\\0 & \mbox{else,}\end{cases} $$
with numbers $t_{\alpha,\alpha'}$ satisfying the following four relations:
\begin{multline}
 (\alpha_1+\alpha_2+n-1)(r'+2\alpha_1'+n-1)t_{(\alpha_1,\alpha_2),(\alpha_1',\alpha_2')}\\
 = (\alpha_1'+\alpha_2+n-1)(r+2\alpha_1+n)t_{(\alpha_1+1,\alpha_2),(\alpha_1'+1,\alpha_2')}\\
 + (\alpha_1-\alpha_1')(r-2\alpha_2-n+2)t_{(\alpha_1,\alpha_2-1),(\alpha_1'+1,\alpha_2')},\label{eq:ComplexRel1}
\end{multline}
\vspace{-.5cm}
\begin{multline}
 (\alpha_1+\alpha_2+n-1)(r'-2\alpha_1'-n+3)t_{(\alpha_1,\alpha_2),(\alpha_1',\alpha_2')}\\
 = (\alpha_1'+\alpha_2+n-2)(r-2\alpha_1-n+2)t_{(\alpha_1-1,\alpha_2),(\alpha_1'-1,\alpha_2')}\\
 + (\alpha_1-\alpha_1'+1)(r+2\alpha_2+n)t_{(\alpha_1,\alpha_2+1),(\alpha_1'-1,\alpha_2')},\label{eq:ComplexRel2}
\end{multline}
\vspace{-.5cm}
\begin{multline}
 (\alpha_1+\alpha_2+n-1)(r'+2\alpha_2'+n-1)t_{(\alpha_1,\alpha_2),(\alpha_1',\alpha_2')}\\
 = (\alpha_1+\alpha_2'+n-1)(r+2\alpha_2+n)t_{(\alpha_1,\alpha_2+1),(\alpha_1',\alpha_2'+1)}\\
 + (\alpha_2-\alpha_2')(r-2\alpha_1-n+2)t_{(\alpha_1-1,\alpha_2),(\alpha_1',\alpha_2'+1)},\label{eq:ComplexRel3}
\end{multline}
\vspace{-.5cm}
\begin{multline}
 (\alpha_1+\alpha_2+n-1)(r'-2\alpha_2'-n+3)t_{(\alpha_1,\alpha_2),(\alpha_1',\alpha_2')}\\
 = (\alpha_1+\alpha_2'+n-2)(r-2\alpha_2-n+2)t_{(\alpha_1,\alpha_2-1),(\alpha_1',\alpha_2'-1)}\\
 + (\alpha_2-\alpha_2'+1)(r+2\alpha_1+n)t_{(\alpha_1+1,\alpha_2),(\alpha_1',\alpha_2'-1)}.\label{eq:ComplexRel4}
\end{multline}
\end{theorem}

\subsection{Multiplicities}

The $(\frakg,K)$-module $(\pi_r)_\HC$ is reducible if and only if $r\in\pm(\rho+2\NN)$. More precisely, for $r=-\rho-2i$ the module $(\pi_r)_\HC$ contains a unique non-trivial finite-dimensional $(\frakg,K)$-submodule
$$ \calF(i)=\bigoplus_{\alpha_1,\alpha_2=0}^i\calE(\alpha_1,\alpha_2) $$
as well as the two non-trivial infinite-dimensional submodules
$$ \calF_+(i)=\bigoplus_{\alpha_1=0}^\infty\bigoplus_{\alpha_2=0}^i\calE(\alpha_1,\alpha_2) \qquad \calF_-(i)=\bigoplus_{\alpha_1=0}^i\bigoplus_{\alpha_2=0}^\infty\calE(\alpha_1,\alpha_2). $$
Then the composition series of $(\pi_r)_\HC$ is given by
$$ \{0\}\subseteq\calF(i)\subseteq\calF_+(i)\subseteq(\calF_+(i)+\calF_-(i))\subseteq\calE $$
(or equivalently with $\calF_+$ and $\calF_-$ switched). Hence the quotients
$$ \calT(i)=\calE/(\calF_+(i)+\calF_-(i)) \qquad \mbox{and} \qquad \calT_\pm(i)=\calF_\pm(i)/\calF(i) $$
are irreducible and infinite-dimensional. Similarly we denote by $\calF'(j)$, $\calF'_\pm(j)$ resp. $\calT'(j)$, $\calT'_\pm(j)$ the subrepresentations resp. -quotients of $(\tau_{r'})_\HC$ for $r'=-\rho'-2j$, $j\in\NN$.

Define
$$ L=\{(r,r')\in\CC^2:r=-\rho-2i,r'=-\rho'-2j,0\leq j\leq i\}. $$

\begin{theorem}\label{thm:ComplexMultiplicities}
\begin{enumerate}
\item The multiplicities between spherical principal series of $G$ and $G'$ are given by
$$ m((\pi_r)_\HC,(\tau_{r'})_\HC) = \begin{cases}1&\mbox{for $(r,r')\in\CC^2\setminus L$,}\\2&\mbox{for $(r,r')\in L$.}\end{cases} $$
\item For $i,j\in\NN$ the multiplicities $m(\calV,\calW)$ between subquotients are given by
\vspace{.2cm}
\begin{center}
 \begin{tabular}{c|cccc}
  \diagonal{.1em}{.85cm}{$\calV$}{$\calW$} & $\calF'(j)$ & $\calT_+'(j)$ & $\calT_-'(j)$ & $\calT'(j)$ \\
  \hline
  $\calF(i)$ & $1$ & $0$ & $0$ & $0$\\
  $\calT_+(i)$ & $0$ & $1$ & $0$ & $0$\\
  $\calT_-(i)$ & $0$ & $0$ & $1$ & $0$\\
  $\calT(i)$ & $0$ & $0$ & $0$ & $1$\\
  \multicolumn{5}{c}{for $j\leq i$,}
 \end{tabular}
 \qquad
 \begin{tabular}{c|cccc}
  \diagonal{.1em}{.85cm}{$\calV$}{$\calW$} & $\calF'(j)$ & $\calT_+'(j)$ & $\calT_-'(j)$ & $\calT'(j)$ \\
  \hline
  $\calF(i)$ & $0$ & $0$ & $0$ & $0$\\
  $\calT_+(i)$ & $0$ & $0$ & $0$ & $0$\\
  $\calT_-(i)$ & $0$ & $0$ & $0$ & $0$\\
  $\calT(i)$ & $1$ & $0$ & $0$ & $0$\\
  \multicolumn{5}{c}{otherwise.}
 \end{tabular}
\end{center}
\vspace{.2cm}
\end{enumerate}
\end{theorem}

To prove Theorem~\ref{thm:ComplexMultiplicities} we proceed similar as in Section~\ref{sec:MultiplicitiesReal}. For this we first reduce the four relations \eqref{eq:ComplexRel1}--\eqref{eq:ComplexRel4} in the four parameters $\alpha_1,\alpha_2,\alpha_1',\alpha_2'$ with $\alpha_1-\alpha_2=\alpha_1'-\alpha_2'$ to two pairs of two relations with only two parameters.

Put
$$ p=\alpha_1+\alpha_2, \qquad q_1=\alpha_1', \qquad q_2=\alpha_2' $$
then
$$ \alpha_1=\frac{p+q_1-q_2}{2}, \qquad \alpha_2=\frac{p-q_1+q_2}{2}, \qquad \alpha_1'=q_1, \qquad \alpha_2'=q_2. $$
Then $0\leq\alpha_1'\leq\alpha_1$, $0\leq\alpha_2'\leq\alpha_2$ and $\alpha_1-\alpha_2=\alpha_1'-\alpha_2'$ if and only if $p,q_1,q_2\in\NN$ with $p-q_1-q_2\in2\NN$. With this reparametrization, the parameter $q_2$ is constant in the identities \eqref{eq:ComplexRel1} and \eqref{eq:ComplexRel2} and the parameter $q_1$ is constant in \eqref{eq:ComplexRel3} and \eqref{eq:ComplexRel4}. Abusing notation and writing $t_{p,q_1,q_2}$ for $t_{(\alpha_1,\alpha_2),(\alpha_1',\alpha_2')}$ the relations \eqref{eq:ComplexRel1}--\eqref{eq:ComplexRel4} become
\begin{multline}
 (p+n-1)(r'+2q_1+n-1)t_{p,q_1,q_2} = (\tfrac{p+q_1+q_2}{2}+n-1)(r+p+q_1-q_2+n)t_{p+1,q_1+1,q_2}\\
 + (\tfrac{p-q_1-q_2}{2})(r-p+q_1-q_2-n+2)t_{p-1,q_1+1,q_2},\label{eq:ComplexRelReduced1}
\end{multline}
\vspace{-.5cm}
\begin{multline}
 (p+n-1)(r'-2q_1-n+3)t_{p,q_1,q_2} = (\tfrac{p+q_1+q_2}{2}+n-2)(r-p-q_1+q_2-n+2)t_{p-1,q_1-1,q_2}\\
 + (\tfrac{p-q_1-q_2}{2}+1)(r+p-q_1+q_2+n)t_{p+1,q_1-1,q_2},\label{eq:ComplexRelReduced2}
\end{multline}
\begin{multline}
 (p+n-1)(r'+2q_2+n-1)t_{p,q_1,q_2} = (\tfrac{p+q_1+q_2}{2}+n-1)(r+p-q_1+q_2+n)t_{p+1,q_1,q_2+1}\\
 + (\tfrac{p-q_1-q_2}{2})(r-p-q_1+q_2-n+2)t_{p-1,q_1,q_2+1},\label{eq:ComplexRelReduced3}
\end{multline}
\vspace{-.5cm}
\begin{multline}
 (p+n-1)(r'-2q_2-n+3)t_{p,q_1,q_2} = (\tfrac{p+q_1+q_2}{2}+n-2)(r-p+q_1-q_2-n+2)t_{p-1,q_1,q_2-1}\\
 + (\tfrac{p-q_1-q_2}{2}+1)(r+p+q_1-q_2+n)t_{p+1,q_1,q_2-1}.\label{eq:ComplexRelReduced4}
\end{multline}
Note that $q_2$ is fixed in \eqref{eq:ComplexRelReduced1} and \eqref{eq:ComplexRelReduced2}, and these relations hold for $p,q_1\in\NN$ with $p-q_1\in q_2+2\NN$. The obvious similar statement holds for \eqref{eq:ComplexRelReduced3} and \eqref{eq:ComplexRelReduced4}.

We first consider the diagonal $p=q_1+q_2$, then relations \eqref{eq:ComplexRelReduced1} and \eqref{eq:ComplexRelReduced3} simplify to
\begin{align}
 (r'+2q_1+n-1)t_{q_1+q_2,q_1,q_2} &= (r+2q_1+n)t_{q_1+q_2+1,q_1+1,q_2},\label{eq:ComplexRelDiagonal1}\\
 (r'+2q_2+n-1)t_{q_1+q_2,q_1,q_2} &= (r+2q_2+n)t_{q_1+q_2+1,q_1,q_2+1}.\label{eq:ComplexRelDiagonal2}
\end{align}
This immediately yields:

\begin{lemma}\label{lem:CplxDiagonalSequences}
\begin{enumerate}
\item For $(r,r')\in\CC^2\setminus L$ the space of diagonal sequences $(t_{q_1+q_2,q_1,q_2})_{q_1,q_2}$ satisfying \eqref{eq:ComplexRelDiagonal1} and \eqref{eq:ComplexRelDiagonal2} has dimension $1$. Any generator $(t_{q_1+q_2,q_1,q_2})_{q_1,q_2}$ satisfies:
\begin{enumerate}
\item for $r\notin-\rho-2\NN$, $r'\notin-\rho'-2\NN$:
$$ t_{q_1+q_2,q_1,q_2} \neq 0 \quad \forall\,q_1,q_2\in\NN, $$
\item for $r=-\rho-2i\in-\rho-2\NN$, $r'\notin-\rho'-2\NN$:
$$ \hspace{1.5cm}t_{q_1+q_2,q_1,q_2} = 0 \quad \forall\,q_1\leq i\mbox{ or }q_2\leq i \qquad \mbox{and} \qquad t_{q_1+q_2,q_1,q_2} \neq 0 \quad \forall\,q_1,q_2> i, $$
\item for $r\notin-\rho-2\NN$, $r'=-\rho'-2j\in-\rho'-2\NN$:
$$ \hspace{1.5cm}t_{q_1+q_2,q_1,q_2} \neq 0 \quad \forall\,q_1,q_2\leq j \qquad \mbox{and} \qquad t_{q_1+q_2,q_1,q_2} = 0 \quad \forall\,q_1>j\mbox{ or }q_2>j, $$
\item for $r=-\rho-2i\in-\rho-2\NN$, $r'=-\rho'-2j\in-\rho'-2\NN$ with $i<j$:
$$ t_{q_1+q_2,q_1,q_2} \neq 0 \quad \forall\,i<q_1,q_2\leq j \qquad \mbox{and} \qquad t_{q_1+q_2,q_1,q_2} = 0 \quad \mbox{else}. $$
\end{enumerate}
\item For $(r,r')=(-\rho-2i,-\rho'-2j)\in L$, the space of diagonal sequences $(t_{q_1+q_2,q_1,q_2})_{q_1,q_2}$ satisfying \eqref{eq:ComplexRelDiagonal1} and \eqref{eq:ComplexRelDiagonal2} has dimension $4$.
\end{enumerate}
\end{lemma}

Next we investigate how a diagonal sequence $(t_{q_1+q_2,q_1,q_2})_{q_1,q_2}$ satisfying \eqref{eq:ComplexRelDiagonal1} and \eqref{eq:ComplexRelDiagonal2} can be extended to a sequence $(t_{p,q_1,q_2})_{p,q_1,q_2}$ satisfying \eqref{eq:ComplexRelReduced1} and \eqref{eq:ComplexRelReduced2} and the corresponding relations in $q_2$. For this note that if we fix, say, $q_2$ and put $p'=p-q_2$, then the relations \eqref{eq:ComplexRelReduced1} and \eqref{eq:ComplexRelReduced2} read
\begin{multline}
 (p'+q_2+n-1)(r'+2q_1+n-1)t_{p',q_1} = (\tfrac{p'+q_1}{2}+q_2+n-1)(r+p'+q_1+n)t_{p'+1,q_1+1}\\
 + (\tfrac{p'-q_1}{2})(r-p'+q_1-2q_2-n+2)t_{p'-1,q_1+1},\label{eq:ComplexRelReducedAgain1}
\end{multline}
\vspace{-.5cm}
\begin{multline}
 (p'+q_2+n-1)(r'-2q_1-n+3)t_{p',q_1} = (\tfrac{p'+q_1}{2}+q_2+n-2)(r-p'-q_1-n+2)t_{p'-1,q_1-1}\\
 + (\tfrac{p'-q_1}{2}+1)(r+p'-q_1+2q_2+n)t_{p'+1,q_1-1},\label{eq:ComplexRelReducedAgain2}
\end{multline}
where we again abuse notation and write $t_{p',q_1}$ for $t_{p,q_1,q_2}$. Similar relations hold if $q_1$ is fixed. We note that \eqref{eq:ComplexRelReducedAgain1} and \eqref{eq:ComplexRelReducedAgain2} have to be satisfied for all $p',q_1\in\NN$ with $p'-q_1\in2\NN$, just as in the case of orthogonal groups, see Diagram~\ref{Ktypes}. Thus, many arguments used in the orthogonal situation can be translated to this context. There are, however, differences to the orthogonal situation. If $r=-\rho-2i\in-\rho-2\NN$ then the coefficient $(r+p'+q_1+n)$ in \eqref{eq:ComplexRelReducedAgain1} vanishes for $p'+q_1=2i$ and the coefficient $(r+p'-q_1+2q_2+n)$ in \eqref{eq:ComplexRelReducedAgain2} vanishes for $p'-q_1=2(i-q_2)$ which we indicate by diagonal lines as in Diagram~\ref{fig:CplxBarriers1}. Further, if $r'=-\rho'-2j\in-\rho'-2\NN$ then the coefficient $(r'+2q_1+n-1)$ in \eqref{eq:ComplexRelReducedAgain1} vanishes for $q_1=j$ which we indicate by a vertical line as in Diagram~\ref{fig:CplxBarriers2}.

\begin{diagram}
\setlength{\unitlength}{4pt}
\begin{picture}(15,15)
\thicklines
\put(1,7){\line(1,-1){6}}
\put(15,7){\line(-1,-1){6}}
\put(1,8){\line(1,0){14}}
{\color{red}
\put(16,0){\line(-1,1){12}}}
\put(0,8){\circle*{1}}
\put(8,0){\circle*{1}}
\put(16,8){\circle*{1}}
\put(4,-3.5){$(p',q_1)$}
\put(-16,10){$(p'-1,q_1+1)$}
\put(10.5,10){$(p'+1,q_1+1)$}
\end{picture}
\hspace{4cm}
\begin{picture}(15,15)
\thicklines
\put(1,1){\line(1,1){6}}
\put(15,1){\line(-1,1){6}}
\put(1,0){\line(1,0){14}}
{\color{red}
\put(4,-4){\line(1,1){12}}}
\put(0,0){\circle*{1}}
\put(8,8){\circle*{1}}
\put(16,0){\circle*{1}}
\put(4.6,10){$(p',q_1)$}
\put(-16,-3.5){$(p'-1,q_1-1)$}
\put(10.5,-3.5){$(p'+1,q_1-1)$}
\end{picture}
\vspace{.3cm}
\caption[]{Barriers for $r=-\rho-2i$}\label{fig:CplxBarriers1}
\end{diagram}
\begin{diagram}
\setlength{\unitlength}{4pt}
\begin{picture}(15,15)
\thicklines
\put(1,7){\line(1,-1){6}}
\put(15,7){\line(-1,-1){6}}
\put(1,8){\line(1,0){14}}
{\color{red}
\put(-10,4){\line(1,0){36}}}
\put(0,8){\circle*{1}}
\put(8,0){\circle*{1}}
\put(16,8){\circle*{1}}
\put(4.6,-3){$(p',q_1)$}
\put(-10,10){$(p'-1,q_1+1)$}
\put(10.5,10){$(p'+1,q_1+1)$}
\end{picture}
\vspace{.3cm}
\caption[]{Barrier for $r'=-\rho'-2j$}\label{fig:CplxBarriers2}
\end{diagram}

\begin{lemma}\label{lem:CplxDiagonalExtensionGeneric}
Let $(r,r')\in\CC^2\setminus L$. Then every diagonal sequence $(t_{q_1+q_2,q_1,q_2})_{q_1,q_2}$ satisfying \eqref{eq:ComplexRelDiagonal1} and \eqref{eq:ComplexRelDiagonal2} has a unique extension to a sequence $(t_{p,q_1,q_2})_{p,q_1,q_2}$ satisfying \eqref{eq:ComplexRelReduced1}--\eqref{eq:ComplexRelReduced4}.
\end{lemma}

\begin{proof}
The proof is similar to the proof of Lemma~\ref{lem:RealDiagonalExtensionGeneric} and we only indicate the relevant steps. 
\textbf{Step 1.} We first treat the case $r\notin-\rho-2\NN$. We fix $q_2$, then the diagonal sequence determines $t_{p',q_1}$ for $p'=q_1$. Since $r\notin-\rho-2\NN$ the coefficient $(r+p'-q_1+2q_2+n)$ in \eqref{eq:ComplexRelReducedAgain2} never vanishes. Hence, \eqref{eq:ComplexRelReducedAgain2} can be used to express $t_{p'+1,q_1-1}$ in terms of $t_{p',q_1}$ and $t_{p'-1,q_1-1}$. As in the proof of Lemma~\ref{lem:RealDiagonalExtensionGeneric}~Step~1, this uniquely determines all numbers $t_{p',q_1}$. Since $q_2$ was arbitrary this determines all numbers $t_{p,q_1,q_2}$.

\textbf{Step 2.} Next assume $r=-\rho-2i\in-\rho-2\NN$ and $r'\notin-\rho'-2\NN$. Then the coefficient $(r+p'-q_1+2q_2+n)$ vanishes if and only if $p'-q_1=2(i-q_2)$. In particular, it does not vanish for $q_2>i$. We can therefore use the technique in Step 1 to extend the diagonal sequence to $t_{p,q_1,q_2}$ for $q_2>i$ and all $p,q_1$. Fixing $q_1$ instead of $	q_2$ we are in the situation that $t_{p',q_2}$ is given on the diagonal $p'=q_2$ and in the region $q_2>i$. Since $r'\notin-\rho'-2\NN$ the coefficient $(r'+2q_2+n-1)$ in \eqref{eq:ComplexRelReducedAgain1} (with $q_1$ and $q_2$ interchanged) never vanishes, so we can use \eqref{eq:ComplexRelReducedAgain1} (with $q_1$ and $q_2$ interchanged) to extend $t_{p',q_2}$ to all $p',q_2$ as in the proof of Lemma~\ref{lem:RealDiagonalExtensionGeneric}~Step~2. Since $q_1$ was arbitrary this determines all numbers $t_{p,q_1,q_2}$.

\textbf{Step 3.} Now let $r=-\rho-2i\in-\rho-2\NN$ and $r'=-\rho'-2j\in-\rho'-2\NN$ with $i,j\in\NN$, $j>i$. Note that to carry out Step 2 we only need that $r'+2q_2+n-1\neq0$ for $q_2\leq i$. This is satisfied since
$$ r'+2q_2+n-1=2(q_2-j)<2(q_2-i)\leq0 $$
by assumption. Hence the technique in Step 2 carries over to this case.
\end{proof}

The case $(r,r')\in L$ has to be handled a little differently from the orthogonal situation.

\begin{lemma}\label{lem:CplxDiagonalExtensionSingular}
Let $(r,r')=(-\rho-2i,-\rho'-2j)\in L$. Then every choice of $t_{0,0,0}$ and $t_{2i+2,0,0}$ determines a unique sequence $(t_{p,q_1,q_2})_{p,q_1,q_2}$ satisfying \eqref{eq:ComplexRelReduced1}--\eqref{eq:ComplexRelReduced4}.
\end{lemma}

\begin{proof}
Fix $q_2=0$, $p'=p-q_2=p$, then by the assumption $t_{p',q_1}$ is known for $(p',q_1)=(0,0)$ and $(2i+2,0)$. This is illustrated in Diagram~\ref{KtypesPropagationL1} where the barriers are as in Diagrams~\ref{fig:CplxBarriers1} and \ref{fig:CplxBarriers2}. Then the techniques from the proof of Lemma~\ref{lem:CplxDiagonalExtensionGeneric} extend $t_{0,0}$ uniquely to the region $p'+q_1\leq2i$, see also Diagram~\ref{KtypesPropagationL1}. To overcome the barrier given by $p'+q_1=2i$ we use \eqref{eq:ComplexRelReducedAgain2} for $p'-q_1=2i$ in which the coefficient $(r+p'-q_1+2q_2+n)$ vanishes. Hence, this relation can be applied to extend along the diagonal line $p'-q_1=2i$ as indicated in Diagram~\ref{KtypesPropagationL1}. It may also be applied anywhere above the diagonal $p'-q_1=2i$ so that we actually extend to the area $p'-q_1\leq2i$, see Diagram~\ref{KtypesPropagationL2}.
\begin{diagram}
\setlength{\unitlength}{4pt}
\begin{picture}(44,38)
\thicklines
\put(0,0){\vector(1,0){40}}
\put(0,0){\vector(0,1){33}}
\put(17,-2){\line(-1,1){19}}
\put(13,-2){\line(1,1){27.5}}
\put(-2,7.5){\line(1,0){42.5}}
\put(0,0){\circle*{1}}
\put(10,0){\circle{.8}}
\put(20,0){\circle*{1}}
\put(30,0){\circle{.8}}
\multiput(5,5)(10,0){4}{\circle{.8}}
\multiput(10,10)(10,0){3}{\circle{.8}}
\multiput(15,15)(10,0){3}{\circle{.8}}
\multiput(20,20)(10,0){2}{\circle{.8}}
\multiput(25,25)(10,0){2}{\circle{.8}}
\put(30,30){\circle{.8}}
\put(39,1){$p'$}
\put(1,31){$q_1$}
\put(8.7,-3.5){$2i$}
\put(17.5,-3.5){$2i+2$}
\put(-3,4.5){$j$}
\put(-6.5,9.5){$j+1$}
\put(-1.5,1){$*$}
\put(18.5,1){$*$}
\end{picture}
\hspace{.5cm}
\begin{picture}(44,38)
\thicklines
\put(0,0){\vector(1,0){40}}
\put(0,0){\vector(0,1){33}}
\put(17,-2){\line(-1,1){19}}
\put(13,-2){\line(1,1){27.5}}
\put(-2,7.5){\line(1,0){42.5}}
\put(5,5){\circle*{1}}
\multiput(0,0)(10,0){3}{\circle*{1}}
\put(30,0){\circle{.8}}
\multiput(15,5)(10,0){3}{\circle{.8}}
\multiput(10,10)(10,0){3}{\circle{.8}}
\multiput(15,15)(10,0){3}{\circle{.8}}
\multiput(20,20)(10,0){2}{\circle{.8}}
\multiput(25,25)(10,0){2}{\circle{.8}}
\put(30,30){\circle{.8}}
{\color{blue}
\put(10,0){\line(1,1){5}}
\put(10,0){\line(1,0){10}}
\put(20,0){\line(-1,1){5}}
\put(16,6){\vector(1,1){3}}}
\put(39,1){$p'$}
\put(1,31){$q_1$}
\put(8.7,-3.5){$2i$}
\put(17.5,-3.5){$2i+2$}
\put(-3,4.5){$j$}
\put(-6.5,9.5){$j+1$}
\put(-1.5,1){$*$}
\put(8.5,1){$*$}
\put(18.5,1){$*$}
\put(3.5,6){$*$}
\end{picture}
\vspace{20pt}
\subcaption{\vbox{\hsize25pc
\begin{legend}
\item[\raise3pt\hbox{\circle*{1}}] $K'$-types with $t_{p',q_1}$ already defined
\item[\raise3pt\hbox{\circle{.8}}] $K'$-types with $t_{p',q_1}$ yet to define
\endgraf
\end{legend}}}
\caption[]{}\label{KtypesPropagationL1}
\end{diagram}
Next we need to overcome the barrier $p'-q_1=2i$ which we do by using \eqref{eq:ComplexRelReducedAgain1} for $q_1=j$. In this relation the coefficient $(r'+2q_1+n-1)$ vanishes and hence we can extend along the line $q_1=j+1$. Using again \eqref{eq:ComplexRelReducedAgain2} even extends to the whole region $q_1>j$, see Diagram~\ref{KtypesPropagationL2}. We note that up to this point we have not yet made use of $t_{2i+2,0}$. This is needed now to extend into the region $\{(p',q_1):p'-q_1>2i,q_1\leq j\}$, see Diagram~\ref{KtypesPropagationL2}. Here both relations \eqref{eq:ComplexRelReducedAgain1} and \eqref{eq:ComplexRelReducedAgain2} are needed.
\begin{diagram}
\setlength{\unitlength}{4pt}
\begin{picture}(44,38)
\thicklines
\put(0,0){\vector(1,0){40}}
\put(0,0){\vector(0,1){33}}
\put(17,-2){\line(-1,1){19}}
\put(13,-2){\line(1,1){27.5}}
\put(-2,7.5){\line(1,0){42.5}}
\multiput(0,0)(10,0){3}{\circle*{1}}
\put(30,0){\circle{.8}}
\multiput(5,5)(10,0){2}{\circle*{1}}
\multiput(25,5)(10,0){2}{\circle{.8}}
\multiput(10,10)(10,0){2}{\circle*{1}}
\put(30,10){\circle{.8}}
\multiput(15,15)(10,0){2}{\circle*{1}}
\put(35,15){\circle{.8}}
\multiput(20,20)(10,0){2}{\circle*{1}}
\multiput(25,25)(10,0){2}{\circle*{1}}
\put(30,30){\circle*{1}}
{\color{blue}
\put(20,10){\line(1,-1){5}}
\put(20,10){\line(1,0){10}}
\put(30,10){\line(-1,-1){5}}
\put(32,10){\vector(1,0){6}}}
\put(39,1){$p'$}
\put(1,31){$q_1$}
\put(8.7,-3.5){$2i$}
\put(17.5,-3.5){$2i+2$}
\put(-3,4.5){$j$}
\put(-6.5,9.5){$j+1$}
\put(-1.5,1){$*$}
\put(8.5,1){$*$}
\put(18.5,1){$*$}
\put(3.5,6){$*$}
\put(13.5,6){$*$}
\put(8.5,11){$*$}
\put(18.5,11){$*$}
\put(13.5,16){$*$}
\put(23.5,16){$*$}
\put(18.5,21){$*$}
\put(28.5,21){$*$}
\put(23.5,26){$*$}
\put(33.5,26){$*$}
\put(28.5,31){$*$}
\end{picture}
\hspace{.5cm}
\begin{picture}(44,38)
\thicklines
\put(0,0){\vector(1,0){40}}
\put(0,0){\vector(0,1){33}}
\put(17,-2){\line(-1,1){19}}
\put(13,-2){\line(1,1){27.5}}
\put(-2,7.5){\line(1,0){42.5}}
\multiput(0,0)(10,0){3}{\circle*{1}}
\put(30,0){\circle{.8}}
\multiput(5,5)(10,0){2}{\circle*{1}}
\multiput(25,5)(10,0){2}{\circle{.8}}
\multiput(10,10)(10,0){3}{\circle*{1}}
\multiput(15,15)(10,0){3}{\circle*{1}}
\multiput(20,20)(10,0){2}{\circle*{1}}
\multiput(25,25)(10,0){2}{\circle*{1}}
\put(30,30){\circle*{1}}
{\color{blue}
\put(15,5){\line(1,-1){5}}
\put(15,5){\line(1,0){10}}
\put(25,5){\line(-1,-1){5}}
\put(27,5){\vector(1,0){6}}}
\put(39,1){$p'$}
\put(1,31){$q_1$}
\put(8.7,-3.5){$2i$}
\put(17.5,-3.5){$2i+2$}
\put(-3,4.5){$j$}
\put(-6.5,9.5){$j+1$}
\put(-1.5,1){$*$}
\put(8.5,1){$*$}
\put(18.5,1){$*$}
\put(3.5,6){$*$}
\put(13.5,6){$*$}
\put(8.5,11){$*$}
\put(18.5,11){$*$}
\put(28.5,11){$*$}
\put(13.5,16){$*$}
\put(23.5,16){$*$}
\put(33.5,16){$*$}
\put(18.5,21){$*$}
\put(28.5,21){$*$}
\put(23.5,26){$*$}
\put(33.5,26){$*$}
\put(28.5,31){$*$}
\end{picture}
\vspace{20pt}
\subcaption{\vbox{\hsize25pc
\begin{legend}
\item[\raise3pt\hbox{\circle*{1}}] $K'$-types with $t_{p',q_1}$ already defined
\item[\raise3pt\hbox{\circle{.8}}] $K'$-types with $t_{p',q_1}$ yet to define
\endgraf
\end{legend}}}
\caption[]{}\label{KtypesPropagationL2}
\end{diagram}
Summarizing, we have extended $t_{0,0,0}$ and $t_{2i+2,0,0}$ uniquely to a sequence $(t_{p,q_1,0})_{p,q_1}$. Next fix $q_1$ and let $p'=p-q_1$. Then $t_{p',q_2}$ is already determined for $(p',q_2)=(p',0)$ with $p'$ arbitrary, see Diagram~\ref{KtypesPropagationL3}. Note that in Relation~\ref{eq:ComplexRelReducedAgain2} (with $q_1$ and $q_2$ interchanged) the coefficient $(r'-2q_2-n+3)$ never vanishes, and hence this relation can be used to extend $(t_{p',0})_{p'}$ uniquely to $(t_{p',q_2})_{p',q_2}$, see Diagram~\ref{KtypesPropagationL3}. Since $q_1$ was arbitrary this finally yields $t_{p,q_1,q_2}$ for any $p,q_1,q_2$ and finishes the proof.
\begin{diagram}
\setlength{\unitlength}{4pt}
\begin{picture}(49,28)
\thicklines
\put(0,0){\vector(1,0){45}}
\put(0,0){\vector(0,1){23}}
\multiput(0,0)(10,0){5}{\circle*{1}}
\multiput(5,5)(10,0){4}{\circle{.8}}
\multiput(10,10)(10,0){4}{\circle{.8}}
\multiput(15,15)(10,0){3}{\circle{.8}}
\multiput(20,20)(10,0){3}{\circle{.8}}
{\color{blue}
\put(15,5){\line(1,-1){5}}
\put(10,0){\line(1,0){10}}
\put(10,0){\line(1,1){5}}
\put(15,7){\vector(0,1){6}}
\put(35,5){\line(1,-1){5}}
\put(30,0){\line(1,0){10}}
\put(30,0){\line(1,1){5}}
\put(35,7){\vector(0,1){6}}}
\put(44,1){$p'$}
\put(1,21){$q_2$}
\put(-1.5,1){$*$}
\put(8.5,1){$*$}
\put(18.5,1){$*$}
\put(28.5,1){$*$}
\put(38.5,1){$*$}
\end{picture}
\vspace{20pt}
\subcaption{\vbox{\hsize25pc
\begin{legend}
\item[\raise3pt\hbox{\circle*{1}}] $K'$-types with $t_{p',q_1}$ already defined
\item[\raise3pt\hbox{\circle{.8}}] $K'$-types with $t_{p',q_1}$ yet to define
\endgraf
\end{legend}}}
\caption[]{}\label{KtypesPropagationL3}
\end{diagram}
\end{proof}

\begin{proof}[{Proof of Theorem~\ref{thm:ComplexMultiplicities}}]
\begin{enumerate}
\item This statement is contained in Lemmas~\ref{lem:CplxDiagonalSequences}, \ref{lem:CplxDiagonalExtensionGeneric} and \ref{lem:CplxDiagonalExtensionSingular}.
\item Composing with embeddings and quotient maps most of the multiplicity statements can be reduced to statements about the (non-)existence of intertwining operators $T:(\pi_r)_\HC\to(\tau_{r'})_\HC$ for particular $r$ and $r'$ such that the numbers $t_{(\alpha_1,\alpha_2),(\alpha_1',\alpha_2')}$ vanish in certain regions. These statements can be checked using the techniques used in Lemmas~\ref{lem:CplxDiagonalSequences}, \ref{lem:CplxDiagonalExtensionGeneric} and \ref{lem:CplxDiagonalExtensionSingular}. This does not work if either $\calV=\calT_\pm(i)$ or $\calW=\calT'_\pm(j)$. We therefore show the multiplicity statements for $m(\calT_+(i),\calT_+(j))$ in detail, using Remark~\ref{rem:IntertwinersBetweenCompositionFactors}. Similar considerations can then be applied to the remaining cases.\par
Let first $\calV=\calT_+(i)$ and $\calW=\calT'_+(j)$. Then, due to Remark~\ref{rem:IntertwinersBetweenCompositionFactors}, an intertwining operator $\calT_+(i)\to\calT_+(j)$ is given by an operator
$$ T:\calF_+(i)\to\bigoplus_{\alpha_1'=j+1}^\infty\bigoplus_{\alpha_2'=0}^j\calE'(\alpha_1',\alpha_2'), \qquad T|_{\calE(\alpha;\alpha')}=t_{\alpha,\alpha'}\cdot R_{\alpha,\alpha'}, $$
such that $T|_{\calF(i)}=0$, and the numbers $t_{\alpha,\alpha'}$ solve the relations \eqref{eq:ComplexRel1}--\eqref{eq:ComplexRel4} whenever the two terms $t_{(\beta_1,\beta_2),(\beta_1',\beta_2')}$ on the right hand sides of \eqref{eq:ComplexRel1}--\eqref{eq:ComplexRel4} satisfy $\beta_1'>j$, $\beta_2'\leq j$ (i.e. the two upper resp. lower vertices of the corresponding triangles are contained in the region $\{(\beta_1',\beta_2'):\beta_1'>j,\beta_2'\leq j\}$).\\
Assume first that $j>i$. Then for any fixed $q_2\leq j$ and $p'=p-q_2$ we are looking for numbers $t_{p',q_1}$ which vanish if either $q_1\leq j$ (i.e. $\alpha_1'\leq j$, the region below the horizontal line in Diagram~\ref{KtypesPropagationCplxTplus1}) or $p'-q_1>2(i-q_2)$ (i.e. $\alpha_2>i$, the region below the diagonal line going into the upper right corner in Diagram~\ref{KtypesPropagationCplxTplus1}). As indicated in Diagram~\ref{KtypesPropagationCplxTplus1}, relation \eqref{eq:ComplexRelReducedAgain1} can be used along the diagonal to obtain $t_{q_1,q_1}=0$ for $q_1>j$. Then using \eqref{eq:ComplexRelReducedAgain2} yields $t_{p',q_1}=0$ for all $p',q_1$, so that $m(\calT_+(i),\calT_+'(j))=0$.
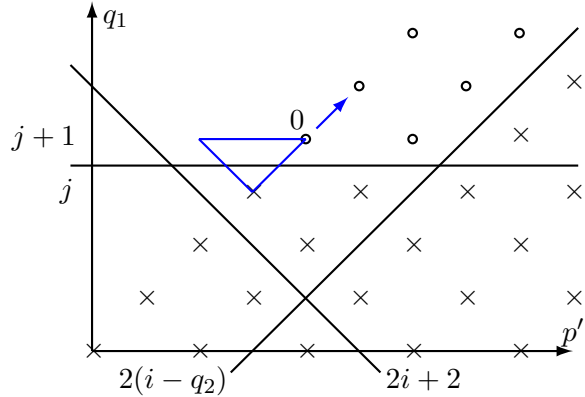
\begin{diagram}
\setlength{\unitlength}{4pt}
\begin{picture}(49,38)
\thicklines
\put(0,0){\vector(1,0){45}}
\put(0,0){\vector(0,1){33}}
\put(27,-2){\line(-1,1){29}}
\put(13,-2){\line(1,1){32.5}}
\put(-2,17.5){\line(1,0){47.5}}
\multiput(-0.9,-0.7)(10,0){5}{$\times$}
\multiput(4.1,4.3)(10,0){5}{$\times$}
\multiput(9.1,9.3)(10,0){4}{$\times$}
\multiput(14.1,14.3)(10,0){4}{$\times$}
\multiput(20,20)(10,0){2}{\circle{.8}}
\put(39.1,19.7){$\times$}
\multiput(25,25)(10,0){2}{\circle{.8}}
\put(44.1,24.7){$\times$}
\multiput(30,30)(10,0){2}{\circle{.8}}
{\color{blue}
\put(10,20){\line(1,-1){5}}
\put(10,20){\line(1,0){10}}
\put(20,20){\line(-1,-1){5}}
\put(21,21){\vector(1,1){3}}}
\put(44,1){$p'$}
\put(1,31){$q_1$}
\put(2.5,-3.5){$2(i-q_2)$}
\put(27.5,-3.5){$2i+2$}
\put(-3,14.5){$j$}
\put(-7.5,19.5){$j+1$}
\put(18.5,21){$0$}
\end{picture}
\vspace{20pt}
\subcaption{\vbox{\hsize25pc
\begin{legend}
\item[\raise3pt\hbox{\circle{.8}}] $K'$-types with $t_{p',q_1}$ to be determined
\item[\hspace{1.858cm}\hbox{$\times$}] $K'$-types with $t_{p',q_1}=0$ by formal reasons
\endgraf
\end{legend}}}
\caption[]{}\label{KtypesPropagationCplxTplus1}
\end{diagram}
Next assume $j\leq i$. Then for fixed $q_2\leq j$ and $p'=p-q_2$ we have to find numbers as indicated in Diagram~\ref{KtypesPropagationCplxTplus2}. Here the relations \eqref{eq:ComplexRelReducedAgain1} and \eqref{eq:ComplexRelReducedAgain2} don't force any of the numbers in the region $\{(p',q_1):p'-q_1\leq2(i-q_2),p'+q_1>2i\}$ to vanish and hence the choice of one $t_{p',q_1}$ determines the remaining numbers. We note that in this case $t_{p',q_1}=0$ for $p'+q_1\leq2i$ and $q_1>j$ as desired.
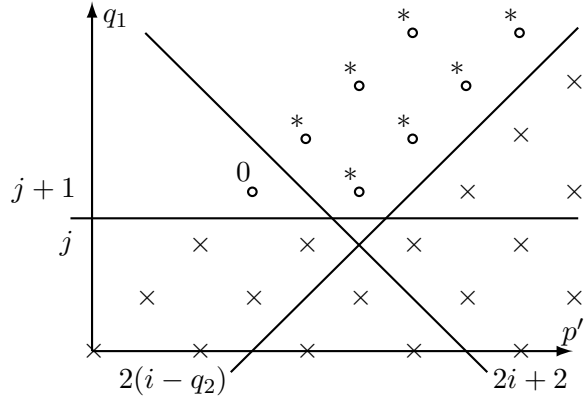
\begin{diagram}
\setlength{\unitlength}{4pt}
\begin{picture}(49,38)
\thicklines
\put(0,0){\vector(1,0){45}}
\put(0,0){\vector(0,1){33}}
\put(37,-2){\line(-1,1){32}}
\put(13,-2){\line(1,1){32.5}}
\put(-2,12.5){\line(1,0){47.5}}
\multiput(-0.9,-0.7)(10,0){5}{$\times$}
\multiput(4.1,4.3)(10,0){5}{$\times$}
\multiput(9.1,9.3)(10,0){4}{$\times$}
\multiput(15,15)(10,0){2}{\circle{.8}}
\multiput(34.1,14.3)(10,0){2}{$\times$}
\multiput(20,20)(10,0){2}{\circle{.8}}
\put(39.1,19.7){$\times$}
\multiput(25,25)(10,0){2}{\circle{.8}}
\put(44.1,24.7){$\times$}
\multiput(30,30)(10,0){2}{\circle{.8}}
\put(44,1){$p'$}
\put(1,31){$q_1$}
\put(2.5,-3.5){$2(i-q_2)$}
\put(37.5,-3.5){$2i+2$}
\put(-3,9.5){$j$}
\put(-7.5,14.5){$j+1$}
\put(13.5,16){$0$}
\put(23.5,16){$*$}
\put(18.5,21){$*$}
\put(28.5,21){$*$}
\put(23.5,26){$*$}
\put(33.5,26){$*$}
\put(28.5,31){$*$}
\put(38.5,31){$*$}
\end{picture}
\vspace{20pt}
\subcaption{\vbox{\hsize25pc
\begin{legend}
\item[\raise3pt\hbox{\circle{.8}}] $K'$-types with $t_{p',q_1}$ to be determined
\item[\hspace{1.858cm}\hbox{$\times$}] $K'$-types with $t_{p',q_1}=0$ by formal reasons
\endgraf
\end{legend}}}
\caption[]{}\label{KtypesPropagationCplxTplus2}
\end{diagram}
Similarly, if we fix $q_1>j$ and let $p'=p-q_1$ we are in the situation of Diagram~\ref{KtypesPropagationCplxTplus3}. More precisely, we need to find numbers $t_{p',q_2}$ satisfying the relations \ref{eq:ComplexRelReducedAgain1} and \ref{eq:ComplexRelReducedAgain2} (with $q_1$ and $q_2$ interchanged) in the region $\{(p',q_2):q_2\leq j,p'+q_2\leq2i\}$ such that $t_{p',q_2}=0$ for $p'-q_2\leq2i$. Again the relations do not force any number in the non-trivial region to vanish (indicated by stars in Diagram~\ref{KtypesPropagationCplxTplus3}). Within this region, the choice of one of the numbers uniquely determines the rest. Together with the previous observation for the case of $q_2\leq j$ fixed we obtain $m(\calT_+(i),\calT_+(j))=1$.\qedhere
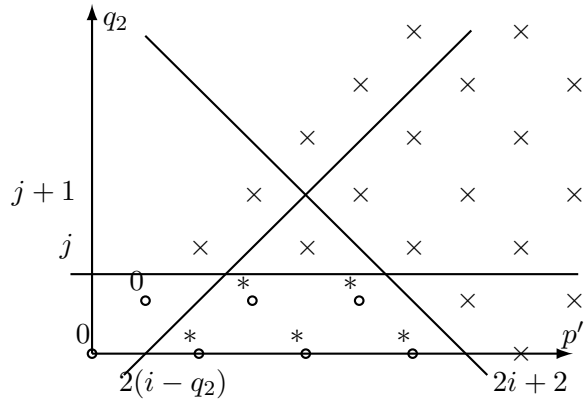
\begin{diagram}
\setlength{\unitlength}{4pt}
\begin{picture}(49,38)
\thicklines
\put(0,0){\vector(1,0){45}}
\put(0,0){\vector(0,1){33}}
\put(37,-2){\line(-1,1){32}}
\put(3,-2){\line(1,1){32.5}}
\put(-2,7.5){\line(1,0){47.5}}
\multiput(0,0)(10,0){4}{\circle{.8}}
\put(39.1,-0.7){$\times$}
\multiput(5,5)(10,0){3}{\circle{.8}}
\multiput(34.1,4.3)(10,0){2}{$\times$}
\multiput(9.1,9.3)(10,0){4}{$\times$}
\multiput(14.1,14.3)(10,0){4}{$\times$}
\multiput(19.1,19.7)(10,0){3}{$\times$}
\multiput(24.1,24.7)(10,0){3}{$\times$}
\multiput(29.1,29.7)(10,0){2}{$\times$}
\put(44,1){$p'$}
\put(1,31){$q_2$}
\put(2.5,-3.5){$2(i-q_2)$}
\put(37.5,-3.5){$2i+2$}
\put(-3,9.5){$j$}
\put(-7.5,14.5){$j+1$}
\put(-1.5,1){$0$}
\put(8.5,1){$*$}
\put(18.5,1){$*$}
\put(28.5,1){$*$}
\put(3.5,6){$0$}
\put(13.5,6){$*$}
\put(23.5,6){$*$}
\end{picture}
\vspace{20pt}
\subcaption{\vbox{\hsize25pc
\begin{legend}
\item[\raise3pt\hbox{\circle{.8}}] $K'$-types with $t_{p',q_1}$ to be determined
\item[\hspace{1.858cm}\hbox{$\times$}] $K'$-types with $t_{p',q_1}=0$ by formal reasons
\endgraf
\end{legend}}}
\caption[]{}\label{KtypesPropagationCplxTplus3}
\end{diagram}
\end{enumerate}
\end{proof}

\subsection{Explicit formula for the spectral function}

As in Section~\ref{sec:RealSpectralFunction} we also find the generic solution to the relations \eqref{eq:ComplexRel1}--\eqref{eq:ComplexRel4} as a meromorphic function in $r,r'\in\CC$.

\begin{proposition}\label{prop:ComplexExplicitEigenvalues}
For $\alpha_1,\alpha_2\in\NN$ and $0\leq\alpha_1'\leq\alpha_1$, $0\leq\alpha_2'\leq\alpha_2$ with $\alpha_1-\alpha_2=\alpha_1'-\alpha_2'$ the numbers
\begin{multline*}
 t_{(\alpha_1,\alpha_2),(\alpha_1',\alpha_2')}(r,r') = \sum_{k=0}^\infty \frac{2^k\Gamma(\frac{\alpha_1+\alpha_2-\alpha_1'-\alpha_2'+2}{2})\Gamma(\frac{\alpha_1+\alpha_2+\alpha_1'+\alpha_2'}{2}+n-1+k)}{k!^2\Gamma(\frac{\alpha_1+\alpha_2-\alpha_1'-\alpha_2'+2}{2}-k)\Gamma(\frac{\alpha_1+\alpha_2+\alpha_1'+\alpha_2'}{2}+n-1)}\\
 \times\frac{\Gamma(\frac{r+n}{2})^2\Gamma(\frac{r'+n-1}{2}+\alpha_1')\Gamma(\frac{r'+n-1}{2}+\alpha_2')\Gamma(\frac{r'-r+1}{2})\Gamma(\frac{r'+r+1}{2}+k)}{\Gamma(\frac{r'+n-1}{2})^2\Gamma(\frac{r+n}{2}+\alpha_1'+k)\Gamma(\frac{r+n}{2}+\alpha_2'+k)\Gamma(\frac{r'+r+1}{2})\Gamma(\frac{r'-r+1}{2}-k)}
\end{multline*}
are rational functions in $r$ and $r'$ satisfying the relations \eqref{eq:ComplexRel1}--\eqref{eq:ComplexRel4}. They are normalized to $t_{(0,0),(0,0)}\equiv1$.
\end{proposition}

\begin{proof}
The proof is similar to the proof of Proposition~\ref{prop:RealExplicitEigenvalues} and we omit some of the details. For simplicity we use the reparametrization $(p,q_1,q_2)$ instead of $(\alpha_1,\alpha_2)$ and $(\alpha_1',\alpha_2')$. Fix $q_2$ and let $p'=p-q_2$, then it is easy to see that for every $k\in\NN$ the expression
$$ \frac{\Gamma(\frac{p'-q_1+2}{2})\Gamma(\frac{p'+q_1}{2}+q_2+n-1+k)\Gamma(\frac{r'+2q_1+n-1}{2})}{\Gamma(\frac{p'-q_1+2}{2}-k)\Gamma(\frac{p'+q_1}{2}+q_2+n-1)\Gamma(\frac{r+2q_1+n}{2}+k)} $$
satisfies \eqref{eq:ComplexRelReducedAgain1}. Further, the series
$$ \sum_{k=0}^\infty b_k\frac{\Gamma(\frac{p'-q_1+2}{2})\Gamma(\frac{p'+q_1}{2}+q_2+n-1+k)\Gamma(\frac{r'+2q_1+n-1}{2})}{\Gamma(\frac{p'-q_1+2}{2}-k)\Gamma(\frac{p'+q_1}{2}+q_2+n-1)\Gamma(\frac{r+2q_1+n}{2}+k)} $$
satisfies \eqref{eq:ComplexRelReducedAgain2} if and only if
$$ b_k = c\frac{2^k\Gamma(\frac{r'+r+1}{2}+k)}{k!^2\Gamma(\frac{r+2q_2+n}{2}+k)\Gamma(\frac{r'-r+1}{2}-k)} $$
for some constant $c=c(r,r',q_2)$ which does not depend on $p'$, $q_1$ and $k$. Plugging in $p'=p-q_2$, using the symmetry of the relations \eqref{eq:ComplexRelReduced1}--\eqref{eq:ComplexRelReduced4} in $q_1$ and $q_2$, and normalizing to $t_{0,0,0}\equiv0$ yields
\begin{multline}
 t_{p,q_1,q_2}(r,r') = \sum_{k=0}^\infty \frac{2^k\Gamma(\frac{p-q_1-q_2+2}{2})\Gamma(\frac{p+q_1+q_2}{2}+n-1+k)}{k!^2\Gamma(\frac{p-q_1-q_2+2}{2}-k)\Gamma(\frac{p+q_1+q_2}{2}+n-1)}\\
 \times\frac{\Gamma(\frac{r+n}{2})^2\Gamma(\frac{r'+n-1}{2}+q_1)\Gamma(\frac{r'+n-1}{2}+q_2)\Gamma(\frac{r'-r+1}{2})\Gamma(\frac{r'+r+1}{2}+k)}{\Gamma(\frac{r'+n-1}{2})^2\Gamma(\frac{r+n}{2}+q_1+k)\Gamma(\frac{r+n}{2}+q_2+k)\Gamma(\frac{r'+r+1}{2})\Gamma(\frac{r'-r+1}{2}-k)}\label{eq:ExplicitFormulaComplexReduced}
\end{multline}
Reparametrizing $p,q_1,q_2$ to $\alpha_1,\alpha_2,\alpha_1',\alpha_2'$ shows the claimed formula. Rewriting \eqref{eq:ExplicitFormulaComplexReduced} as
\begin{equation}
 t_{p,q_1,q_2}(r,r') = \sum_{k=0}^{\frac{p-q_1-q_2}{2}} \frac{2^k(-\frac{p-q_1-q_2}{2})_k(\frac{p+q_1+q_2}{2}+n-1)_k(\frac{r'+n-1}{2})_{q_1}(\frac{r'+n-1}{2})_{q_2}(\frac{r-r'+1}{2})_k(\frac{r'+r+1}{2})_k}{k!^2(\frac{r+n}{2})_{q_1+k}(\frac{r+n}{2})_{q_2+k}}\label{eq:ExplicitFormulaComplexReducedRational}
\end{equation}
further shows that this is a rational function in $r$ and $r'$.
\end{proof}

Also the next two results are proven along the same lines as Corollary~\ref{cor:RealRenormalizations} and Theorem~\ref{thm:RealExplicitBasis}.

\begin{corollary}\label{cor:ComplexRenormalizations}
\begin{enumerate}
\item The renormalized numbers
$$ t_{(\alpha_1,\alpha_2),(\alpha_1',\alpha_2')}^{(1)}(r,r') = \frac{1}{\Gamma(\frac{r+\rho}{2})^2}t_{(\alpha_1,\alpha_2),(\alpha_1',\alpha_2')}(r,r') $$
are holomorphic in $(r,r')\in\CC^2$ for all $(\alpha_1,\alpha_2)$, $(\alpha_1',\alpha_2')$. Further, $t_{(\alpha_1,\alpha_2),(\alpha_1',\alpha_2')}^{(1)}(r,r')=0$ for all $(\alpha_1,\alpha_2)$, $(\alpha_1',\alpha_2')$ if and only if $(r,r')\in L$.
\item Fix $r'=-\rho'-2j$, $j\in\NN$, then the renormalized numbers
$$ t_{(\alpha_1,\alpha_2),(\alpha_1',\alpha_2')}^{(2)}(r,r') = \frac{\Gamma(\frac{(r+\rho)-(r'+\rho')}{2})}{\Gamma(\frac{r+\rho}{2})^2}t_{(\alpha_1,\alpha_2),(\alpha_1',\alpha_2')}(r,r') $$
are holomorphic in $r\in\CC$ for all $(\alpha_1,\alpha_2)$, $(\alpha_1',\alpha_2')$. We have $t_{(\alpha_1,\alpha_2),(\alpha_1',\alpha_2')}^{(2)}(r,r')\equiv0$ whenever $\alpha_1'>j$ or $\alpha_2'>j$. Further, for every $r\in\CC$ there exist $(\alpha_1,\alpha_2)$, $(\alpha_1',\alpha_2')$ with $t_{(\alpha_1,\alpha_2),(\alpha_1',\alpha_2')}^{(2)}(r,r')\neq0$.
\item Fix $N\in\NN$ and let $r'+\rho'=r+\rho+2N$, then the renormalized numbers
$$ t_{(\alpha_1,\alpha_2),(\alpha_1',\alpha_2')}^{(3)}(r,r') = \frac{\Gamma(\frac{r'+\rho'}{2})^2}{\Gamma(\frac{r+\rho}{2})^2}t_{(\alpha_1,\alpha_2),(\alpha_1',\alpha_2')}(r,r') $$
are holomorphic in $r\in\CC$ for all $(\alpha_1,\alpha_2)$, $(\alpha_1',\alpha_2')$. Further, for every $r\in\CC$ there exists $\alpha_0\in\NN$ such that $t_{(\alpha_1,\alpha_2),(\alpha_1,\alpha_2)}^{(3)}(r,r')\neq0$ for $\alpha_1,\alpha_2\geq\alpha_0$.
\end{enumerate}
\end{corollary}

\begin{theorem}\label{thm:ComplexExplicitBasis}
For $i=1,2,3$ we let $T^{(i)}(r,r')$ be the intertwining operators $(\pi_r)_\HC\to(\tau_{r'})_\HC$ corresponding to the numbers $t^{(i)}_{(\alpha_1,\alpha_2),(\alpha_1',\alpha_2')}(r,r')$ in Corollary~\ref{cor:ComplexRenormalizations}. Then the operator $T^{(1)}(r,r')$ is defined for $(r,r')\in\CC^2$, the operator $T^{(2)}(r,r')$ is defined for $r'\in-\rho'-2\NN$ and the operator $T^{(3)}(r,r')$ is defined for $(r+\rho)-(r'+\rho')\in-2\NN$. We have
$$ \Hom_{(\frakg',K')}((\pi_r)_\HC|_{(\frakg',K')},(\tau_{r'})_\HC) = \begin{cases}\CC T^{(1)}(r,r') & \mbox{for $(r,r')\in\CC^2\setminus L$,}\\\CC T^{(2)}(r,r')\oplus\CC T^{(3)}(r,r') & \mbox{for $(r,r')\in L$.}\end{cases} $$
\end{theorem}

\begin{remark}\label{rem:RenormalizationForSubquotientsCplx}
We remark that also every intertwining operator between subquotients $\calV=\calF(i),\calT_\pm(i),\calT(i)$ and $\calW=\calF'(j),\calT'_\pm(j),\calT'(j)$ can be obtained from the holomorphic family $T^{(i)}(r,r')$ by restricting and renormalizing. More precisely, if $\calV$ is a quotient of $(\pi_r)_\HC$ and $\calW$ is a subrepresentation of $(\tau_{r'})_\HC$ then any intertwining operator $T:\calV\to\calW$ gives rise to an intertwining operator $(\pi_r)_\HC\to(\tau_{r'})_\HC$ and is hence of the form $T^{(i)}(r,r')$ for some $i=1,2,3$. This constructs all except the intertwiners $\calT_\pm(i)\to\calT'_\pm(j)$ for $0\leq j\leq i$. These can be obtained from $T^{(1)}(r,r')$ as follows:\par
We first construct an intertwining operator $T^+:\calF_+(i)\to(\tau_{r'})_\HC$ for $r'=-\rho'-2j$ such that $T^+(\calF_+(i))\subseteq\calF_+(j)$. Since $\calF_+(i)$ consists of all $K$-type $\calE(\alpha_1,\alpha_2)$ with $\alpha_2\leq i$ it is given by a sequence $(t^+_{(\alpha_1,\alpha_2),(\alpha_1',\alpha_2')})_{\alpha_2\leq i}$. Reparametrizing to $p,q_1,q_2$ this means that we have to find a sequence $(t^+_{p,q_1,q_2})_{p-q_1+q_2\leq2i}$ satisfying the necessary relations. Let $r'+\rho'=r+\rho+2N$ with $N=i-j\in\NN$ and define
$$ t^+_{p,q_1,q_2}(r,r') := \frac{\Gamma(\frac{r'+\rho'}{2})}{\Gamma(\frac{r+\rho}{2})}t_{p,q_1,q_2}(r,r'), \qquad p-q_1+q_2\leq2i. $$
Then by \eqref{eq:ExplicitFormulaComplexReducedRational} we have
\begin{align*}
 & t^+_{p,q_1,q_2}(r,r')\\
 ={}& \sum_k \frac{2^k(-\frac{p-q_1-q_2}{2})_k(\frac{p+q_1+q_2}{2}+n-1)_k(\frac{r+n}{2}+k+q_1)_{N-k}(\frac{r+n}{2}+N)_{q_2}(-N)_k(r+N+1)_k}{k!^2(\frac{r+n}{2})_{q_2+k}}
\end{align*}
In the sum all terms for $k>\frac{p-q_1-q_2}{2}$ vanish, so that $k\leq\frac{p-q_1-q_2}{2}=\frac{p-q_1+q_2}{2}-q_2\leq i-q_2$. This implies that the denominator does not vanish at $r=-\rho-2i$. Therefore $t^+_{p,q_1,q_2}(r,r')$ is holomorphic in $r=-\rho-2i$ and evaluation there yields
\begin{multline*}
 t^+_{p,q_1,q_2} = t^+_{p,q_1,q_2}(-\rho-2i,-\rho'-2j)\\
 = \sum_{k=0}^{i-j} \frac{2^k(-\frac{p-q_1-q_2}{2})_k(\frac{p+q_1+q_2}{2}+n-1)_k(k+q_1-i)_{i-j-k}(-j)_{q_2}(j-i)_k(1-n-i-j)_k}{k!^2(-i)_{q_2+k}}.
\end{multline*}
The sequence $t^+_{p,q_1,q_2}$ clearly satisfies the necessary relations since it is simply a renormalization of the sequence $t_{p,q_1,q_2}$, and hence it defines an intertwining operator $T^+:\calF_+(i)\to(\tau_{r'})_\HC$. We note that for $q_2>j$ the term $(-j)_{q_2}$ vanishes so that $t^+_{p,q_1,q_2}=0$. Therefore $T^+(\calF_+(i))\subseteq\calF'_+(j)$. Composing with the quotient map $\calF'_+(j)\to\calT'_+(j)$ yields an intertwiner $T^+:\calF_+(i)\to\calT'_+(j)$. We claim that this intertwiner vanishes on $\calF(i)$ and hence factorizes through $\calT_+(i)$. In fact, for $\alpha_2=\frac{p+q_1-q_2}{2}\leq i$ and $q_1>j$ we have $\frac{p-q_1-q_2}{2}\leq i-q_1$ so that we may take the sum over all $k\leq i-q_1$. But then $q_1+k-i\leq0$ and therefore $(q_1+k-i)_{i-j-k}=0$ whence $t^+_{p,q_1,q_2}=0$. This implies that $T^+:\calF_+(i)\to\calT'_+(j)$ factorizes to an intertwiner $T^+:\calT_+(i)\to\calT_+(j)$. To finally see that this intertwiner is non-trivial we note that for all $q_1>i$, $q_2\leq j$ and $p=q_1+q_2$ we have
$$ t^+_{p,q_1,q_2} = \frac{(q_1-i)_{i-j}(-j)_{q_2}}{(-i)_{q_2}}\neq0. $$
\end{remark}

\begin{remark}
The operators $T^{(1)}(r,r')$ are related to the meromorphic family of singular integral operators constructed in \cite{MOO13}. Further, the family $T^{(3)}$ is (up to a constant) equal to the differential restriction operators on the Heisenberg group constructed in \cite{MOZ16}. They can be viewed as a generalization of Juhl's conformally invariant operators (cf. Remark~\ref{rem:JuhlOperators}). It would be interesting to carry out a detailed investigation of all operators $T^{(i)}(r,r')$, $i=1,2,3$, in the non-compact picture as in \cite{KS13}.
\end{remark}

As in the real case, we can prove automatic continuity using the full classification in Theorem~\ref{thm:ComplexExplicitBasis} in terms of the holomorphic family $T^{(1)}(r,r')$. Note that the corresponding holomorphic family of intertwining operators in the smooth category was also constructed in \cite{MOO13}.

\begin{corollary}\label{cor:HCtoInftyCplx}
For $(G,G')=(\upU(1,n),\upU(1,n-1))$ the natural injective map
\begin{equation*}
 \Hom_{G'}(\pi|_{G'},\tau)\to\Hom_{(\frakg',K')}(\pi_\HC|_{(\frakg',K')},\tau_\HC)
\end{equation*}
is an isomorphism for all spherical principal series $\pi$ of $G$ and $\tau$ of $G'$ and their subquotients.
\end{corollary}

\newpage

\appendix

\section{Orthogonal polynomials}\label{app:OrthPoly}

\subsection{Gegenbauer polynomials}\label{app:Gegenbauer}

The classical Gegenbauer polynomials $C_n^\lambda(z)$ can be defined by (see \cite[equation 10.9~(18)]{EMOT81})
$$ C_n^\lambda(z) = \sum_{m=0}^{\lfloor\frac{n}{2}\rfloor} \frac{(-1)^m(\lambda)_{n-m}}{m!(n-2m)!}(2z)^{n-2m}. $$
They obviously satisfy the parity condition (see \cite[equation 10.9~(16)]{EMOT81})
\begin{equation}
 C_n^\lambda(-z) = (-1)^nC_n^\lambda(z).\label{eq:GegenbauerParity}
\end{equation}
The special value at $z=0$ can be written as
\begin{equation}
 C_n^\lambda(0) = \frac{2^n\sqrt{\pi}\Gamma(\lambda+\frac{n}{2})}{n!\Gamma(\frac{1-n}{2})\Gamma(\lambda)} \stackrel{(n=2k)}{=} \frac{(-1)^k\Gamma(\lambda+k)}{k!\Gamma(\lambda)}.\label{eq:GegenbauerValue0}
\end{equation}

\subsection{Jacobi polynomials}\label{app:Jacobi}

The classical Jacobi polynomials $P_n^{(\alpha,\beta)}(z)$ can be defined by (see \cite[equation 10.8~(12)]{EMOT81})
$$ P_n^{(\alpha,\beta)}(z) = 2^{-n}\sum_{m=0}^{n} {n+\alpha\choose m}{n+\beta\choose n-m}(x-1)^{n-m}(x+1)^m. $$
The special value at $z=1$ is given by
\begin{equation}
 P_n^{(\alpha,\beta)}(1) = {n+\alpha\choose n}.\label{eq:JacobiValue1}
\end{equation}

\section{Spherical harmonics}\label{app:SphericalHarmonics}

\subsection{Real spherical harmonics}\label{app:RealSphericalHarmonics}

Let $\calH^\alpha(\RR^n)$ denote the space of harmonic homogeneous polynomials of degree $\alpha$ on $\RR^n$. Endowed with the natural action of $\upO(n)$ the space $\calH^\alpha(\RR^n)$ is an irreducible representation. It is unitary with respect to the norm on $\calH^\alpha(\RR^n)$ given by
$$ \|\phi\|_{L^2(S^{n-1})}^2 = \int_{S^{n-1}} |\phi(x)|^2 dx, $$
where $dx$ denotes the Euclidean measure on $S^{n-1}$. Upon restriction to the subgroup $\upO(n-1)$ the representation $\calH^\alpha(\RR^n)$ decomposes into
\begin{equation}
 \calH^\alpha(\RR^n) \simeq \bigoplus_{0\leq\alpha'\leq\alpha} \calH^{\alpha'}(\RR^{n-1}).\label{eq:BranchingRealSphericalHarmonics}
\end{equation}
Explicit $\upO(n-1)$-equivariant embeddings of the direct summands are given by (see \cite[Fact 7.5.1]{KM11})
\begin{equation}
 I^n_{\alpha'\to\alpha}:\calH^{\alpha'}(\RR^{n-1}) \to \calH^\alpha(\RR^n), \quad I^n_{\alpha'\to\alpha}(\phi)(x',x_n)=\phi(x')C_{\alpha-\alpha'}^{\frac{n-2}{2}+\alpha'}(x_n),\label{eq:ExplicitBranchingRealSphericalHarmonics}
\end{equation}
where $x=(x',x_n)\in S^{n-1}$. The following Plancherel formula holds for $\phi\in\calH^{\alpha'}(\RR^{n-1})$ (see \cite[Fact 7.5.1~(3)]{KM11}, note the different normalization of the Gegenbauer polynomials):
\begin{equation}
 \|I_{\alpha'\to\alpha}^n(\phi)\|_{L^2(S^{n-1})}^2 = \frac{2^{3-n-2\alpha'}\pi\Gamma(n-2+\alpha+\alpha')}{(\alpha-\alpha')!(\alpha+\frac{n-2}{2})\Gamma(\alpha'+\frac{n-2}{2})^2}\|\phi\|_{L^2(S^{n-2})}^2.\label{eq:PlancherelEmbeddingReal}
\end{equation}

For $\phi\in\calH^\alpha(\RR^n)$ we have
\begin{equation}
 x_j\phi = \phi_j^++|x|^2\phi_j^-\label{eq:DecompositionOfMultRealSphericalHarmonics}
\end{equation}
with $\phi_j^\pm\in\calH^{\alpha\pm1}(\RR^n)$ given by
$$ \phi_j^+ = x_j\phi-\frac{|x|^2}{n+2\alpha-2}\frac{\partial\phi}{\partial x_j}, \qquad \phi_j^- = \frac{1}{n+2\alpha-2}\frac{\partial\phi}{\partial x_j}. $$

\subsection{Complex spherical harmonics}\label{app:CplxSphericalHarmonics}

Identifying $\RR^{2n}\simeq\CC^n$ we embed $\upU(n)$ into $\upO(2n)$. Then the restriction of the irreducible representation $\calH^\alpha(\RR^{2n})$ of $\upO(2n)$ to the subgroup $\upU(n)$ decomposes into
$$ \calH^\alpha(\RR^{2n}) = \bigoplus_{\alpha_1+\alpha_2=\alpha} \calH^{\alpha_1,\alpha_2}(\CC^n), $$
where $\calH^{\alpha_1,\alpha_2}(\CC^n)$ denotes the space of harmonic polynomials on $\CC^n$ which are holomorphic of degree $\alpha_1$ and antiholomorphic of degree $\alpha_2$. Endowed with the natural action of $\upU(n)$ the space $\calH^{\alpha_1,\alpha_2}(\CC^n)$ is an irreducible representation. It is unitary with respect to the norm $\|\cdot\|_{L^2(S^{2n-1})}$ where we view $S^{2n-1}$ as the unit sphere in $\CC^n$. Upon restriction to the subgroup $\upU(n-1)$ the representation $\calH^{\alpha_1,\alpha_2}(\CC^n)$ decomposes into
\begin{equation}
 \calH^{\alpha_1,\alpha_2}(\CC^n) = \bigoplus_{\substack{0\leq\alpha_1'\leq\alpha_1\\0\leq\alpha_2'\leq\alpha_2}} \calH^{\alpha_1',\alpha_2'}(\CC^{n-1}).\label{eq:BranchingComplexSphericalHarmonics}
\end{equation}
Explicit $\upU(n-1)$-equivariant embeddings
$$ I^n_{(\alpha_1',\alpha_2')\to(\alpha_1,\alpha_2)}:\calH^{\alpha_1',\alpha_2'}(\CC^{n-1}) \to \calH^{\alpha_1,\alpha_2}(\CC^n) $$
are given by
\begin{multline}
 I^n_{(\alpha_1',\alpha_2')\to(\alpha_1,\alpha_2)}(\phi)(z',z_n) =\\ \phi(z')\begin{cases}z_n^{(\alpha_1-\alpha_2)-(\alpha_1'-\alpha_2')}P_{\alpha_2-\alpha_2'}^{((\alpha_1-\alpha_2)-(\alpha_1'-\alpha_2'),\alpha_1'+\alpha_2'+n-2)}(1-2|z_n|^2)\\\hspace{6.5cm}\mbox{for $\alpha_1-\alpha_2\geq\alpha_1'-\alpha_2'$,}\\\overline{z}_n^{(\alpha_1'-\alpha_2')-(\alpha_1-\alpha_2)}P_{\alpha_1-\alpha_1'}^{((\alpha_1'-\alpha_2')-(\alpha_1-\alpha_2),\alpha_1'+\alpha_2'+n-2)}(1-2|z_n|^2)\\\hspace{6.5cm}\mbox{for $\alpha_1-\alpha_2\leq\alpha_1'-\alpha_2'$,}\end{cases}\label{eq:ExplicitBranchingComplexSphericalHarmonics}
\end{multline}
where $z=(z',z_n)\in S^{2n-1}$. 
For $\phi\in\calH^{\alpha_1,\alpha_2}(\CC^n)$ we have
\begin{equation}
 z_j\phi = \phi_j^{+,\hol}+|z|^2\phi_j^{-,\ahol}, \qquad \overline{z}_j\phi = \phi_j^{+,\ahol}+|z|^2\phi_j^{-,\hol},\label{eq:DecompositionOfMultComplexSphericalHarmonics}
\end{equation}
with $\phi_j^{\pm,\hol}\in\calH^{\alpha_1\pm1,\alpha_2}(\CC^n)$ and $\phi_j^{\pm,\ahol}\in\calH^{\alpha_1,\alpha_2\pm1}(\CC^n)$ given by
\begin{align*}
 \phi_j^{+,\hol} &= z_j\phi-\frac{|z|^2}{\alpha_1+\alpha_2+n-1}\frac{\partial\phi}{\partial\overline{z}_j}, & \phi_j^{-,\hol} &= \frac{1}{\alpha_1+\alpha_2+n-1}\frac{\partial\phi}{\partial z_j},\\
 \phi_j^{+,\ahol} &= \overline{z}_j\phi-\frac{|z|^2}{\alpha_1+\alpha_2+n-1}\frac{\partial\phi}{\partial z_j}, & \phi_j^{-,\ahol} &= \frac{1}{\alpha_1+\alpha_2+n-1}\frac{\partial\phi}{\partial\overline{z}_j}.
\end{align*}


\begin{thebibliography}{10}

\bibitem{BOO96}
Thomas Branson, Gestur {\'O}lafsson, and Bent {\O}rsted, \emph{Spectrum
  generating operators and intertwining operators for representations induced
  from a maximal parabolic subgroup}, J. Funct. Anal. \textbf{135} (1996),
  no.~1, 163--205.

\bibitem{EMOT81}
Arthur Erd{\'e}lyi, Wilhelm Magnus, Fritz Oberhettinger, and Francesco~G.
  Tricomi, \emph{Higher transcendental functions. {V}ol. {II}}, Robert E.
  Krieger Publishing Co. Inc., Melbourne, Fla., 1981, Based on notes left by
  Harry Bateman, Reprint of the 1953 original.

\bibitem{Juh09}
Andreas Juhl, \emph{Families of conformally covariant differential operators,
  {$Q$}-curvature and holography}, Progress in Mathematics, vol. 275,
  Birkh\"auser Verlag, Basel, 2009.

\bibitem{Kob14}
Toshiyuki Kobayashi, \emph{Shintani functions, real spherical manifolds, and
  symmetry breaking operators}, Developments and retrospectives in {L}ie
  theory, Dev. Math., vol.~37, Springer, Cham, 2014, pp.~127--159.

\bibitem{Kob15}
\bysame, \emph{A program for branching problems in the representation theory of
  real reductive groups}, Representations of reductive groups, Progr. Math.,
  vol. 312, Birkh\"auser/Springer, Cham, 2015, pp.~277--322.

\bibitem{KM11}
Toshiyuki Kobayashi and Gen Mano, \emph{The {S}chr\"odinger model for the
  minimal representation of the indefinite orthogonal group {${\rm O}(p,q)$}},
  Mem. Amer. Math. Soc. \textbf{213} (2011), no.~1000.

\bibitem{KS13}
Toshiyuki Kobayashi and Birgit Speh, \emph{Symmetry breaking for
  representations of rank one orthogonal groups}, Mem. Amer. Math. Soc.
  \textbf{238} (2015), no.~1126.

\bibitem{Lok01}
Hung~Yean Loke, \emph{Trilinear forms of {$\mathfrak{gl}_2$}}, Pacific J. Math.
  \textbf{197} (2001), no.~1, 119--144.

\bibitem{MO13}
Jan M\"{o}llers and Bent {\O}rsted, \emph{Estimates for the restriction of
  automorphic forms on hyperbolic manifolds to compact geodesic cycles}, Int.
  Math. Res. Not. IMRN (2017), no.~11, 3209--3236.

\bibitem{MOO13}
Jan M\"{o}llers, Bent {\O}rsted, and Yoshiki Oshima, \emph{Knapp--{S}tein type
  intertwining operators for symmetric pairs}, Adv. Math. \textbf{294} (2016),
  256--306.

\bibitem{MOZ16}
Jan M{\"o}llers, Bent {\O}rsted, and Genkai Zhang, \emph{Invariant
  {D}ifferential {O}perators on {H}-{T}ype {G}roups and {D}iscrete {C}omponents
  in {R}estrictions of {C}omplementary {S}eries of {R}ank {O}ne {S}emisimple
  {G}roups}, J. Geom. Anal. \textbf{26} (2016), no.~1, 118--142.

\bibitem{MOZ15}
\bysame, \emph{On boundary value problems for some conformally invariant
  differential operators}, Comm. Partial Differential Equations \textbf{41}
  (2016), no.~4, 609--643.

\bibitem{MO12}
Jan M{\"o}llers and Yoshiki Oshima, \emph{Restriction of most degenerate
  representations of {$O(1,N)$} with respect to symmetric pairs}, J. Math. Sci.
  Univ. Tokyo \textbf{22} (2015), no.~1, 279--338.

\bibitem{OV04}
Bent {\O}rsted and Jorge Vargas, \emph{Restriction of square integrable
  representations: discrete spectrum}, Duke Math. J. \textbf{123} (2004),
  no.~3, 609--633.

\bibitem{SV11}
Birgit Speh and T.~N. Venkataramana, \emph{Discrete components of some
  complementary series}, Forum Math. \textbf{23} (2011), no.~6, 1159--1187.

\bibitem{SZ12}
Binyong Sun and Chen-Bo Zhu, \emph{Multiplicity one theorems: the {A}rchimedean
  case}, Ann. of Math. (2) \textbf{175} (2012), no.~1, 23--44.

\bibitem{Zha15}
Genkai Zhang, \emph{Discrete components in restriction of unitary
  representations of rank one semisimple {L}ie groups}, J. Funct. Anal.
  \textbf{269} (2015), no.~12, 3689--3713.

\end{thebibliography}

\providecommand{\bysame}{\leavevmode\hbox to3em{\hrulefill}\thinspace}

\end{document}